\documentclass[12pt,final]{amsart}
\usepackage{amssymb,euscript}
\usepackage{amscd}
\usepackage[notref,notcite]{showkeys}
\usepackage{latexsym,todonotes}
\usepackage{epsfig,pxfonts}
\usepackage{amsfonts}

\usepackage{xr,xr-hyper}
\usepackage{bbm,ifpdf}
\ifpdf
  \usepackage{pdfsync,hyperref}
\fi
\usepackage[margin=3cm,footskip=25pt,headheight=21pt]{geometry}
\usepackage{fancyhdr,mathrsfs,nicefrac}
\newtheorem{theorem}{Theorem}[section]
\newtheorem{lemma}[theorem]{Lemma}
\newtheorem{assumption}[theorem]{Assumption}
\newtheorem{proposition}[theorem]{Proposition}
\newtheorem{corollary}[theorem]{Corollary}

\newtheorem{conjecture}[theorem]{Conjecture}
\newtheorem{itheorem}{Theorem}

{
\theoremstyle{plain}
\newtheorem{definition}[theorem]{Definition}
\newtheorem{example}[theorem]{Example}

\newtheorem{remark}[theorem]{Remark}
}

\newcommand{\nc}{\newcommand}
\nc{\cat}{\mathcal{V}}

\newcommand{\arxiv}[1]{\href{http://arxiv.org/abs/#1}{\tt arXiv:\nolinkurl{#1}}}

\newcommand{\supp}{\operatorname{supp}}

\newcommand{\umod}{\operatorname{-\underline{mod}}}

\newcommand{\im}{\operatorname{im}}

\renewcommand{\dim}{\operatorname{dim}}

\newcommand{\ck}{\kappa}

\newcommand{\stR}{R}

\newcommand{\Bi}{\mathbf{i}}
\newcommand{\Bj}{\mathbf{j}}
\newcommand{\Bs}{\mathbf{s}}
\newcommand{\Bt}{\mathbf{t}}
\newcommand{\K}{\mathbf{k}}

\newcommand{\Z}{\mathbb{Z}}
\newcommand{\Q}{\mathbb{Q}}

\newcommand{\om}{\omega}

\newcommand{\LLoc}{\mathbb{L}\!\operatorname{Loc}}
\newcommand{\Rsecs}{\mathbb{R}\Gamma_\bS}
\newcommand{\R}{\mathbb{R}}

\newcommand{\C}{\mathbb{C}}

\newcommand{\wt}{\operatorname{wt}}

\newcommand{\la}{\leftarrow}

\nc{\Bv}{\mathbf{v}}
\nc{\Bu}{\mathbf{u}}
  \nc{\Bw}{\mathbf{w}}
\nc{\coho}{\EuScript{G}}
\nc{\sllhat}{\mathfrak{\widehat{sl}}_\ell}
\nc{\slehat}{\mathfrak{\widehat{sl}}_e}
\nc{\glehat}{\mathfrak{\widehat{gl}}_e}

\newcommand{\secs}{\Gamma_\bS}

\renewcommand{\la}{\lambda}

\newcommand{\al}{\alpha}

\newcommand{\Hom}{\operatorname{Hom}}

\nc{\lift}{\gamma}

\newcommand{\cO}{\mathcal{O}}
\newcommand{\cL}{\mathcal{L}}
\newcommand{\cT}{\mathcal{T}}

\newcommand{\becircled}{\mathaccent "7017}

\newcommand{\Ext}{\operatorname{Ext}}

\newcommand{\cD}{\mathcal{D}}
\newcommand{\bS}{\mathbb{S}}
\newcommand{\bT}{\mathbb{T}}
\newcommand{\bt}{\mathbbm{t}}

\newcommand{\cM}{\mathcal{M}}

\newcommand{\cN}{\mathcal{N}}

\newcommand{\Loc}{\operatorname{Loc}}

\newcommand{\Diff}{\mathfrak{D}}
\newcommand{\excise}[1]{}

\newcommand{\fS}{\mathfrak{S}}

\newcommand{\End}{\operatorname{End}}
\newcommand{\Aut}{\operatorname{Aut}}

\newcommand{\fM}{\mathfrak{M}}
\newcommand{\fN}{\mathfrak{N}}
\newcommand{\fX}{\mathfrak{X}}
\newcommand{\bmu}{\boldsymbol{\mu}}

\newcommand{\fp}{\mathfrak{p}}
\newcommand{\fg}{\mathfrak{g}}
\newcommand{\ft}{\mathfrak{t}}
\newcommand{\fG}{\mathfrak{G}}
\newcommand{\mmod}{\operatorname{-mod}}
\newcommand{\dgmod}{\operatorname{-dg-mod}}
\newcommand{\red}{\mathfrak{r}}
\newcommand{\cOg}{\mathcal{O}_{\!\operatorname{g}}}
\newcommand{\dOg}{D_{\cOg}}
\newcommand{\preO}{p\cOg}
\newcommand{\dpreO}{D_{p\cOg}}
\newcommand{\dpreOz}{D_{p\cOg^0}}

\newcommand{\alg}{T}

\newcommand{\bla}{{\underline{\boldsymbol{\la}}}}
\newcommand{\cOa}{\mathcal{O}_{\!\operatorname{a}}}
\newcommand{\Lotimes}{\overset{L}{\otimes}}
\newcommand{\thetitle}{On generalized category $\cO$ for a quiver variety}

\begin{document}

\renewcommand{\theitheorem}{\Alph{itheorem}}
\usetikzlibrary{decorations.pathreplacing,backgrounds,decorations.markings}
\tikzset{wei/.style={draw=red,double=red!40!white,double distance=1.5pt,thin}}

\noindent {\Large \bf 
\thetitle}
\bigskip\\
{\bf Ben Webster}\footnote{Supported by the NSF under Grant DMS-1151473 and the Alfred P. Sloan Foundation}\\
Department of Mathematics,  University of Virginia, Charlottesville, VA
\bigskip\\
{\small
\begin{quote}
\noindent {\em Abstract.} 
In this paper, we give a method for relating the generalized category
$\cO$ defined by the author and collaborators to explicit finitely
presented algebras, and apply this to quiver varieties.  This allows us
to describe combinatorially not just the structure of these category
$\cO$'s but also how certain interesting families of derived
equivalences, the shuffling and twisting functors, act on them.

In the case of Nakajima quiver varieties, the algebras that appear are weighted KLR algebras and
their steadied quotients, defined by the
author in earlier work.  In particular, these give a geometric
construction of canonical bases for simple representations, tensor
products and Fock spaces.  If the $\C^*$-action used to define the
category $\cO$ is a ``tensor product action'' in the sense of
Nakajima, then we arrive at the unique categorifications of  tensor products; in
particular, we obtain a geometric description of the braid group
actions used by the author in defining categorifications of
Reshetikhin-Turaev invariants.  Similarly, in affine type A, an arbitrary
action results in the diagrammatic algebra equivalent to blocks of category
$\cO$ for cyclotomic Cherednik algebras.  

This approach also allows us to show that these categories are
Koszul and understand their Koszul duals; in particular, we can show
that categorifications of minuscule tensor products in types ADE are Koszul.

In the affine case, this shows that our category $\cO$'s are Koszul
and their Koszul duals are given by category $\cO$'s with rank-level
dual dimension data, and that this duality switches shuffling and
twisting functors.
\end{quote}
}
\bigskip

\section{Introduction}

In this paper, our aim is to introduce a new technique for relating
geometric and algebraic categories.  Since the categories on both
sides of this correspondence are
probably not familiar to many readers, we will provide a teaser
for the results before covering any of the details.  Amongst the
results we will cover are:
\begin{itemize}
\item A new geometric construction of categories (using quiver
  varieties) equivalent to cyclotomic Khovanov-Lauda-Rouquier (KLR) algebras and their weighted
  generalizations.  In particular, this gives geometric constructions
  of the canonical bases of simple representations and their tensor
  products in types ADE, and on higher level Fock spaces in affine
  type A as the simple objects in a category of sheaves.

Alternatively, the reader can think of these as Kazhdan-Lusztig type
character formulae for the decomposition multiplicities of category $\cO$ in these
cases. The r\^{o}le of Kazhdan-Lusztig polynomials is played by the
coefficients of the canonical basis.  
\item A geometric description of the braid group action used in
  \cite{Webmerged} to construction categorifications of
  Reshetikhin-Turaev invariants.  
\item A new method for proving Koszulity of these algebras, providing
  the first proof that categorifications of minuscule tensor products
  in types D and E are Koszul, and giving a new proof that blocks of
  category $\cO$ for Cherednik algebras of $\Z/\ell\Z \wr S_n$ are
  Koszul.
\end{itemize}

In recent years, a great deal of attention in representation theory
has been directed toward generalizations of the BGG category $\cO$;
perhaps the most famous of these is that for Cherednik algebras
\cite{GGOR}, but Braden, Licata, Proudfoot and the author have also considered
these for ``hypertoric enveloping algebras'' 
\cite{BLPWtorico}.  The author and the same group of collaborators have
given a definition subsuming all these examples into a uniform
definition of a category $\cO$ attached to a conical symplectic
singularity, a Hamiltonian $\C^*$-action commuting with the conical structure and a choice of {\bf period}
(which the reader should think of as a central character)
\cite{BLPWquant,BLPWgco}.  Many natural properties of this category
follow from general principles, but at the moment, the techniques for
understanding these categories are not well-developed.  

One consistent theme is that these categories have simple
realizations as representation categories of finite
dimensional algebras with combinatorial presentations.
\begin{itemize}
\item For BGG category $\cO$, this follows from the combinatorial
  descriptions of the categories of Soergel bimodules by Elias,
  Khovanov and Williamson \cite{ElKh,ElW}.
\item For hypertoric category $\cO$, these are provided by certain
  combinatorial algebras defined by Braden, Licata, Proudfoot and the
  author \cite{GDKD}.
\item For a Cherednik algebra of type $G(r,1,\ell)$, these are given
  in \cite{WebRou} by certain variants of Hecke and KLR algebras.
\end{itemize}
One of the most interesting symplectic singularities is a
Nakajima quiver variety for a quiver $\Gamma$, especially for a finite or affine Dynkin
diagram.  We have already proven that this category carries a
categorical action of the Lie algebra attached to $\Gamma$
\cite{Webcatq}, and we know that such a structure
has a lot of power over the category it acts on.

For any Hamiltonian $\C^*$-action commuting with the conical structure
on a quiver variety, we have a diagrammatically defined algebra which
is a candidate for category $\cO$: a reduced
steadied quotient of a weighted Khovanov-Lauda-Rouquier algebra
\cite{WebwKLR}.  Of course, this is not such a familiar object, but it
is purely combinatorial in nature, and certain special cases have had
a bit more of a chance to filter into public consciousness.

To simplify things for the introduction, assume that $\Gamma$ is 
finite or affine type ADE.  In this case, we separate appropriate
$\C^*$-actions into 2 cases:
\begin{itemize}
\item Case 1: the action is a {\bf tensor product action} as used
  by Nakajima \cite{Nak01}; this is induced by a cocharacter $\C^*\to
  G_{\Bw}$ in Nakajima's notation.
\item Case 2: $\Gamma$ is a cycle, and the sum of
  the weights of the edges of the cycle is non-zero.
\end{itemize}
In the first case, the Grothendieck group of category $\cO$ has a
natural map to a tensor product of simple representations; in the
second, it has a natural map to a higher-level Fock space.
\begin{itheorem}\label{main1}
  For a generic integral period, category $\cO$ is equivalent to the
  heart of a $t$-structure on:
  \begin{itemize}
  \item [Case 1:] the category of dg-modules over a tensor product
    algebra $T^\bla$ from \cite[\S 4.2]{Webmerged}.
  \item [Case 2:] the category of dg-modules over a Fock space
    algebra $T^\vartheta$, as studied in \cite[\S 4.1]{WebRou}.
  \end{itemize}
In Case 2, we also have an explicit description of the
  abelian category as a block of category $\cO$ for a cyclotomic
  Cherednik algebra.
\end{itheorem}
In each case, we actually get a stronger graded version of this
theorem.  In Cases 1 and 2 above, category $\cO$ has a graded lift $\tilde{\cO}$,
which we construct geometrically using Hodge theory.  We compare these
with the gradings on the algebras $T^\bla,T^\vartheta$.

\excise{Let $A$ be a graded algebra such that every graded simple module is
self-dual up to an even grading shift (as a vector space). 
Generalizing \cite{MOS},  we call a complex of projectives $\cdots \to P_i\to P_{i+1}\to \cdots$ over $A$
{\bf linear} if in the $i$th homological degree, the head $H$ of any summand
of $P_i$ is isomorphic $H^*(2i)$, the dual of $H$ with grading
shifted downward by $2i$.  If $A$ is basic, this is the same as
requiring that $P_i$ is generated by elements of degree $-i$.  This is
an abelian category if $A$ is graded Morita equivalent to a
non-negatively graded algebra with degree 0 part semi-simple; this
condition is satisfied by 
$T^\bla$ and $T^\vartheta$. }
By \cite[8.7]{WebCB} and \cite[5.23]{WebRou}, the algebras $T^\bla$ and
$T^\vartheta$ are Morita equivalent to non-negatively graded algebras $\becircled{T}^\bla$ and $\becircled{T}^\vartheta$
which have degree 0 part semi-simple.  We call a complex of
graded projectives over such an algebra $A$ {\bf linear} if its $i$th
homological degree is generated in degree $-i$.
The category of linear complexes of projectives
$\mathsf{LCP}(A)$ is a graded abelian category, where
the functor of ``grading shift'' by $a$ is given by shifting the homological
grading by $-a$, and the internal grading by $a$ (to maintain
linearity).  
Theorem
\ref{main1} can then be strengthened to:
\addtocounter{itheorem}{-1}
\renewcommand{\theitheorem}{\Alph{itheorem}'}
\begin{itheorem}
 In Case 1, resp. Case 2, the graded lift $\tilde{\cO}$ is equivalent to the
 graded category of
linear projective complexes $\mathsf{LCP}(\becircled T^\bla)$, resp.
$\mathsf{LCP}(\becircled T^\vartheta)$. 
This equivalence sends simples in $\tilde{\cO}$ to indecomposable
projectives over the algebras $\becircled T^\bla$ and
$\becircled T^\vartheta$ (thought of as
one step complexes).  
This induces isomorphisms of $\Z[q,q^{-1}]$-modules:
  \begin{itemize}
  \item [Case 1:] $K(\tilde{\cO})\cong K(T^\bla\operatorname{-gmod})\cong
V_{\la_1}\otimes \cdots \otimes V_{\la_\ell}$ to the
Grothendieck group of graded projective modules over $T^\bla$, which
is in turn isomorphic to a tensor product of (Lusztig integral forms
of) simple modules over the
quantum group $U_q(\mathfrak{g}_\Gamma)$.
  \item [Case 2:] $K(\tilde{\cO})\cong
    K(T^\vartheta\operatorname{-gmod}) \cong \mathbf{F}_\vartheta$ to the
Grothendieck group of graded projective modules over $T^\vartheta$, which
is in turn isomorphic to a higher level $q$-Fock space.  
  \end{itemize}
From the results of
\cite[Th. A]{WebCB} and \cite[Th. B(5)]{WebRou}, it
follows that in both cases, the simples in $\tilde{\cO}$ match the
canonical bases of the tensor product and $q$-Fock space.
\end{itheorem} 
\renewcommand{\theitheorem}{\Alph{itheorem}}

These results follow from a more general yoga for understanding
category $\cO$ of symplectic reductions of vector spaces by studying
the category of D-modules before reduction.  We can also use the same
principles to recover the hypertoric results of \cite{BLPWtorico}.  At
the moment, we are not aware of any other especially interesting
symplectic reductions of vector spaces on which to apply these
techniques, but such examples may present themselves in the future.
We also expect that the same ideas can apply to other categories of
representations, the most obvious possibility being the Harish-Chandra
bimodules considered in \cite{BLPWquant}.  It is also worth noting
that this approach only covers certain special parameters and changes will be needed to understand the structure of these
categories at general parameters. 

One interesting application of this approach is that it gives a new
geometric proof of the Koszulity of these category $\cO$'s,
conjectured in \cite{BLPWgco}.
\begin{itheorem}
  If the Hamiltonian $\C^*$-action has isolated fixed points on the
  symplectic quotient, and $E$
  and $G$ satisfy certain geometric conditions
  (which we denote \hyperlink{maltese}{$(\maltese)$} or
  \hyperlink{dagger}{$(\dagger)$}), then the category $\cO$ is
  Koszul.  

  In particular, if $\bla$ is a list of minuscule weights, then the
  tensor product algebra
  $T^\bla$ is Koszul.
\end{itheorem}
This also gives a new proof of known Koszulity results for affine
quiver varieties and hypertoric varieties.

Furthermore, there are two classes of auto-functors defined on
categories $\cO$:
\begin{itemize}
  \item {\bf twisting functors}, which arise from tensoring with the
    sections of quantized line bundles, are defined in \cite[\S 6]{BLPWquant}.
  \item {\bf shuffling functors}, which arise from changing the
    $\C^*$-action, are defined in \cite[\S 7]{BLPWgco}.
\end{itemize}
These can be identified with diagrammatically defined functors:
\begin{itheorem}\label{twist-shuffle}
  The
  functors of Chuang
  and Rouquier's braid group action are intertwined with twisting functors in Cases 1 and 2. The shuffling functors are
  intertwined with the diagrammatic functors given by reordering tensor
  factors defined in \cite[\S 6]{Webmerged} (in Case 1) or
  change-of-charge functors as defined in \cite[\S 5.3]{WebRou} (in Case 2).

In Case 2, the Koszul duality between rank-level dual category
$\cO$'s interchanges these two affine braid group actions.
\end{itheorem}

In particular, this theorem gives an indication of how to give a
geometric construction of the knot invariants categorifying
Reshetikhin-Turaev invariants from \cite{Webmerged}, using quiver varieties.  The above
theorem gives a geometric description as shuffling functors of the braid group action used
in that paper; describing the cups and caps will
require techniques we will not delve into here.  

\section{Reduction and quotients}

\subsection{Symplectic quotients}
\label{sec:symplectic-quotients}

Let $E$ be a complex vector space, and let $G$ be a connected reductive
algebraic group with a fixed faithful linear action on $E$.  Let $H=\Aut_{G}(E)$ and
let $Z=H\cap G=Z(G)$.  

In this case, we consider the symplectic
GIT quotient $\fM=T^*E/\!\!/\!\!/\!\!/_\al G=\mu^{-1}(0)/\!\!/_\al G$ for $\al$ a
character of $G$ such that all semi-stable points for $\al$ are
stable.  Since the action of $G$ on the stable points is locally free, the variety $\fM$ is
rationally smooth.  Furthermore, it is equipped with a projective map
$\pi\colon\fM\to 
T^*E/\!\!/\!\!/\!\!/_0 G=\mu^{-1}(0)/\!\!/_0 G$ sending a stable orbit
to the unique closed orbit in its closure.  Let $\fN$ be the
image of this map; if $\fN$ is normal, this is the same as the affinization of $\fM$.
In many cases, the map $\pi$ is
a symplectic rational resolution of
$\fN$.  
Examples of such varieties are Nakajima
quiver varieties and hypertoric varieties (for more general
discussions of these varieties, see \cite{Nak98} and \cite{Pr07},
respectively).  
\excise{Assume that there is a $G$-invariant splitting $E\cong E(0)\oplus
E(-1)$ such that on $E'\cong E(0)\oplus E(-1)^*$, there are no non-constant $G$-invariant
polynomials.  We have a canonical isomorphism $T^*E\cong T^*E'$
sending the conormal Lagrangian to $E(0)$ to $E'$, and the conormal to
$E(-1)$ to the cotangent fiber $T^*_0E'$.  
Let $\bS$ be the
copy  of $\C^*$ acting on $T^*E'$ with weight $-1$ on the cotangent fibers and
trivially on the base, and consider the induced action on $T^*E$.  }

Let $\bS$ be a copy of $\C^*$ acting on $T^*E\cong E\oplus E^*$ via
the action $s\cdot(x,\xi)=(s^{-a}x,s^{-b}\xi)$ for some $a,b\geq 0$ with
$a+b=n>0$. We assume that the $\bS$-fixed points $(T^*E)^\bS$ have no non-constant
$G$-invariant functions; if $a,b>0$, this is automatic, but if $a=0$
or $b=0$,
it may fail.   We'll also use $\bS$ to denote the
resulting conical action on $\fM$.  

Also fix a Hamiltonian $\C^*$-action on $\fM$ which factors through
$H/Z$; to avoid confusion (and match notation from
\cite{BLPWquant,BLPWgco}), we denote this copy of $\C^*$ by $\bT$.

There's a natural action of $H/Z$ and thus $\bT$ on the ring $\C[T^*E]^G$.  Since $G$
and $H$ are reductive, the ring  $\C[T^*E]^G$ is finitely generated
and we can choose each generator to be a $\bT$ weight vector.
We'll also want to make a slightly unusual assumption:
\begin{assumption}\label{positivity}
  The algebra $\C[T^*E]^G$ has a generating set such that any
  generator of negative $\bT$-weight $-k$ has $\bS$-weight $\leq k$.
\end{assumption}
Obviously, this can be achieved by simply replacing $\bT$ with a
positive multiple.  None of our constructions are changed by replacing
the action of $\bT$ with a positive multiple, so we lose no generality
in assuming this; it will simply make some of our later statements a
bit cleaner.  
If
we are given two commuting $\C^*$ actions, we take their product to be
the pointwise product of their values.  We let $\bT'$ be the action of
$\C^*$ on $\fM$ given by the pointwise product of $\bT$ with $\bS$.

\subsection{Quantizing symplectic quotients}
\label{sec:quant-sympl-quot}

Let $\Diff_E$ be the ring of differential operators on $E$; we wish to
think of this as a non-commutative version of the cotangent bundle
$T^*E$.  The action of $G$ extends to one on $\Diff_E$, which is inner; we have a non-commutative moment map $\mu_\chi\colon
U(\fg)\to \Diff_E$ given by
$\mu_\chi(X)=X_E-\chi(X)$ where $X_E$ is the vector
field on $E$ attached to the action of $X\in\fg$, thought of as a
differential operator.  

After fixing $\chi$, there is a non-commutative algebra given by the
quantum Hamiltonian reduction $A_\chi'=\big(\Diff_E/\Diff_E\cdot\mu_\chi(\fg)
\big)^G$.  We let $\xi$ be the vector field
induced by the action of $\bT$; you should think of this as a quantization
of the moment map for the action of $\bT$ on $\fM$.   Since the
actions of $\bT$ and
$G$ commute, the element $\xi$ induces an element of $A_\chi'$, which we
also denote by $\xi$.    

There is also a sheafified version of this construction.  Attached to $\chi$, we have a sheaf of algebras on $\fM$ called
$\cD_\chi$; you should think of $\fM$ as ``morally'' a cotangent bundle
and $\cD_\chi$ as an analogue of the sheaf of twisted microlocal
differential operators.

More explicitly, we can think of the $\bS$-action as
encoding a filtration on the Weyl algebra $\Diff_E$ where linear
functions on $E$ have degree $a$, and constant vector fields degree
$b$.  Thus, $a=0,b=1$ corresponds to the usual order filtration.  
The $h$-adically complete Rees algebra $R(\Diff_E)$ of this filtration is generated by the elements $h^{a/n}x_i,
h^{b/n}\frac{\partial}{\partial x_i}\in \Diff_E[[h^{1/n}]]$.  We have
an obvious isomorphism $  R(\Diff_E)/h^{1/n}R(\Diff_E)\cong
\C[T^*E]$.  Thus, we can define a sheaf $\mathscr{E}_E$ on $T^*E$, by
letting $\mathscr{E}_E(U)$ for $U\subset T^*E$ be the algebra $R(\Diff_E)$ localized at the
multiplicative system of elements of $R(\Diff_E)$ which are
non-vanishing on $U$ modulo $h^{1/n}$.  We let
$\cD_E:=\mathscr{E}_E[h^{-1}]$ and let 
\[\cD_\chi
:=\psi_*(\End_{\cD_E}^{\operatorname{op}}(\cD_E/\cD_E\cdot\mu_\chi(\fg))|_{\mu^{-1}(0)^{ss}})\]
where $\psi$ is the quotient map $\mu^{-1}(0)^{ss}\to \fM$.
 This is a sheaf of free $\C((h))$-modules, and it is
equipped with a free $\C[[h]]$-lattice $\cD_\chi(0)$ induced by the
image of $\mathscr{E}_E$.    We let $A_\chi:=\Gamma(\fM;\cD_\chi)^\bS$.
 If the
algebras of global functions on $\fN$ and $\fM$ coincide, i.e.
$\C[\fN]=\C[\mu^{-1}(0)]^G\cong \C[\fM]$, then by
\cite[3.13]{BLPWquant}, we have that $A_\chi'=A_\chi$.
Otherwise, we just have a map $A_\chi'\to A_\chi$ which will be
surjective if $\fN$ is normal, but may not be
injective.  Let us abuse notation and let $\xi$ denote the image of
the vector field of same name in $A_\chi'$.  

The category of
good $\bS$-equivariant modules over $\cD_\chi$
is related to the category of $A_\chi$-modules by an adjoint pair of
functors $$\secs\colon\cD_\chi\mmod\to A_\chi\mmod\qquad \Loc\colon
A_\chi\mmod\to \cD_\chi\mmod$$ discussed in \cite[\S 4.2]{BLPWquant}.
We wish to compare two special subcategories under these functors.

\begin{definition}[\mbox{\cite[3.10]{BLPWgco}}]
  The category $\cOa$ is the category of finitely generated
  $A_\chi$-modules on which $\xi$ acts locally finitely with finite
  dimensional generalized eigenspaces, such that the real parts of the
  eigenvalues of $\xi$ are bounded above by some real number $b$.
\end{definition}
This is what most representation theorists would probably
think of as the basic object considered in this paper; however, its dependence on $\chi$ is
quite complicated, and beyond our ability to analyze in any generality
in this work.  Instead, we'll replace it with a category which is
easier to analyze, and gives the same answer for many $\chi$.  
Let $o\in \fN$ be the image of the origin in $T^*E$.  
\begin{definition}[\mbox{\cite[3.15]{BLPWgco}}]
  The category $\cOg$ is the category of coherent $\cD_\chi$-modules
  which are supported on the subvariety 
\[\fM^+=\big\{p\in \fM\mid \lim_{\bT\ni t\to 0}t\cdot p \in \pi^{-1}(o)\}\] and which are {$\bT$-regular}, that is, they possess a
$\cD_\chi(0)$-lattice which is preserved by the induced action of the
Lie algebra $\bt$.
\end{definition}
For those readers (probably most) who are not familiar with these sheaves
and the general machinery of deformation quantization in algebraic
geometry (see \cite{BK04a,KSdq}), a $\cD_\chi(0)$-lattice should be thought
of as an equivalent of a good filtration on a D-module.  All these
issues are discussed in great detail in \cite[\S 4]{BLPWquant}.

The category $\cOg$ is a ``local'' version of $\cOa$; our primary
results in this paper will be stated in terms of this category, since
it is more conducive to analysis from a topological perspective.  The
relation between these categories is close, but subtle. In particular,
by \cite[3.19]{BLPWgco}, the functors $\Loc$ and $\secs$ preserve these
subcategories, so we have induced functors
\[ \Loc\colon \cOa\to \cOg \qquad \secs \colon \cOg \to \cOa.\]
These are adjoint equivalences for many values of $\chi$ (though far from
 always). In this case, we say {\bf localization holds} at $\chi$.
 Though its
details are best left to other papers, we note for the benefit of
those who worry that $\cOg$ is too far afield from
representation theory that these functors show that:
\begin{itemize}
\item $\cOg$ for a given parameter $\chi$ is always equivalent to $\cOa$
  for some possibly different parameter $\chi+n\al$  with $n\gg 0$
  \cite[5.11]{BLPWquant} \& \cite[3.19]{BLPWgco}.
\item often even when these functors are not equivalences, the derived
  functors $\LLoc$ and $\Rsecs$ induce an equivalence between the
  subcategories of the derived category with cohomology in 
  $\cOa$ and $\cOg$ \cite[4.13]{BLPWquant}, see also \cite{MN}.
\end{itemize}

One important consequence of these results is that:
\begin{corollary}
  The categories $\cOg$ corresponding to a fixed $\bT$-action attached to different choices of $\bS$-action
  (i.e. different choices of $a,b$) are canonically isomorphic.
\end{corollary}
\begin{proof}
  This equivalence can be constructed by tensoring with a quantization
  of a large power of an ample line bundle, taking $\secs$, and then
  $\Loc$ for the different $\bS$-action, since $\cOa$ only depends on $\bT$.  More canonically,
  \cite[5.8]{BLPWquant} describes the category $\cD_\chi$ as finitely generated modules over
  a 
  $\Z$-algebra ${}_iZ_j$ which is independent of the choice of $\bS$-action,
  modulo bounded modules.  The category $\cOg$ is the full subcategory
  of modules $\{{}_iN\}$ where almost all ${}_iN$ lie in the category
  $\cOa$ for ${}_iZ_i$ by 
\cite[3.19]{BLPWgco}.
\end{proof}

\subsection{Quotients of dg-categories}
\label{sec:localization}

At the base of this work is the interplay between $G$-equivariant
$\Diff_E$-modules on $E$ and $\cD_\chi$-modules.

The relevant category to consider is that of $(G,\chi)$ strongly
equivariant $\Diff_E$-modules.  
\begin{definition}
  We call a $\Diff_E$-module $M$ {\bf strongly $(G,\chi)$-equivariant} if the
  action of $\fg$ via left multiplication by $\mu_\chi$ integrates to a
  $G$-action.
\end{definition}
As usual, we wish to consider these as the heart of a $t$-structure on
the dg-category of complexes of $\Diff_E$-modules on the Artin stack
$E/G$ with strongly $(G,\chi)$-equivariant cohomology bounded-below in degree; that is to say,
we wish to work in the dg-enriched equivariant bounded-below derived category of
$E$. This is perhaps most conveniently understood as the dg-category of
D-modules on the simplicial manifold given by the Borel space of this
action:
\[\cdots \mathrel{\substack{\textstyle\rightarrow\\[-0.4ex]
                      \textstyle\rightarrow \\[-0.4ex]
                      \textstyle\rightarrow \\[-0.4ex]
                      \textstyle\rightarrow}} G\times G\times E  \mathrel{\substack{\textstyle\rightarrow\\[-0.4ex]
                      \textstyle\rightarrow \\[-0.4ex]
                      \textstyle\rightarrow}}  G\times E \rightrightarrows E\]

Recall that Drinfeld (and others) have introduced the quotient of a
dg-category by a dg-subcategory; see \cite{DrDG} for a more detailed
discussion of this construction.  In intuitive terms, the quotient of
a dg-category by a dg-subcategory is the universal dg-category which
receives a functor killing all objects in that subcategory.

Kashiwara and Rouquier \cite[2.8]{KR07} have shown that for a free action
$\cD_\chi$-modules on a quotient can be interpreted as strongly
equivariant modules on the source.  It is more useful for us to use a
slightly stronger version of this result.

There is a microlocalization functor for $\Diff_E$-modules, given
$\boldsymbol{\mu}\cM:=\cD_E\otimes_{\Diff_E}\cM $.
We can use this to construct the reduction functor (called the {\bf Kirwan functor} in
\cite[\S
5.4]{BLPWquant}) \[\red(\cM)=\psi_*\mathcal{H}om_{\cD_E}(\cD_E/\cD_E\cdot\mu_\chi(\fg),
\boldsymbol{\mu} \cM)|_{\mu^{-1}(0)^{ss}}\colon \cD_E\mmod\to \cD_\chi \mmod,\] given by
taking the homomorphism sheaf from $\cD_E/\cD_E\cdot\mu_\chi(\fg) $
to $\cM$ (which is the same as the sheaf of sections $m$ satisfying
$\mu_\chi(\fg)  \cdot m=0$), restricting it to the semi-stable locus, and
pushing forward to the quotient $\fM$.  This is a left module over
$\cD_\chi$, since $\cD_E/\cD_E\cdot\mu_\chi(\fg)
|_{\mu^{-1}(0)^{ss}}$ is naturally a right module over $\psi^*\cD_\chi
$.  

In \cite[5.19]{BLPWquant},
we define a left adjoint $\red_{!}$ to this reduction functor.  A
version of this construction with more general assumptions is also given in \cite[2.8]{BLet}.
The functor is most readily constructed by thinking of $\cD_\chi\mmod$
in terms of $\Z$-algebras;  we have a reduction $\Z$-algebra ${}_i\EuScript{Z}_j=\Diff_E/\big(\Diff_E\cdot\mu_{\chi+j\al}(\fg)
+\mu_{\chi+i\al}(\fg)\cdot \Diff_E\big)$ for $i>j\geq 0$; since $G$ is reductive, this is
the same as the semi-invariants of
$\Diff_E/\Diff_E\cdot\mu_{\chi+j\al}(\fg)$ for the character
$\al^{i-j}$.  This has a natural map to the quantum homogeneous
coordinate ring ${_iZ_{j}}$ defined in \cite[5.6]{BLPWquant}, defined by
$\Gamma(\fM;{_{i}\cT_{j}'})^{\bS}$, the sections of quantizations of
powers of the ample line bundle defined by the GIT description.   This
map is not necessarily an isomorphism, but its kernel and cokernel are
bounded and thus it induces an equivalence
\begin{equation}
\EuScript{Z}\mmod/\EuScript{Z}\mmod_{\operatorname{bd}}\cong
{Z}\mmod/{Z}\mmod_{\operatorname{bd}}\cong \cD\mmod\label{eq:equivalences}
\end{equation}
by
\cite[5.8]{BLPWquant}.  If 
$N_j=N/\mu_{\chi+j\al}(\fg)\cdot N$, then we have a natural
multiplication ${}_i\EuScript{Z}_j\otimes N_j\to N_i$, which makes
$\{N_i\}$ into a ${}_i\EuScript{Z}_j$.  For $j$ sufficiently large,
we have \[N_j\cong \red^\Z(N)_i:=
\Gamma(\fM;{_{j}\cT_{0}'}\otimes_{\cD_\chi}\red(N))^{\bS},\] so modulo
bounded submodules, $\{N_i\}$ and $\red^\Z(N)$ coincide.   Thus the 
equivalences \eqref{eq:equivalences} displayed above intertwine $\red$ with the
functor of taking semi-invariants.  The adjoint $\red_!$ can thus be
defined using the adjoint of semi-invariants. We consider the
homogeneous sections $\Gamma_\Z(\cM)$, as a module over
$\EuScript{Z}$, and we can define
\begin{equation}
\red_!(\cM):=Y\otimes_{\EuScript{Z}}\Gamma_\Z(\cM)\label{eq:bang1}
\end{equation}
where $Y$ is the
$\Diff_E\operatorname{-}\EuScript{Z}$ bimodule $\{Y_k:=\Diff_E/\Diff_E\mu_{\chi+k\al}(\fg)\}$.
By the
description in \cite[5.18]{BLPWquant}, this is the same as the $\Diff_E$-module
\begin{equation}
\red_!(\cM)=Y_{k}
\otimes_{A_{\chi+k\al}}\Gamma(\fM;{_{k}\cT_{0}}\otimes_{\cD_\chi}\cM)
\text{ for $k\gg 0$.}  \label{eq:bang2}
\end{equation}
Both left and right adjoints to $\red$
are constructed in \cite{MNmorse} using a different approach.  

\begin{proposition}\label{dg-quotient}
  The functor $\red$ realizes the dg-category of $\bS$-equivariant $\cD_\chi$-modules on $\fM$ 
  as the quotient of the dg-category of strongly
  $(G,\chi)$-equivariant $\Diff_{E}$-modules by the subcategory of
  complexes whose cohomology has 
  singular support in the unstable locus.
\end{proposition}
\begin{proof}
  Here, we apply \cite[1.4]{DrDG}.
  \begin{itemize}
  \item The functor $\red$ is essentially surjective
  after passing to the homotopy category, since $\red\red_!$ is
  isomorphic to the identity functor by \cite[5.19]{BLPWquant}.
\item All complexes whose
  cohomology is supported on the unstable locus are sent to exact,  and thus contractible,
  complexes by pullback to the stable locus.  
\item The cone of the natural map $\red_!\red (M)\to M$ (which
  represents the functor given by the cone of the map
  $\Ext^\bullet(M,-) \to \Ext^\bullet(\red(M),\red(-))$) has
  cohomology supported on the unstable locus, since the induced map
  $\red \red_!\red( M)\to \red (M)$ is an isomorphism.
  \end{itemize}
Thus, the result follows immediately.
\end{proof}
While psychologically satisfying, this proposition doesn't directly
give us a great deal of information about the category of
$\bS$-equivariant $\cD$-modules since the quotient is a category
in which it is difficult to calculate.  On the other hand, it serves as a hand-hold on the
way to more concrete results.

In particular, the adjoint $\red_!$ provides a slightly more concrete approach to calculation in
$\cD_\chi\mmod$, since $\Ext^\bullet_{\cD_\chi}(M,N)\cong
\Ext^\bullet_{\Diff_E}(\red_!M,\red_!N)$; of course, the principle of
conservation of trouble suggests that computing $\red_!M$ is not
particularly easy, but this at least allows one to work in a better
understood category than $\cD_\chi\mmod$.

\begin{remark}
  We should also mention the work of Zheng and Li, which was an
  important inspiration for us.  In this case, they work with the
  localization of the category of constructible sheaves on $E/G$ by
  those with certain bad singular supports; this category can be
  compared with
  $\cD\mmod$ using the Riemann-Hilbert correspondence.  Using
  D-modules instead of constructible sheaves has the very significant
  advantage that the quotient category is a concrete category of
  modules over a sheaf of algebras, rather than an abstraction.
\end{remark}

\subsection{Riemann-Hilbert and sheaves on the relative precore}

Now, assume that $\bT$ acts on the space $\fM$ via a fixed cocharacter
$\vartheta\colon \bT\to H/Z$.  
Let 
\[\tilde{G}=\big\{(h,g,t,t)\in H\times_Z G\times \bS\times \bT\mid h\equiv
\vartheta(t)\pmod Z  \big\} \] denote the
fiber product of the maps $H\times_Z G\times \bS\to
H/Z\times\bS\leftarrow\bT'$; the group $   \tilde{G}$ acts on $ T^*E$,
with the induced action on $\fM$ factoring through the projection
character $\nu\colon \tilde{G}\to \bT'$.  We call a splitting of this
character $\gamma\colon \bT'\to \tilde{G}$ a {\bf lift} of the
$\bT'$-action on $\fM$.  A rational (real, etc.) lift is a splitting
of the derivative of $\nu$ on the rational Lie algebras $\bt'_{\Q}\to
\tilde{\fg}$.

Consider the action of the  group $   \tilde{G}$ on $\mu^{-1}(0)\times \C$ via
restriction of the action on $T^*E$ and the character $\nu$.  
\begin{definition}
We let the {\bf relative precore} $pC$ of a cocharacter $\vartheta$ be
the set of
elements in $\mu^{-1}(0)$ that have a limit which
lies in $(T^*E)^{\bS}$ under a lift $\bT\to H\times_Z G$.
\end{definition}
It will often be more useful for us to use an alternative
characterization of this space.  The equivalence of these definitions
is why we require Assumption \ref{positivity}:
\begin{proposition}
The set $pC$ is the unstable
points of $\mu^{-1}(0)$ with respect to $\nu$.  That is, it is the vanishing locus on $\mu^{-1}(0)$ of  the functions  which
 are G-invariant and 
  positive weight for some (equivalently all) lifts of $\bT'$.
  Equivalently, 
\[pC:=\big\{x\in \mu^{-1}(0)\mid \overline{\tilde{G}\cdot (x,1)}\cap
T^*E\times \{0\}\neq \emptyset\big\}.\]   
\end{proposition}
\begin{proof}
The set $(T^*E)^{\Bs}$ is the vanishing set of functions with positive
$\bS$-weight.  Thus, the relative precore is the vanishing set of $G$-invariant
functions of non-negative weight under $\bT$ and positive weight under
$\bS$.  Of course, any such function has positive weight for $\bT'$,
and thus vanishes on the unstable locus.  On the other hand, the
functions of positive $\bS$ weight form an ideal of codimension 1 in
$\C[T^*E]^{G}$, since otherwise we would have a non-constant invariant
polynomial on $(T^*E)^{\bS}$.  Thus, $ \C[T^*E]^{G}$ is a finitely
generated algebra and we can choose the generators to all have
positive $\bS$-weight and be $\bT$-weight vectors.  Any monomial in
the generators of Assumption \ref{positivity} which
contains a non-negative $\bT$ weight term lies in the vanishing set of
the relative precore.  If a function has positive weight under $\bT'$ and
doesn't vanish on the relative precore, it must be a sum of monomials in
generators with negative weight; however, we can choose our
presentation so that  these all have non-positive
weight under $\bT'$, by Assumption \ref{positivity} so this is impossible. 
\end{proof}

\begin{example}
  Let $E=\C^2$ with the scalar action of $G=\C^*$, the map $\bT\to H/Z$ trivial and $a=0,b=1$.
  In this case, $\tilde{G}=G\times \bS$, and a lift $\bT'\to \tilde{G}$
  determined by the weight $n$ with which $\bT$ acts on $\C^2$.  If
  $n\geq 0$, then the points in $E$ all have limits; they are all
  fixed if $n=0$, and all have limit 0 if $n>0$.  If $n<0$, then the
  points with limits are those in the
  cotangent fiber $E^*=T_0^*E$ (since the weight of $\bT'$ on this
  fiber is $-n-1$).  Thus, the relative precore in this case is the
  union $E\cup E^*$.  

Instead, we could choose $\bT$ to be the induced Hamiltonian action on
$T^*E$ by the cocharacter $t\mapsto \left[
\begin{smallmatrix}
  t & 0\\
0& 1
\end{smallmatrix}\right ]$, so the different lifts of $\bT'$ possible have
weights 
$n+1,n$ on $E$ and $-n-2,-n-1$ on $E^*$.  For $n\geq 0$, we get $E$ as
before, and for $n\leq -2$, we get $E^*$.   But if $n=-1$, then the
points with limits are the conormal space to the $x$-axis, since points on
the $x$-axis have limits, and those on the $y$-axis do not.   Thus,
the relative precore consists the union of these 3 spaces.  
\end{example}

\begin{proposition}\label{prop:isotropic}
  The relative precore is an isotropic subvariety with finitely
  many components. 
\end{proposition}

\begin{proof}
Since the relative precore is an unstable locus, it carries a decomposition
into Kirwan-Ness strata.  Each Kirwan-Ness stratum corresponds to a
conjugacy class of rational lifts $\bt\to \tilde{\fg}$.  Let $Y$
denote the space of elements with non-negative weight under a particular lift
$\lift$ in this conjugacy class.
This space is isotropic since $\lift$ acts on the symplectic form with
weight $a+b$, so the annihilator of $Y$ is the space on which the
symplectic form acts with weight $>-a-b$.  An open subset
$Y^{KN}\subset Y$ lies in the corresponding Kirwan-Ness stratum.
Attached to this cocharacter, we have a parabolic $P_\lift\subset
\tilde{G}$ preserving $Y^{KN}$, and we have $G\cdot Y^{KN}\cong
G\times_{P_\lift}Y^{KN}$.  Let $X^{KN}=G\cdot Y^{KN}\cap \mu^{-1}(0)$. Now, consider the derivative $D\mu$ at a point in
$x\in X^{KN}\cap Y$. By definition, the derivative of the moment map $D\mu$ at $x$ sends $T_x(T^*E) \to \fg^*$ by the dual of the action map
$D\alpha\colon \fg\to T_x(T^*E)$.  So, if two vectors are tangent to $X^{KN}$, both are
symplectic orthogonal to the $G$-orbit through this point.  Thus, they
lie in $W_x=(D\alpha(\fg)+T_xY)\cap D\alpha(\fg)^\perp$; since
$W_x^\perp=(D\alpha(\fg)+(T_xY)^\perp)\cap D\alpha(\fg)^\perp$,
this space is isotropic.

By our GIT type description of the relative precore, we have that the relative precore
is the vanishing set of the semi-invariants for positive powers of
$\nu$; being a semi-invariant for a (fractional) power of
$\nu$ is the same as being a $G$-invariant function, so this is the
same as $G$-invariant functions with positive $\bT'$-weight.
\end{proof}

We call a $\Diff_E$-module $\cM$ {\bf $\bT$-regular} if the vector field $\xi$
acts locally finitely.

\begin{definition}
  We let $\preO$ be category of $\bT$-regular $(G,\chi)$-equivariant $\Diff_E$-modules with
  singular support on the relative precore.
\end{definition}

This is the natural category whose reductions lie in $\cOg$ for any
different choice of GIT stability condition.
\begin{proposition}\label{red-functor}
The functor $\red$ sends modules in $\preO$ to sheaves in $\cOg$.
The functor $\red_{!}$ sends sheaves in $\cOg$ to modules in $\preO$.
\end{proposition}
\begin{proof}
  If a module $M$ is in $\preO$ then obviously $\red(M)$ is supported
  on $\fM^+$.  Furthermore, $M$ has a good filtration which is
  invariant under the action of $\xi$ (since $M$ must be generated by
  a $\xi$-stable finite dimensional subspace).  This induces a lattice
  on $\red(M)$ which is also $\xi$-invariant. Thus $\red(M)$ is in $\cOg$.

Conversely, if $\cM$ is in $\cOg$, then the classical limit $\cM/h\cM$
is killed by all functions which vanish on the scheme $\fM^+$, which
includes all positive $\bS$-weight $G$-invariant functions.  The same is
true for ${_{k}\cT_{0}}\otimes_{\cD_\chi}\cM$, which lies in $\cOg$ by
\cite[3.17]{BLPWgco}.  Thus $\Gamma(\fM;{_{k}\cT_{0}}\otimes_{\cD_\chi}\cM)$ is a
$\xi$-locally finite module whose associated graded is killed by all
positive weight $G$-invariant functions.   By \eqref{eq:bang2}, the $\Diff_E$-module
$\red_!(\cM)$ can be realized as the tensor product \[K=\red_{!}(\cM)=Y_{k}
\otimes_{A_{\chi+k\al}}\Gamma(\fM;{_{k}\cT_{0}}\otimes_{\cD_\chi}\cM)\] for $k\gg 0$.
Both of the desired properties are preserved by tensor product with the bimodule
$Y_\chi=\Diff_E/\Diff_E\cdot \mu_{\chi}(\fg)$ since it is $\xi$-locally finite for
the adjoint action, and the left and right actions of $G$-invariant
functions on its associated graded coincide.  Thus, $\red_!(\cM)$ is
supported on the relative precore and is $\bT$-regular.
\end{proof}

Let $\dOg$ and $\dpreO$ be the dg-subcategories of the bounded dg-category of
all $\cD_\chi$-modules and equivariant $\Diff_E$-modules, respectively, generated
by $\cOg$ and $\preO$.  The same proof as Proposition \ref{dg-quotient} shows that:
\begin{corollary}\label{preO-quotient}
  The category $\dOg$ is the dg-quotient of $\dpreO$ by the
  dg-subcategory of complexes whose cohomology is supported on the unstable components of
  the relative precore.
\end{corollary}

Let $\EuScript L$ be the sum of one copy of each isomorphism class of
simple objects in $\preO$, and let $\displaystyle \EuScript Q:=
\Ext_{\dpreO}^\bullet(\EuScript L,\EuScript L)^{\operatorname{op}}$;
throughout the paper, we'll only consider examples where
\begin{itemize}
\item[$(*)$] \hypertarget
{formal} {this} dg-algebra is formal (i.e. quasi-isomorphic to its
  cohomology).
\end{itemize}
We have a quasi-equivalence\footnote{Quasi-equivalence is the correct
  notion of equivalence for dg-categories.  A dg-functor $F$ is a
  quasi-equivalence if it induces an equivalence on homotopy
  categories, and induces a quasi-isomorphism on morphism complexes,
  that is, it induces an isomorphism $\Ext^i(M,N)\cong
  \Ext^i(F(M),F(N))$ for all $M$ and $N$.} $\dpreO\cong \EuScript Q\operatorname{-dg-mod}$ given by
$\Ext^\bullet_{\dpreO}(\EuScript L,-)$.  Under this quasi-equivalence, the simples of
$\preO$ are sent to the indecomposable summands of $\EuScript Q$ and modules
supported on the unstable locus are sent to those which are
quasi-isomorphic to an iterated cone of summands of $\EuScript Q$
corresponding to unstable simples. 

Let $I$ be the 2-sided ideal in
$Q$ generated by
every map factoring through a summand supported on the unstable locus
and let
$\EuScript{R}\cong \EuScript Q/I.$ 
\begin{proposition}\label{equivalence}
  We have a quasi-equivalence $\dOg\cong \EuScript{R}\operatorname{-dg-mod}$ intertwining $\red$
  with the functor $\EuScript{R}\Lotimes_{\EuScript Q}-$.  In particular,
  $\Ext_{\dOg}(\red(\EuScript L),\red(\EuScript L))\cong \EuScript{R}$ as formal dg-algebras.
\end{proposition}
\begin{proof}
  The functor $\EuScript{R}\Lotimes_{\EuScript Q}- \colon
  \EuScript Q\operatorname{-dg-mod}\to \EuScript{R}\operatorname{-dg-mod}$ is a
  dg-quotient functor.  By
  Corollary \ref{preO-quotient}, and the uniqueness of dg-quotients up
  to quasi-equivalence
  (\cite[6.1]{DrDG}), the result follows.
\end{proof}

On any smooth variety over the complex numbers, there is an
equivalence of categories between appropriate dg-categories of regular holonomic D-modules and
constructible sheaves, typically called the {\bf Riemann-Hilbert
correspondence}.  

Assume that every object in $\preO$ is a regular $\Diff_E$-module;
this will, for example, be the case when the supports of such modules
are a finite collection of $G$-orbits.

\begin{proposition}
If $\chi$ is integral, the category $\dpreO$ is quasi-equivalent to the dg-category of
sheaves of $\C$-vector spaces on $E/G$ with constructible cohomology
  and  singular support on the relative precore.
\end{proposition}

Thus, we can use the geometry of constructible sheaves in order to
understand the modules on the relative precore, which in turn can lead to an
understanding of category $\cOg$. 

\subsection{Koszulity and mixed Hodge modules}
\label{sec:kosz-mixed-hodge}

Given that we have assumed that the algebra $\EuScript Q$ is formal, there is a
natural graded lift of $\EuScript Q\dgmod$: the derived category of graded
modules over $H^*(\EuScript Q)$.  One might naturally ask if this lift  can be
interpreted in a geometric way.  At least in certain cases, this is
indeed the case, using the formalism of mixed Hodge modules.  The
category of mixed Hodge modules on a quotient Artin stack such as
$E/G$ is considered by Achar in \cite{Achmix}.  We will only use basic
properties of this category, particularly those of
\cite[0.1-2]{SaitoMHM}:
\begin{proposition}
  On any quasi-projective variety $X$, there is a graded abelian category $\mathsf{MHM} $ where each object
  is a $\Diff_X$-module $\cM$ with additional structure\footnote{This
    addition structure is a $\mathbb{Q}$-form $K$ of the perverse sheaf
    $\operatorname{DR}(\cM)$, compatible weight filtrations of $\cM$
    and $K$, and good filtration on $\cM$.}, whose morphisms are
  morphisms of  $\Diff$-modules preserving said additional structure.
  If $X$ is smooth, then the structure sheaf $\fS_X$ lies naturally in
  $\mathsf{MHM} $ (with an obvious choice of additional structure).
  The derived category $D^b(\mathsf{MHM})$ has a six-functor formalism
  defining $f_*,f^*,f_!,f^!, \otimes$ and $\mathcal{H}om$ in a way
  compatible with the forgetful functor $\operatorname{rat}$ passing to the underlying
  complex of $\Diff_X$-modules.  
\end{proposition}
The most important consequence of this theorem for is that if we have
two mixed Hodge modules $L_1$ and $L_2$, then 
\begin{multline*}
  \bigoplus_m \Ext_{\mathsf{MHM}}^i (L_1,L_2(m))\cong
  \bigoplus_m\mathbf{R}^i \Gamma(\mathcal{H}om_{\mathsf{MHM}}(L_1,L_2(m)))\\ \cong
  \mathbf{R}^i\Gamma(\mathcal{H}om_{\Diff_X\mmod}(L_1,L_2))\cong
  \Ext^i_{\Diff_X\mmod}(L_1,L_2),
\end{multline*}
but the former group has the additional structure of being graded (we
think of $\Ext_{\mathsf{MHM}}^i (L_1,L_2(m))$ as being homogeneous
elements of degree $m$).  

For a given $\Diff_X$-module $\cM$, we will call a choice of object in
$\mathsf{MHM} $ with $\cM$ as underlying $\Diff_X$-module a {\bf Hodge structure}
on $\cM$.  
From now on, we will only consider examples where:
\begin{itemize}
\item [$(\maltese)$] \hypertarget
{maltese} {Every} simple $L_\nu$ in $\preO$ can be endowed
  with a pure Hodge structure such that the induced Hodge structure on
  $\Ext_{\mathsf{MHM}}^\bullet (L_\nu,L_\mu)$ in the category $\mathsf{MHM} $
  of mixed Hodge modules is pure, that is, $\Ext_{\mathsf{MHM}}^m (L_\nu,L_\mu(m))\cong \Ext_{\Diff_E\mmod}^m (L_\nu,L_\mu).$ 
\end{itemize}
This may seem like a rather abstruse assumption, but we'll see below
that it is quite geometrically natural in many situations.
Importantly, the assumption $(\maltese)$ implies the assumption \hyperlink{formal}{$(*)$}
discussed earlier, since if $\Ext_{\mathsf{MHM}}^\bullet
(L_\nu,L_\mu)$ has a pure Hodge structure, the compatibility of the
dg-structure with the Hodge structure shows that it is formal as a
dg-algebra.  

We'll call an object $L$ in $\preO$ {\bf gradeable} if it can be
endowed with a mixed Hodge structure.  Note that this is equivalent to
the $Q$-modules $\Ext_{\Diff_E\mmod}(-,L)$ or
$\Ext_{\Diff_E\mmod }(L,-)$ being formal.  
Given any objects $M$ and $N$, we have a universal extension 
\[0\longrightarrow \Ext^1(M,N)^*\otimes N\longrightarrow U(N,M)
\longrightarrow M \longrightarrow 0.\]
corresponding to the canonical class in
$\Ext^\bullet(M,\Ext^1(M,N)^*\otimes N)\cong \Ext^1(M,N)^*\otimes
\Ext^1(M,N)$.
This operation obviously preserves gradability, and every
indecomposable projective in $\preO$ can be constructed by beginning
with a simple $L$, and applying this operation inductively with other
simples, until the result is projective.  In particular, projective
objects in $\preO$ are always gradeable.

We'll call an object $\cM$ of $\cOg$ {\bf gradeable} if $\red_!\cM$ is
gradeable as
defined above.  Since $\red_!$ preserves projectives, the projectives
of $\cOg$ are always gradeable.  Thus, assuming $(\maltese)$, we can
give the projectives $\red_!P_\al$ for the projectives of $\cOg$
unique mixed
Hodge structures with the head pure of degree 0, and thus a natural
grading on the endomorphism
ring \[\EuScript{R}^!:=\End(\oplus_{\al}P_\al)\cong \End(\oplus_{\al}\red_!P_\al).\]
 Similarly, we can define $\EuScript Q^!$ as the graded
endomorphism ring of the sum of all projectives in $\preO$.

\begin{proposition}
  If the assumption \hyperlink{maltese}{$(\maltese)$} holds, the ring $\EuScript{R}^!$ is the quadratic dual (in the sense of
  \cite{MOS}) of the ring $\EuScript{R}$, and similarly $\EuScript Q^!$ the
  quadratic dual of $\EuScript Q$.  That is, the induced graded lift of the category
  $\cOg$ is naturally equivalent to the category of linear complexes
  of projectives over  $\EuScript{R}^!$.  
\end{proposition}
\begin{proof}
The category $\cOg$ is the heart of the standard $t$ structure on
$\dOg$.  Thus, we need only study the image of the standard
$t$-structure under the functor $\Ext^\bullet(L,-)$.  The subcategory
$D^{\geq 0}$ is sent to the category of complexes of projectives where
the $i$th term is generated in degrees $\geq -i$; similarly $D^{\leq
  0}$ is sent to the category of such complexes generated in degree
$\leq -i$.  Thus, the intersection of these subcategories is precisely
the linear complexes of projectives.  The ring $\EuScript{R}^!$ is the
endomorphism algebra of a projective generator in this subcategory,
and thus is the quadratic dual.
\end{proof}
Thus, from \cite[Thm. 30]{MOS}, we conclude:

\begin{theorem}\label{maltese-koszul}
  If the assumption \hyperlink{maltese}{$(\maltese)$} holds, then the inclusion functor $D^b(\cOg)\to
  \dOg$ is a quasi-equivalence if and only if the category $\cOg$ is Koszul.
  Similarly, the functor $D^b(\preO)\to \dpreO$   is a quasi-equivalence if and only if the category $\preO$ is Koszul.
\end{theorem}
As shown in \cite[5.17]{BLPWgco}, if the fixed points of $\bT$ on
$\fM$ are
isolated, the inclusion $D^b(\cOg)\to
  \dOg$ is always an equivalence.  Thus:
  \begin{corollary}
    If  the fixed points of $\bT$ in $\fM$ are
isolated and \hyperlink{maltese}{$(\maltese)$} holds, the
    category $\cOg$ will be Koszul.
  \end{corollary}
Of course, this is a
  very serious geometric assumption, but it holds in many cases, as we
  will see.

\subsection{Constructing objects in \texorpdfstring{$\preO$}{pO}}
\label{sec:constr-objects-preo}

There's a general method for constructing objects in $\cOg$ in a way which
is conducive to calculation; let $\lift$ be a rational lift of $\bT'$.
 We'll only
consider lifts where the space $(T^*E)_\lift$ of elements with non-negative weight is
Lagrangian.   For such a lift, we let 
$E_\lift\subset E$ be the sum of the non-negative weight spaces of $\lift$ acting
on $E$; the conormal to $E_\lift$ is contained in, and thus is equal to $(T^*E)_\lift$. Let
$\fg_\lift$ be the subalgebra with non-negative weights under
$\lift$ in the adjoint representation (this is a parabolic
subalgebra) and let $G_\lift$ be corresponding parabolic subgroup.

Let $X_\lift$ be the fiber product
$G\times_{G_{\lift}}E_\lift$.  This is a vector bundle over
$G/G_\lift$ via the obvious projection map, and has a proper map
$p_\lift\colon X_\lift\to E$ induced by sending $(g,e)$ to
$g\cdot e$. 

The set
 of rational lifts is infinite, but the space $X_\lift$ is unchanged
 if we conjugate $\lift$ by an element of $\tilde{G}$.  Thus, we lose no
 generality by only considering lifts in a fixed maximal torus
 $\tilde{T}$ of
 $\tilde{G}$.
The space of rational lifts in $\tilde{T}$ will be an affine space
over $\mathbb{Q}$.
  There will be finitely many affine hyperplanes in this space given by the
  vanishing sets of weights in $E$ and $\fg$.  The spaces $E_\lift$ and
  $G_\lift$ are constant on the faces of this arrangement.  Thus, only finitely many different
  $G_\lift$ and $E_\lift$ will occur up to conjugacy.  
\begin{remark}
  Note that attached to this data, we have a {\bf generalized Springer theory} in the sense defined by
  Sauter \cite{SauSpr}, for the quadruple $(G,\{G_{\lift}\},E,\{E_{\lift}\})$
  where $\lift$ ranges over conjugacy classes of generic rational lifts.  

  Every interesting example we know of a generalized Springer theory is of
  this form or is obtained from it by the ``Fourier transform'' 
   operation sending $(G,P_i,V,F_i)$ to $(G,P_i,V^*,F_i^\perp)$: 
  \begin{itemize}
  \item the classical Springer theory of a group is obtained when
    $E=\mathfrak{g}$ and $\bT$ acts trivially. This is usually
    presented in the Fourier dual 
    form, so $F_i=\mathfrak{b}^\perp$ instead of $F_i=\mathfrak{b}$.
\item the quiver Springer theory is obtained by taking $E$ to be the
  space of representations of a quiver in a fixed vector space, and
  $G$ the group acting on such representations by change of basis, and
  $\bT$ acting trivially.    This definition recovers Lusztig's
  sheaves defining the canonical basis \cite{Lus91} if we take a
  quiver without loops, and the
  generalization of defined by Bozec in \cite{Bozecperv} if it
  possesses loops.
  \end{itemize}
 Thus at the moment we believe this is the best scheme for
  viewing these examples.
\end{remark}

As in \cite{BLPWgco}, to avoid confusion we will denote the structure
sheaf of a scheme $X$ by $\fS_X$ (rather that $\cO_X$, as would be
more standard).  We will always be considering the structure sheaves
of varieties as left D-modules on those varieties; in particular,
pushforward will always refer to pushforward of D-modules, not of
coherent sheaves. 
\begin{theorem}\label{th:pushforward}
  The pushforward $L_\lift=(p_\lift)_*\fS_{X_\lift}$ is a sum of shifts
  of simple modules in $\preO$.  In fact, its singular support is a
  union of components on which $\lim_{t\to 0}\lift(t)\cdot x$ exists.

If there is a cocharacter
  $\varpi\colon \C^*\to G$ such that $(T^*E)_{\lift}\subset (T^*E)_{\varpi}$ and
  $\langle \al,\varpi\rangle>0$, then $L_\lift$ is supported on
  the unstable locus for $\al$ in $T^*E$.  
\end{theorem}
When such a cocharacter $\varpi$ exists, we call $\lift$ {\bf unsteady}.
\begin{proof}
  This pushforward is a sum of shifts of simple modules by the
  decomposition theorem, and these are all regular since
  $\fS_{X_\lift}$ is regular.  Thus, we need only show that the
  singular support of $(p_\lift)_*\fS_{X_\lift}$ lies in the
  relative precore.
If $a=0,b=1$, then as in \cite[\S 13]{Lus91}, we can compute the singular support of  $(p_\lift)_*\fS_{X_\lift}$ using the geometry of this situation.
Let \[Y=\{(x,\xi)\in p_\lift^*T^*E\mid\langle
T_xX_\lift,\xi\rangle=0\}.\]
That is, these are the pairs where $x\in X_\lift$, and $\xi$ is a cotangent vector at
$p_\lift(x)$ which is perpendicular to the image of the derivative $(Dp_\lift)_x$.
By \cite[9a \& b]{BernD}, we have $SS((p_\lift)_*\fS_{X_\lift})\subset p_\lift(Y)$; thus, we need only show
that for any point in $y\in Y$, there is a rational lift of $\bT$ which attracts
$p_{\lift}(y)$ to a limit in $E$.  Since the pushforward
$(p_\lift)_*\fS_{X_\lift}$ is equivariant for the conical action on
$E$, its singular support is unchanged by changing $a$ and $b$.  

Consider a point  $x\in X_\lift$.  By definition,  $p_\gamma(x)$ is
attracted to a limit $z$ by $g\lift g^{-1}$ for
some $g\in G$.  Thus, the image of the derivative at $p_\gamma(x)$
includes the subspace $g\cdot E_\lift=E_{g\lift g^{-1}}$.  If $(x,\xi)\in Y$, then in
particular, $\xi$ must kill $g\cdot E_\lift=E_{g\lift g^{-1}}$.
Since the positive weight subspace of the dual space is
the annihilator of the non-negative weight space of the primal, the covector $\xi$ must be a sum of covectors of positive weight for
$g\lift g^{-1}$.  In
particular, $p_{\lift}(y)$ has a limit under this lift, and we are
done.

Note that the same argument shows that any point in $p_\lift(Y)$
has a limit under a conjugate of $\varpi$; the Hilbert-Mumford criterion shows that
this point is unstable in the sense of GIT.
\end{proof}

Unfortunately, not all simple objects
in $\preO$ are necessarily summands of such modules; the classical
Springer theory of simple Lie algebras outside type A provides
counterexamples.  
Furthermore, this technique only works for $\chi$ integral (though
we could use more interesting $G_\lift$-equivariant D-modules on
$E_\vartheta$ to produce objects for other twists).

The subcategory of $\dpreO$ generated by
$(p_\lift)_*\fS_{X_\lift}$ can be understood by considering
the dg-Ext algebra of the sum of these objects.  Fix a set $B$ of
rational lifts $\lift$. 
We let
$L=\bigoplus_{\lift\in B}(p_\lift)_*\fS_{X_\lift}$ be the
sum of pushforwards 
over this collection, and let $X=\bigsqcup_{\lift}X_\lift$ be
the union of the corresponding varieties.

The sheaf $L$ plays the
role of the Springer sheaf, the variety $X\times_E X$ the role of
Steinberg variety, and $H^{B\! M}_G(X\times_{E}X)$ the (equivariant)
Steinberg algebra for the Springer theory $(G,\{G_{\lift}\},E,\{E_{\lift}\})$.

We should note that the pushforwards
$(p_\lift)_*\fS_{X_\lift}$ have a natural lift to mixed Hodge
modules, since the structure sheaf on a quasi-projective variety carries a canonical Hodge
structure, and mixed Hodge modules have a natural pushforward.  

\begin{proposition}\label{formal-pure}
  We have a quasi-isomorphism of dg-algebras
\[Q:=\Ext^\bullet(L,L)\cong H^{B\! M}_G(X\times_{E}X)\]
where the latter is endowed with convolution product and trivial
differential; the induced Hodge structure on $\Ext^\bullet(L,L)$ is pure.
\end{proposition}
\begin{proof}
  In general we have that $\Ext^\bullet(L,L)$ is quasi-isomorphic to
  the dg-algebra of Borel-Moore chains on $X\times_{E}X$ endowed with
  the convolution multiplication by Ginzburg
  and Chriss \cite[8.6.7]{CG97}.  We need only see that the latter is formal.
  The product $X_{\lift}\times_{E}X_{\lift'}$ has a
  $G$-equivariant map $X_{\lift}\times_{E}X_{\lift'}\to
  G/G_{\lift}\times G/G_{\lift'}$ with fibers given by vector
  spaces $g_1\cdot E_{\lift}\cap g_2\cdot E_{\lift'}$.  The variety $ G/G_{\lift}\times G/G_{\lift'}$ has
  finitely many $G$-orbits which are all affine bundles over partial
  flag varieties, so $X\times_{E}X$ is a finite union of
  affine bundles over partial flag varieties.  Since each of these
  pieces has a pure Hodge structure on its Borel-Moore homology, the
  Borel-Moore homology of $X\times_{E}X$ is pure as well and the higher
  $A_\infty$-operations must vanish on any minimal model.  Thus, we
  are done.
\end{proof}
In \cite{SauSpr}, the algebra $Q$ is called the {\bf Steinberg
  algebra}.
For every rational lift $\lift$, we have a diagonal embedding of
$X_\lift\hookrightarrow X\times_EX$.  Thus, we can view the
fundamental class $\Delta_*[X_\lift]$ as an idempotent element of
$Q$. Under the isomorphism to the Ext-algebra, this corresponds to the projection $L\to
L_\lift$.  

Now assume that:
\begin{itemize}
\item [$(\dagger)$] \hypertarget
{dagger} {we} have chosen a set $B$ of rational lifts such that each simple module in $\preO$ is a summand of a shift of $L$,
  and every simple with unstable support is a summand of
  $L_\lift$ for $\lift\in B$ unsteady.
\end{itemize}
Note that, by Proposition
\ref{formal-pure}, we have that
\begin{corollary}\label{dagger-implies-maltese}
The condition $(\dagger)$ implies condition
\hyperlink{maltese}{$(\maltese)$} and thus also \hyperlink{formal}{$(*)$}.
\end{corollary}
\begin{proof}
  The object $L$ has a canonical Hodge structure which is pure of
  weight 0, since it is a sum of
  pushforwards.  Thus, it is a sum of simple mixed Hodge modules, each
  with a simple underlying $\Diff_E$-module.  Thus, every summand of
  $L$ can be endowed with a pure Hodge structure, so every simple in
  $\preO$ is gradeable.  In order to establish that
  $\Ext^\bullet(L_\mu,L_\nu)$ has pure mixed Hodge structure, it's
  enough to check this for $\Ext^\bullet(L,L)$ since it contains
  $\Ext^\bullet(L_\mu,L_\nu)$ as a summand.  This is a conclusion of Proposition
\ref{formal-pure}.
\end{proof}

Let $I\subset Q$ be
the ideal generated by the classes $\Delta_*[X_\lift]$ for $\lift$
unsteady, and let $\stR=Q/I$.  Note that if $(\dagger)$ holds,
then the algebras $Q$ and $\EuScript{Q}$ are Morita equivalent, since
they are Ext-algebras of semi-simple objects in which the same simples
appear.  This further induces a Morita equivalence between
$\stR$ and $\EuScript{R}$.
Combining Corollary \ref{preO-quotient} and Proposition
\ref{formal-pure}, we see that:
\begin{theorem}\label{O-E-equivalent}
  Subject to hypothesis \hyperlink{dagger}{$(\dagger)$}, category $\dOg$ is quasi-equivalent to $\stR\dgmod$, via
  the functor $\cM\mapsto \Ext^\bullet(L,\cM)$.
\end{theorem}
In the remainder of the paper, we will consider particular examples,
with the aim of confirming hypothesis \hyperlink{dagger}{$(\dagger)$}
in these cases, and identifying the ring $\stR$, thus giving an
algebraic description of $\dOg$.

\section{Hypertoric varieties}

If $G$ is a torus, with $G\times \bT$ acting on $\C^n$ diagonally, then we can describe
much of the geometry of the situation using an associated hyperplane
arrangement.   Since the structure of the associated category $\cO$ is
set forth in great detail in \cite{BLPWtorico,GDKD}, we'll just give a
sketch of how to apply our techniques in this case as a warm-up to
approaching quiver varieties.  Let $D$ be the full
group of invertible diagonal matrices; we let
$\mathfrak{g}_\R,\mathfrak{d}_\R$ be the corresponding Lie algebras, and
$\mathfrak{g}_\R,\mathfrak{d}_\R$ be the Lie algebras of the maximal
compact subgroups of these tori.  Throughout this section, we'll take
the $\bS$-action associated to
$a=0,b=1$.  Choosing a GIT parameter $\eta\in \mathfrak{g}^*_\R$ and the derivative
$\xi:\R\to \mathfrak{d}_\R/\mathfrak{g}_\R$ of the
$\bT$-action, we obtain a {\bf polarized hyperplane arrangement} in
the sense of \cite{GDKD}, by intersecting the coordinate hyperplanes
of $\mathfrak{d}^*_\R$ with the affine space of functionals which
restrict to $\eta$ on $\mathfrak{g}_\R$.  For simplicity, we assume
that this arrangement is unimodular, i.e. all subsets of normal
vectors that span over $\R$ span the same lattice over $\Z$.

For each rational lift $\lift\colon \bT\to D$, the weights of $\lift$ form a vector
$\mathbf{a}=(a_1,\dots,a_n)$.  The space $E_\lift$ is the sum of coordinate lines
where $a_i\geq 0$; we will encode this by replacing $\mathbf{a}$ with
a corresponding sign vector $\boldsymbol{\sigma}$ (where for our purposes, $0$ becomes a
plus sign).
For each sign vector, we let
\[\fg_{\boldsymbol{\sigma}}=\{(b_1,\dots, b_n)\in
\fg_\R| b_i\geq 0\text{ if }\sigma_i=+,\,  b_i< 0\text{ if }\sigma_i=-\}.\] 

A coordinate subspace will appear as $E_\lift$ for some rational lift $\lift$ if and only if there's
an element of $\xi+\mathfrak{g}_\R$ which lies in $\fg_{\boldsymbol{\sigma}}$, that is, if it is {\bf feasible}. In this case, we denote
$E_{\boldsymbol{\sigma}}:=E_\lift$.
The subspace $E_{\boldsymbol{\sigma}}$ will be unstable if there's an element $\varpi$ of
$\mathfrak{g}_{\boldsymbol{\sigma}}$
with $\langle\eta,\varpi\rangle>0$; that is, if $\mathfrak{g}_{\boldsymbol{\sigma}}$
has no minimum for $\eta$ and thus is {\bf unbounded}.

The arrangement on $\xi+\mathfrak{g}_\R$ is  the Gale dual of the one
usually used to describe the torus action;  as discussed in
\cite{GDKD}, the fundamental theorem of linear programming shows that
feasibility and boundedness switch under Gale duality.

Thus, in the usual indexing, we find that:
\begin{proposition}
  The relative relative precore is a union of the conormal bundles to certain coordinate
  subspaces.  A subspace lies in the relative relative precore if and only if its chamber
  is bounded, and contained in the unstable locus if and only if its
  chamber is infeasible.
\end{proposition}

Since $G$ is abelian, we have that $G_\lift=G$ for every rational lift, and
the pushforward $L_\lift$ is just the
trivial local system on $E_\lift$.  As before, let $L=\bigoplus
L_\lift$, where the set $B$ consists of one lift corresponding to
each bounded sign vector.

\begin{proposition}
If the character $\chi$ is integral, then hypothesis \hyperlink{dagger}{$(\dagger)$}
holds for this set $B$.  Every simple $\cD_\chi$-module in $\preO$ is of the form
$L_\lift$; if $L_\lift$ has unstable support, then $\lift$
is unsteady.  
\end{proposition}
\begin{proof}
  Certainly, every simple object $\cM$ in $\preO$ is the intermediate extension of
  a local system on a coordinate subspace minus its intersections with
  smaller coordinate subspaces.  By applying partial Fourier transform, we can
  assume that this coordinate subspace is all of $\C^n$.

The monodromy of moving around a smaller coordinate subspace is given
by integrating the vector field $x_i\frac{\partial}{\partial x_i}$
where $x_i$ is the coordinate forgotten.  If
$x_i\frac{\partial}{\partial x_i} f=af$, then the monodromy acts by
multiplication by 
$e^{2\pi i a}$, as we can see from the case where $f=x_i^a$.  We wish to prove that this
monodromy is trivial, which is equivalent to
$x_i\frac{\partial}{\partial x_i}$ acting with integral eigenvalues on
$\cM$.

The set of $i$ for which $x_i\frac{\partial}{\partial x_i}$ acts with
integral eigenvalues forms a subarrangement;  we'll call these
coordinates integral.  In order to prove that all coordinates are
integral, it suffices to prove this for a basis of
$\mathfrak{d}/\mathfrak{g}$ by unimodularity.  

On the other hand, if $x_i$ is not an integral coordinate, then the
vanishing cycles along $x_i=0$ of $\cM$ are non-trivial.  Thus, we can
use partial
Fourier transforms, switching $x_i$ and $\frac{\partial}{\partial x_i}$,
and again obtain the intermediate extension of a local system on the
complement of coordinate subspaces.  Thus, between the corresponding
chamber and any unbounded one, there must be at least one integral
coordinate hyperplane.  This will only be the case if
$x_i\frac{\partial}{\partial x_i}$ for $i$ integral span
$\mathfrak{d}/\mathfrak{g}$.  As noted before, unimodularity proves
this is only possible if all coordinates are integral.

Thus, the
  the local system induced on any component of the pre-core has trivial
  monodromy around any isotropic coordinate subspace; thus, we have
  that every simple is a summand of $L_\lift$ for some
  $\lift$.  Since $L_\lift$ is just the pushforward of the
  functions on a linear subspace, it is already simple, so every
  simple is of this form. Note that if $\chi$ is not integral, we can
  have non-trivial local systems.  

  Thus, every simple in $\preO$ is on the form $L_\lift$.  If this
  simple has unstable support, then every point in the subspace $N^*E_\lift$ must
  have a limit as $t\to 0$ under some cocharacter $\varpi$.  The
  cocharacter $\varpi$ thus unsteadies the rational lift $\lift$.  
\end{proof}
By Theorem \ref{O-E-equivalent}, we can give a description 
description of $\dOg$.  
\begin{corollary}
    The Steinberg algebra $Q=H_*^{BM,G}(X\times_E X)$ is the ring
  $A^!_{\operatorname{pol}}(-\xi,-)$ defined in \cite[\S
  8.6]{BLPWtorico}. The quotient algebra $\stR=Q/I$ is isomorphic to
 $ A^!(-\xi,-\eta)$.
\end{corollary}
\begin{proof}
  First, consider the constant sheaves of all coordinate subspaces in
  the equivariant derived category for $D$. 
We can identify this with the algebra
  denoted $Q_n$ in \cite[\S 3.4]{BLPWtorico}; for two sign vectors that differ
  by single entry, the element $\al\to \beta$ is given by the
  pushforward or pullback maps on the constant sheaves on coordinate
  spaces.  In this situation, one will be codimension 1 inside the
  other, and the composition of pullback and pushforward in either
  order is the equivariant Euler class of the normal bundle.  This is
  simply the usual polynomial generators of the cohomology of the
  classifying space of the full diagonal matrices, which we identify
  with the elements $\theta_i$ in the notation of \cite{BLPWtorico}.

  Restricting to objects that lie in $\preO$, we only consider the
  vertices of the $n$-cube corresponding to the bounded sign vectors
  for the corresponding hyperplane arrangement; furthermore
  restricting the group acting from $D$ to $G$ has
  the effect of imposing linear relations on the $\theta_i$'s, exactly
  those in the kernel of $\mathfrak{d}^*\to \mathfrak{g}^*$.  Thus, we
  obtain the ring $A^!_{\operatorname{pol}}(-\xi,-)$.  

  The classes $\Delta_*[X_\lift]$ for $\lift$ unsteady are
  exactly the idempotents that generate the kernel of the map
  $A^!_{\operatorname{pol}}(-\xi,-)\to A^!(-\xi,-\eta)$.  The result follows.
\end{proof}
\begin{remark}
  Braden, Licata, Proudfoot and the author have proven this result in
  a different way, by analyzing the category of projectives in $\cOa$
  over the section algebra of $\cD$ twisted so that localization
  holds.  We show that it is an algebra
  $A(\eta,\xi)$, which we had previously shown was Koszul dual to
  $A^!(-\xi,-\eta)$ in \cite{GDKD}.  The result above gives a more direct geometric proof
  of this fact.

  A bit more care in non-unimodular and non-integral cases would also yield the
  general case of \cite[4.7]{BLPWtorico}, but since this is reproving an old
  result, we leave the details to the reader.
\end{remark}

\section{Quiver varieties: general structure}
\label{sec:quiver-varieties} 

\subsection{Background}
\label{sec:background} 

Our primary application is the study of quiver varieties. Quiver
varieties are perhaps the most interesting examples of symplectic
singularities in the wild.  Furthermore, since the work of Ringel and
Lusztig in the late 80's and early 90's, it has been quite clear that
they have a powerful tie to Lie theory, in particular to its
categorification.  Fix an unoriented graph $\Gamma$ without loops.    Let $\fG$
denote the Kac-Moody algebra associated to $\Gamma$, that is, the
Kac-Moody algebra with Cartan matrix $C=2I-A$, with $A$ the adjacency
matrix of $\Gamma$.  As usual, we let $\al_i,\al_i^\vee$ denote the
simple roots and coroots of this algebra.

\begin{definition}
 For each
orientation $\Omega$ of $\Gamma$ (thought of as a subset of the edges
of the oriented double), a {\bf representation of
  $(\Gamma,\Omega)$ with shadows} is \begin{itemize}
\item a pair of finite dimensional $\C$-vector spaces $V$ and $W$,
  graded by the vertices of $\Gamma$, and
\item a map $x_e:V_{t(e)}\to V_{h(e)}$ for each oriented edge
  (as usual, $t$ and $h$ denote the head and tail of an
  oriented edge), and
\item a map $q:V\to W$ that preserves grading.
\end{itemize}
We let $\Bw$ and $\Bv$ denote $\Gamma$-tuples of integers.
\end{definition}  
For now, we fix an orientation $\Omega$, though we will sometimes wish
to consider the collection of all orientations.
With this choice, we
have the {\bf universal $(\Bw,\Bv)$-dimensional representation}
$$E^\Bw_\Bv=\bigoplus_{i\to  j}\Hom(\C^{v_i},\C^{v_j})
\oplus\bigoplus_i \Hom(\C^{v_i},\C^{w_i}).$$ 
In moduli terms, this is the moduli space of actions of the quiver (in
the sense above) on the vector spaces $V=\bigoplus_i \C^{v_i},W=\bigoplus_i \C^{w_i}$, with their
chosen bases considered as additional structure.  Let $\epsilon_{i,j}$
be the number of arrows with $t(e)=i$ and $h(e)=j$.

This can be thought of in terms of usual quiver representations by
adding a new vertex $\infty$ with $w_i$ edges from $i$ to $\infty$,
forming the {\bf Crawley-Boevey quiver}.  See \cite[\S 3.1]{WebwKLR}
for a longer discussion.

If we wish to consider the moduli space of representations where
$V$ has fixed graded dimension (rather than of actions on a fixed
vector space), we should quotient by the group of isomorphisms of
quiver representations; that is, by the product $G_\Bv=\prod_i\operatorname{GL}(\C^{v_i})$ acting by pre- and
post-composition.  The result is the {\bf moduli stack of $\Bv$-dimensional representations shadowed by $\C^\Bw$}, which we can
define as the stack quotient $$X^\Bw_\Bv=E_{\Bv}^\Bw/G_\Bv.$$ As
before, this is
a smooth Artin stack, which  we can understand using the simplicial Borel space construction.

We'll wish to ``double'' this construction and consider
$T^*E^\Bw_\Bv$; we can think of this as a space of representations of
the doubled quiver of $\Gamma$, with maps $\bar{q}\colon W\to V$ and
$\bar{x}_e=x_{\bar{e}}\colon V_{h(e)}\to V_{t(e)}$.

By convention, if $w_i=\al_i^\vee(\la)$ and $\mu=\la-\sum v_i\al_i$,
then $X^\la_\mu=X^\Bw_\Bv, E^\la_\mu=E^\Bw_\Bv$ (if the difference is not in the positive
cone of the root lattice, then these spaces are by definition empty), and $X^\la=\dot \sqcup_{\xi} X^\la_{\xi}$.
 Let
\[\fM^\la_\mu=T^*E^\la_\mu/\!\!/_{\det}\,G_{\mu}=\mu^{-1}(0)^s/G_{\mu}\qquad
\fN^\la_\mu=\operatorname{image}(\pi\colon \fM^\la_\mu \to T^*E^\la_\mu/\!\!/_{0}\,G_{\mu})\]
be the Nakajima quiver
varieties attached to $\la$ and $\mu$.  
If  $\mu$ is dominant, then 
$\fN^\la_\mu=T^*E^\la_\mu/\!\!/_{0}\,G_{\mu}$.  Note that in order to apply our
construction to these varieties when they contain loops or oriented
cycles, we must use an $\bS$-action with $a,b>0$, since one with $a=0$
or $b=0$ will have invariant polynomials on the fixed points of
$\bS$.  If there are no loops or oriented cycles, it suffices to
assume that $a+b>0$.

See \cite{Nak94,Nak98} for a more detailed
discussion of the geometry of these varieties.  We are interested in
categories of modules over quantizations of these varieties, and
specifically, the categories $\cOg$.  
 In
\cite[Th. A]{Webcatq}, we proved that for integral parameters, the categories
$\cD_\chi\mmod$ for different dimension vectors.
carry a categorical $\fG$-action.  This action preserves the category
$\cOg$, since it acts by tensor product with Harish-Chandra bimodules
as discussed in \cite[3.3]{Webcatq}, which preserve category $\cO$ by
\cite[2.17 \& 3.14]{BLPWquant}.

The Grothendieck group $K(\cOg)$ for any conical symplectic variety
carries a 2-sided cell filtration induced by the decomposition of
$\fN$ into symplectic strata; following the notation of
\cite[\S 7]{BLPWgco}, we let $\cOg^S$ be the subcategory in $\cOg$ of
objects supported on $\pi^{-1}(\bar {S})$, and $\cOg^{\partial S}$ the
subcategory of objects supported on $\pi^{-1}(\partial S)$.  We call a
stratum $S$ {\bf special} if $\cOg^S\neq \cOg^{\partial S}$.

In the case of quiver varieties, all the
varieties $\fN^\la_\mu$ are embedded simultaneously into the affine
variety $\fN^\la_\infty$ of semi-simple representations of the
preprojective algebra up to stabilization.  Thus, we can think about these
filtrations simultaneously on all the category $\cO$'s of
$\fM^\la_\mu$ for all $\mu$.  The most important
example of strata are the subvarieties of the form $\fN^\la_{\mu'}$
for various $\mu'$, but other subvarieties can occur if $\Gamma$ is not of finite type.
\begin{lemma}\label{lem:cell-invariant}
  The 2-sided cell filtration on $K(\cOg)$ is invariant under the
  induced action of $\mathfrak{G}$.
\end{lemma}
\begin{proof}
  Since the categorical action is given by convolution with
  Harish-Chandra sheaves by the construction of \cite{Webcatq}, this follows immediately from the same
  argument as \cite[7.10]{BLPWgco}.
\end{proof}
For any
representation $U$ of $\fG$, we let the {\bf isotypic filtration} be
the filtration indexed by the poset of dominant weights where $U_\mu$
is the sum of the isotypic components for $\mu'\geq \mu$ in the usual
partial 
order on dominant weights.  

By definition, every vector of weight $\mu'$ (and thus every highest
weight vector)  in $K(\cOg)$
is a sum of classes of objects supported
on $\fN^\la_{\mu'}$ and thus lies in $K(\cOg^{\fN^\la_{\mu}})$ for
$\mu\leq \mu'$; that is, the isotypic filtration is ``smaller'' than
the 2-sided cell filtration.  On the other hand, as discussed in  \cite[\S
6]{BLPWgco}, we can also compare this filtration with the BBD
filtration on $H^{BM}_*(\fM^\la_\mu)$, which is of necessity ``bigger.''

\subsection{Weighted KLR algebras}
\label{sec:weight-klr-algebr}

As before, let $G=G_{\Bv}$ and $H=\operatorname{Aut}_{G_{\Bv}}(E_\Bv^{\Bw}).$
\begin{proposition}
  The group $H$ is the product $\prod_{i,j\in
    \Gamma_{\Bw}}GL(\C^{\epsilon_{i,j}})$ over ordered pairs of vertices in the
  Crawley-Boevey graph  $\Gamma_\Bw$, where $\epsilon_{i,j}$ is the
  number of arrows directed from $i$ to $j$.  This acts
  by replacing the maps along edges with linear combinations of the
  maps along parallel edges; in particular, the
  contribution of the pair $(i,\infty)$ is $GL(\C^{w_i})$.
\end{proposition}
\begin{proof}
The group $H$ is a product of general linear groups of the
multiplicity spaces of $G$ acting on $E$.  The spaces
$\C^{\epsilon_{i,j}}$ are precisely the multiplicity space of
  $\Hom(\C^{v_i},\C^{v_j})$ in $E_{\Bv}^\Bw$.
\end{proof}

If $\Gamma$ is a tree, then $H/Z$ is just $PG_\Bw=\prod
GL(\C^{w_i})/\C^*$ where the $\C^*$ represents the diagonal embedding
of scalar matrices.  If $\Gamma$ is a cycle, then $H/Z$ is a quotient
of 
$PG_\Bw\times \C^*$ by a central cyclic group with $n$ elements; the
last factor of $\C^*$ acts with weight 1 on one edge of the cycle.

In what follows, we'll pick a rational cocharacter $\vartheta\colon \bT\to H$, which we fix.
By replacing $\vartheta$ by a conjugate, we can assume that
$\vartheta$ acts diagonally on $\C^{\epsilon_{i,j}}$, i.e. that it
acts by scaling the map along the edge $e$ by weight $\vartheta_e$.  

Having fixed $\vartheta$, its different rational lifts
$\bT\to G\times_ZH$ are all obtained by taking the pointwise product
$\xi\cdot \vartheta$ of this cocharacter
with a rational cocharacter $\xi\colon \bT\to G$.  
Assuming that this cocharacter is generic, we can record its conjugacy
class as a {\bf loading} in the sense of
\cite{WebwKLR}; that is, a map $\Bi\colon \Z\to \Gamma\cup \{0\}$ which
sends any integer which appears as {\it minus} the weight of $\xi$ to the vertex on
which it appears, and all others to 0.  Any loading which appears this
way must have $\#\Bi^{-1}(i)=v_i$,
i.e. exactly $v_i$ non-trivial elements are sent to $i\in \Gamma$.
In this case, we write $|\Bi|=\Bv$.  

If we take $\xi$ to be a
rational cocharacter, we take a
map $\Bi\colon \Q\to \Gamma\cup \{0\}$. Actually, we could use a map
$\Bi\colon \R\to \Gamma\cup \{0\}$ as in \cite{WebwKLR}; this does not
correspond to a $\C^*$ action, but it does have a corresponding vector field
on $T^*E^\la_\mu$, and one can make sense of $\lim_{t\to 0}$
using this vector field or an associated Morse function.  This is not
really necessary for our purposes, though.

The varieties $X_\Bi$ for different lifts of a fixed cocharacter
$\vartheta\colon \bT\to H/Z$ have already appeared; they are precisely
the {\bf loaded flag spaces} discussed in \cite[\S 4.1]{WebwKLR}.

\begin{definition}
  We let an $\Bi$-loaded flag on $V$ be a flag of $\Gamma$-homogeneous
  subspaces $F_a\subset V$ for each $a\in \R$ such that
  $F_b\subset F_a$ for $b\leq a$, and $\dim F_a=\sum_{b\leq a}
  \Bi(b)$.  Even though this filtration is indexed by rational numbers,
  only finitely many different spaces appear; the dimension vector can only change at points in the
  support of the loading, by adding the simple root labeling that
  point to the dimension vector.  Let $\operatorname{Fl}_\Bi$ denote the space of $\Bi$-loaded flags.
\end{definition}

\begin{proposition}
  For $\Bi$ a loading with $|\Bi|=\Bv$ with corresponding cocharacter $\xi\cdot\vartheta$,
  we have an isomorphism of the variety $X_{\xi\cdot\vartheta}$ defined above with
  the loaded flag space \[X_{\Bi}\cong\{(f,F_\bullet)\in E_{\nu}\times \operatorname{Fl}_\Bi |
  f_e(F_a)\subset F_{a+\vartheta_e} \}.\] 
\end{proposition}
\begin{proof}
  In fact, the choice of $\xi\colon \bT\to
  G$ induces a grading on $V$ by {\it minus} the eigenvalues.  The
  weight of $\bT$ acting on the matrix coefficient mapping 
  homogeneous elements of weight $g$ to one of weight $h$ along the
  edge $e$ is $\vartheta_e-h+g$.   We will thus have a limit if
  $\vartheta\geq h-g$.  That is, a representation will
  have a limit if and only if the map $f_e$ is a sum of homogeneous
  maps of degree $\leq \vartheta_e$.  This is exactly the condition
  $f_e(F_a)\subset F_{a+\vartheta_e}$ where $F_a$ is the sum of spaces
  of degree $\leq a$.  Finally, the quotient by
    $G_{\xi\cdot\vartheta}$ forgets the grading, only remembering the flag $F_a$.
\end{proof}
Let $B(\nu)$ be a set of loadings such that every component of the
relative precore is associated to one of them; in the language of
\cite{WebwKLR}, there is a relation called {\bf equivalence} on
loadings, and we take one loading from each of equivalence classes.
As discussed before, the sets $E_\lift$ and $G_\lift$ are constant on
the chambers of an affine hyperplane arrangement on the set of lifts.
Though there are infinitely many loadings, we only need finitely many
to get the finite set of components of the relative precore.  In fact, we could
without loss of generality assume that these loadings are over $\Z$.
Let $p\colon X_{\Bi}\to E_\nu$ be the map forgetting the flags, and
let
\begin{equation}
L_\Bi:=(p_{\Bi})_*\fS_{X_{\Bi}}[\dim
X_\Bi].\label{eq:L-def}
\end{equation}
We let \[\mathbb{X}=\bigsqcup_{\Bi,\Bj\in
  B(\nu)}X_{\Bi}\times_{E_\nu}X_{\Bj}\qquad \text{ and } \qquad
L_\nu=\bigoplus_{ \Bi\in B(\nu) } L_\Bi\]

In \cite[\S 3.1]{WebwKLR}, we also defined a diagrammatic algebra, the
{\bf reduced weighted KLR algebra}, which depends on a choice of 1-cocycle in $\R$
on $\Gamma$; of course, for any rational cocharacter
$\vartheta:\bT\to H$, we can
think of it as a 1-cocycle, and consider its reduced wKLR algebra
$\bar{W}^\vartheta$.  
Applying Proposition \ref{formal-pure}, we arrive at the conclusion that:
\begin{theorem}[\mbox{\cite[4.3]{WebwKLR}}]
  $H^{G,BM}_\bullet(\mathbb{X})\cong \Ext^\bullet(L_\nu,L_\nu)\cong \bar W^\vartheta_\nu$.
\end{theorem}
In \cite[\S 2.6]{WebwKLR}, we considered the quotient of $\bar
W^\vartheta_\nu$ which corresponds to killing the sheaves $L_\xi$
attached to unsteady cocharacters.
For a GIT parameter $\varpi$, we denote this {\bf steadied quotient} by $\bar W^\vartheta_\nu(\varpi)$.
By Proposition \ref{equivalence}, we have the immediate corollary:
\begin{corollary}\label{dagger-W}
  If the hypothesis  \hyperlink{dagger}{$(\dagger)$} holds, then the
  category $\dOg$ is quasi-equivalent to $\bar
  W^\vartheta_\nu(\varpi)\dgmod$ and $\dpreO$ to $\bar W^\vartheta_\nu\dgmod$
\end{corollary}

Note that Theorem \ref{main1} will follow from this corollary once we
know that  hypothesis  \hyperlink{dagger}{$(\dagger)$} holds in the relevant
cases.

\subsection{The case of \texorpdfstring{$\bT$}{T} trivial}
\label{sec:property-l}

Let us first consider the special case where $\bT$ acts trivially and
$\la=0$ (i.e. $\Bw=0$).  We fix $\Bv$ and use the notation
$E:=E^0_\Bv$.  If we orient
$\Gamma$ so that there are no oriented cycles, then 
the relative precore is the Lagrangian subvariety $\Lambda$ considered by
Lusztig\footnote{
If $\Gamma$ does have loops, we obtain the
generalization of $\Lambda$ defined by Bozec \cite{Bozecloop}; as
discussed before, this action does not satisfy all our hypotheses.}
\cite{Lus91}; this is the intersection of $\mu^{-1}(0)$ with the nilcone of the group $G=G_\Bv$ acting
on $T^*E$.
In this case,
the simple $\Diff_E$-modules that appear as summands of $L_\nu$ 
are the images under the Riemann-Hilbert correspondence of the
sheaves $\mathscr{P}_{\Bv,\Omega}$ considered by Lusztig \cite{Lus91} in his categorification of
the upper half of the universal enveloping algebra.  The corresponding
algebra is the original KLR algebra of \cite{KLI,Rou2KM} by \cite[3.6]{VV}.

One fact we will need to use in this paper is that:
\begin{proposition}\label{prop-L}
  When $\bT$ acts trivially and $\chi$ is a character of $G$, every
  simple in $\preO$ on $T^*E$ is a summand of
  $L_\lift$ for some lift $\lift$.  
\end{proposition}
The category $\preO$ is the category of strongly $G$-equivariant
D-modules on $E$ whose singular support is contained in the
subvariety Lusztig calls $\Lambda$; thus, this proposition shows that
every such D-module lies in the set Lusztig denotes
$\mathscr{P}_{\Bv,\Omega}$ in \cite[\S 2.2]{Lus91}. This is closely allied with the
hypothesis \hyperlink{dagger}{$(\dagger)$} discussed earlier, but with
no assumptions about stability.  
 
\begin{proof}[Proof of \ref{prop-L} in finite type]
If $\Gamma$ is finite type, then $G$ acts on $E$ with finitely many
orbits.  The automorphism group of any $\Gamma$-representation is an
extension of a product of general linear groups extended by a
unipotent subgroup. Thus every orbit is equivariantly homotopic to a
point modulo a product of general linear groups, which is simply
connected.  Thus, any $G$-equivariant D-module on $E$ is the
intersection cohomology sheaf of an orbit with the trivial local
system.  Each one of these lies in $\mathscr{P}_{\Bv,\Omega}$ by the
pigeonhole principle, since
this set has size equal to the Kostant partition function of $\sum
v_i\al_i$ by \cite[10.17(a)]{Lus91}, and the same is true of the number of $G$-orbits on $E$ by
Gabriel's theorem.
\end{proof}

For more general types, this will be deduced from the following
(unpublished) result of Baranovsky and
Ginzburg \cite{BaGi}:
\begin{theorem}\label{BaGi}
  Let $\fM$ be a conic symplectic resolution with quantization $\cD$
  and $C$ the preimage of the cone point in $\fM$.  Then the Grothendieck group of
  sheaves of modules over $\cD$ with support in $C$ injects into
  $H^{BM}_{top}(L;\Z)$ under the characteristic cycle map.  
\end{theorem}

\begin{proof}[Proof of \ref{prop-L} in general type]
Consider a highest weight $\la$ and let $E^\la:=E^\la_{\la-\sum
  v_i\al_i}$.  Note that we have a projection map $p\colon E^\la\to E$
forgetting the component $\Hom(V_i,W_i)$.  
By Theorem \ref{BaGi}, if we show that Lusztig's construction
supplies simples with support in $C$ whose characteristic classes span
$H^{BM}_{top}(C;\Z)$, then we will know that these are the only
simples supported on $C$.  
 
Given a union of components $D\subset \Lambda\subset T^*E$, we can construct a
corresponding Lagrangian subvariety of $T^*E^\la$ by taking the 
preimage of $D$ and then its image under the correspondence $T^*E\leftarrow
E^*\oplus E^\la\to T^*E^\la$.  Then intersecting with the stable locus and
projecting to $\fM$, we obtain a union of components (possibly empty) of $C\subset \fM$, which we
denote $n(D)$.  
Each component  of $C\subset \fM^\la$ obtained as $n(D)$ for some component
$D\subset \Lambda$, since $n$ is the left inverse of the map
$\kappa\circ \iota$ defined by Saito \cite[4.6.2]{Saito02}.
There is a corresponding operation on $\Diff_E$-modules $\cM$, which
is to consider $\red(p^*\cM)$.  One can see directly that
$\supp(\red(p^*\cM))=n(SS(\cM))$.  

There is an order on the components of $\Lambda$ such that for
each component $D$, one can construct one of Lusztig's sheaves $L_D$ which has
multiplicity 1 along that component and trivial multiplicity along
higher ones, by \cite[6.2.2(2)]{KS97}.  
Thus, the classes of the sheaves $\red(p^*L_D)$ with $n(D)$ ranging
over the components of $C$ generate 
$H^{BM}_{top}(C;\Z)$.  This shows that if $\preO$ for $E$ contains
any objects which are not produced from Lusztig's construction, they
are killed by performing $\red\circ p^*$.

  Now, let $\cM$ be an arbitrary object in $\preO$ for $E$.  We claim
  that there is some $\la$
  such that the pullback $p^*\cM$ is not
  killed by the reduction functor $\red$, since its support is not in
  the unstable locus, i.e. $n(  SS(\cM))\neq \emptyset$.
If, for example, we choose $w_i\geq v_i$, then for any $e\in T^*E$, there
is always a point in
$p^{-1}(e)\in E^*\oplus E^\la\subset T^*E^\la$ where the maps
$q_i\colon V_i\to W_i$ are injective.  Thus,
there is no nontrivial subrepresentation killed by all $q_i$, since there is no
nontrivial subspace killed by all $q_i$.  
  Since $p^*\cM$ is
  not killed by $\red$, it's a summand of $p^*L_\Bj$, the
  pullback of one of
  Lusztig's sheaves. Since \[\Hom(L_\Bj ,L_\Bj )\cong \Hom(L_\Bj
  ,p_*p^*L_\Bj )\cong \Hom(p^*L_\Bj ,p^*L_\Bj ),\] every summand of
  $p^*L_\Bj $ is the pullback of a summand of $L_{\Bi}$.   Thus $\cM$ is a summand of $L_{\Bj}$ and we're done.
\end{proof}

\subsection{Twisting functors}
\label{sec:twisting-functors}

\nc{\GWeyl}{\mathsf{W}}

The category $\cO$'s attached to different GIT parameters are related
by functors, which we call {\bf twisting functors}, introduced in
\cite[\S 6]{BLPWquant}.  In that paper, we focused more on the induced
functors on the algebraic categories $\cOa$ for different parameters,
but these functors have a natural geometric interpretation.

For each quiver variety, we can identify the sets of GIT parameters
with $\mathfrak{H}^*$, the  dual Cartan of $\fG$.  Nakajima's usual
stability condition is identified with the dominant Weyl chamber, and
each GIT wall corresponds to the vanishing set of a coroot of $\fG$,
though not all coroots contribute.  In fact, there are only finitely
many GIT walls, so if $\fG$ is not finite type, necessarily almost all
coroots do not contribute a wall. 
Let $\GWeyl$ be the Weyl group of $\fG$.  For any $w\in \GWeyl$, the
image of the dominant Weyl chamber under $w$ lies in a single
GIT chamber, defining a map from $\GWeyl$ to the set of chambers, with
image given by the chambers in the Tits cone.  
Choose $\eta\in \mathfrak{H}^*$, a strictly dominant integral weight.

\begin{proposition}[\mbox{\cite[Cor.  B.1]{BLPWquant}}]
  For any fixed character $\chi\colon \fg\to \C$ and any finite subset $\GWeyl_0\subset \GWeyl$,
  there is an integer $n\gg 0$ such that
  localization holds for the character $\chi+nw\cdot \eta$ for
  every $w\in \GWeyl_0$ on the GIT quotient for $w\cdot \eta$.
\end{proposition}
We can identify the category $\cOg$ for different GIT parameters with
category $\cOa$ for different quantization parameters using
localization functors; throughout this section, we'll implicitly
identify $\cOg$ for different parameters that differ by integral
amounts, using tensor product with quantized line bundles (the {\bf
  geometric twisting functors} of \cite[\S 6]{BLPWquant}).  
We let $\cOg^{v\eta}$ be the geometric category $\cO$ of the GIT
quotient $\mu^{-1}(0)/\!\!/_{v\eta}G_\mu$, and $\LLoc^{v\eta}_\chi$ and $
\R\Gamma^{v\eta}_\chi$ be the localization and $\bS$-invariant sections functors on this GIT
quotient at the character $\chi$.

In this notation, identifying $\cOg^{v\eta}$ with $\cOa$ for the
parameter $\chi+nv\cdot \eta$  intertwines the twisting functors from
\cite[\S 6]{BLPWquant} (in
\cite{BLet}, these are called {\bf wall-crossing functors}) with the functors given by 
\[\EuScript{T}^{\eta_1,\eta_2}_\chi:=\LLoc^{\eta_1}_\chi\circ
\R\Gamma^{\eta_2}_\chi\colon \cOg^{\eta_2}\to \cOg^{\eta_1}\]
by  \cite[6.29]{BLPWquant}.
Note that these functors depend on $\chi$; compositions of these
functors that begin and end at a single $\eta$ generate
a very large group of autoequivalences of the derived category
$D^b(\cOg)$.  On the other hand, this dependence of $\chi$ is limited
in an important way: the functors
$\EuScript{T}^{\eta',\eta}_{\chi+n\eta}$ and
$\EuScript{T}^{\eta',\eta}_{\chi+n\eta'}$ stabilize for $n\gg 0$, since
the categories $\cOa$ are related by twisting functors
$\Phi^{\chi+n\eta}_{\chi+m\eta}$ which are equivalences intertwining the
functors $\LLoc^{\eta_1}_{\chi+n\eta}, \R\Gamma^{\eta_2}_{\chi+n\eta}$ for $m,n\gg 0$.  We
can also describe these functors in terms of reduction functors and
their adjoints.
\begin{proposition}[\mbox{\cite[(4.10)]{BLet}}]\label{adjoint-twist}
  For $n\gg 0$, we have isomorphisms of functors
  $\EuScript{T}^{\eta',\eta}_{\chi+n\eta'}\cong \red^{\eta'}\circ \red^\eta_*$ and
$\EuScript{T}^{\eta',\eta}_{\chi+n\eta}\cong \red^{\eta'}\circ \red^\eta_!$. 
\end{proposition}
Only the second equality is proven in \cite{BLet}, but the first is
simply the adjoint of the second.
\excise{\begin{proof}
By taking adjoints, it suffices to show the second isomorphism.
Since localization holds at $\chi+n\eta$ for the stability condition
$\eta$, we have that $\red^\eta_!(\cD_{\chi+n\eta})$ is just the module
$\Diff_{E}/\Diff_{E}\cdot \mu_{\chi+n\eta}(\fg)$. Of course,
$\red^{\eta'}(\Diff_{E}/\Diff_{E}\cdot \mu_{\chi+n\eta}(\fg))$ is
$\cD_{\chi+n\eta}$ (this time on $\mu^{-1}(0)/\!\!/_{\eta'}G_\mu$),
with the functor inducing the obvious isomorphism of endomorphism
rings.  On the other hand, by the localization we already observed,
$\cD_{\chi+n\eta}$ generates the category of $\cD_{\chi+n\eta}$-modules,
so $\EuScript{T}^{\eta',\eta}_{\chi+n\eta}$ is the unique functor that
sends $\cD_{\chi+n\eta}$ (on $\mu^{-1}(0)/\!\!/_{\eta}G_\mu$) to $\cD_{\chi+n\eta}$ (on $\mu^{-1}(0)/\!\!/_{\eta'}G_\mu$)  and we are done.
\end{proof}}

Let $\EuScript{T}_w$ be the functor given by
$\EuScript{T}^{vw^{-1}\cdot \eta,v\cdot\eta}_{\chi+nv\cdot \eta}$ for $n\gg
0$, for all $v\in \GWeyl$.  
\begin{proposition}
  The functors $\EuScript{T}_w$ define a strong action of the Artin
  braid group of $\fg$ on the categories $\cOg^{v\cdot \eta}$ for
  $v\in \GWeyl$.  
\end{proposition}
\begin{proof}
  This follows immediately from \cite[6.32]{BLPWquant}; the braid
  relations for the Artin braid group are relations in the
  Weyl-Deligne groupoid of the Coxeter arrangement. The important
  point here is that the path $v\eta\to vs_{i_1}\eta \to
  vs_{i_1}s_{i_2}\eta\to \cdots $ for any reduced expression
  $s_{i_m}\cdots s_{i_1}$ is minimal length (that is, it crosses the
  minimal number of hyperplanes to connect those two points); thus,
  the functor $\EuScript{T}_{s_{i_m}}\cdots
  \EuScript{T}_{s_{i_1}}=\EuScript{T}_{w}$ is independent of reduced
  expression.  This establishes all of the Artin braid relations.
  Since the isomorphism $\EuScript{T}_{w}
  \EuScript{T}_{w'}=\EuScript{T}_{ww'}$ for
  $\ell(w)+\ell(w')=\ell(ww')$ is associative (it induces a unique
  isomorphism $\EuScript{T}_{w}
  \EuScript{T}_{w'} \EuScript{T}_{w''}=\EuScript{T}_{ww'w''}$ when
  $\ell(w)+\ell(w')+\ell(w'')=\ell(ww'w'')$), this action is strong.
\end{proof}

By work of Maffei \cite[Th. 26]{MaffWeyl}:
\begin{proposition}
  There is a $H\times \bS$-equivariant isomorphism of symplectic varieties
  \[\phi\colon \mu^{-1}(0)/\!\!/_{\eta}G_\nu\to \mu^{-1}(0)/\!\!/_{w
    \eta}G_{w\nu}.\]
\end{proposition}
\begin{proof}
  The only part we need to prove is $H\times \bS$-equivariance.  This
  is the same as proving the $H\times \bS$-invariance of the
  correspondence $Z_i^\la(\Bv)$ of \cite[Def. 26]{MaffWeyl}.  Expanding
  on 
  Maffei's notation, we let 
\[T_i^{\operatorname{out}}:= \bigoplus_{e\colon i\to j}V_j\cong
\oplus_jV_j\otimes \C^{ \epsilon_{i,j}}\qquad
T_i^{\operatorname{in}}:= \bigoplus_{e\colon j\to i}V_j\cong
\oplus_jV_j\otimes \C^{\epsilon_{j,i}}\qquad
T_i:=T_i^{\operatorname{out}}\oplus T_i^{\operatorname{in}}.\]
The group $H\times \bS$ acts on the spaces
$T_i^{\operatorname{out}},T_i^{\operatorname{in}}$ in two slightly
different ways:  $H$ always acts
through its natural action on $\C^{\epsilon_{i,j}}$ on
$T_i^{\operatorname{out}}$ and by the dual action on $T_i^{\operatorname{in}}$; in the A-action, $\bS$ acts
with weight 1 on $T_i^{\operatorname{out}}$ and weight 0 on
$T_i^{\operatorname{in}}$ and in the B-action by weights $0$ and $-1$ respectively.  Thus, the effect of this action on 
  \[ a_i(s)\colon V_i\to T_i \qquad b_i(s)\colon T_i\to V_i\] is
  simply the induced action on the mapping spaces (with $V_i$ having
  trivial $ H\times \bS$-action) for the A-action and B-action respectively.  The same is true of the maps 
 \[ a_i(s')\colon V_i'\to T_i \qquad b_i(s')\colon T_i\to V_i'.\]
 Now consider Maffei's conditions:
 \begin{itemize}
 \item [C1:] This condition is just that maps along arrows not touching $i$ are
   unchanged.  This is obviously unchanged by the $H\times
   \bS$-action.
\item [C2:] This is that $\ker b_i(s)=\operatorname{image} a_i(s')$.
  The A-action and B-action differ by a scalar, so their induced
  actions on the subspaces of $T_i$ coincide.  
  Both this kernel and image transform according to the action of $H\times
   \bS$ on $T_i$, so their equality is invariant.
\item [C3:] This states that $a_i(s')b_i(s')=a_i(s)b_i(s)$ (since we
  are only considering the 0 level of the moment map).  Both of these
  maps transform according to pre-composition with the A-action and
  post-composition with the B-action
  on $T_i$, so their equality is preserved.
\item [C4:] This condition is just that the points both lie in the zero
  level of the moment map; this is $H\times
   \bS$-invariant, since the moment map is $H\times
   \bS$-equivariant with $H$ acting trivially and $\bS$ acting by
   scaling on $\fg$.  
 \end{itemize}
Thus the correspondence is invariant, and the induced isomorphism is equivariant.
\end{proof}

We let $(\cOg)_\nu^\eta$
denote geometric category $\cO$ for $ \mu^{-1}(0)/\!\!/_{\eta}G_\nu$ and an integral
choice of $\chi$; if we omit $\eta$, it is assumed to be dominant, so
the underlying variety is $\fM^\la_\nu$.
\begin{proposition}
  The isomorphism of varieties $\phi$ induces an equivalence of categories.
  $(\cOg)_\nu^\eta\cong (\cOg)^{w\eta}_{w\nu}$ for any fixed
  cocharacter $\bT\to H$.  
\end{proposition}
\begin{proof}
We have Kirwan maps
\[\mathsf{K}_\nu\colon(\fg_\nu^{G_\nu})^*\to
H^2(\mu^{-1}(0)/\!\!/_{\eta}G_\nu)\qquad
\mathsf{K}_{w\nu}\colon (\fg_{w\nu}^{G_{w\nu}})^* \to H^2(\fM^\la_{w\nu})\]
If $\al_i^\vee(\la-\nu)$ and $\al_i^\vee(\la-w\cdot\nu)$ are both
positive, then $G_\nu$ and $G_{w\cdot \nu}$ are products of equal
numbers of general linear groups and we have a canonical isomorphism $(\fg_\mu^{G_\mu})^*\cong
(\fg_{w\mu}^{G_{w\mu}})^*$. In the degenerate cases where one of the
vertices of the quiver gives 0 in the dimension vector, we add in a
trivially acting $\C^*$ to fix this isomorphism (which is killed by
the Kirwan map).  Under Maffei's isomorphism, we have that
$\phi^*\mathsf{K}_{w\cdot \nu}(\chi)\cong \mathsf{K}_{\nu}(w^{-1}\cdot \chi)$.
This shows us how to compare quantizations on the two varieties by
comparing their periods.  

By \cite[3.14]{BLPWquant}, 
the quantization $\cD_{\chi-\nicefrac{\rho_\nu}{2}}$ of
$\mu^{-1}(0)/\!\!/_{\eta}G_\nu$ has period
$\mathsf{K}_\nu(\chi)$ and 
the quantization $\cD_{w\chi-\nicefrac{\rho_{w\nu}}{2}}$ of
$\mu^{-1}(0)/\!\!/_{w\eta}G_{w\nu}$ has period $\mathsf{K}_{w\nu}(w\chi)$. Here $\rho_\mu$ is the
character of $\det(E_\mu)$; this indexing is chosen so that
$\cD_{\chi-\nicefrac{\rho_\mu}{2}}^{\operatorname{opp}}\cong
\cD_{-\chi-\nicefrac{\rho_\mu}{2}}^{\mbox{}}$. 

Thus isomorphism $\phi$ identifies these two quantizations and thus
category $\cO$ over them by $H$-equivariance.
\end{proof}
Thus, in place of fixing a weight space and considering all GIT
conditions, we can instead fix the dominant stability condition and
vary the weight space.  
\begin{definition}
  We let the functors \[\mathscr{T}_{w}\colon
  (\cOg)_\nu^\eta\to (\cOg)^{w^{-1}\eta}_{\nu}\cong (\cOg)_{w \nu}^\eta\] be the transport of the
  $\EuScript{T}_w$ via this isomorphism.  These again define a strong
  action of the Artin braid group.
\end{definition}

When $\eta$ is chosen to be dominant, there is another such braid
action on any category with a categorical $\fG$-action, that given by
Rickard complexes $\Theta_i$ as defined by Chuang and Rouquier
\cite[\S 6.1]{CR04}; these were shown to satisfy the braid relations
by Cautis and Kamnitzer \cite[6.3]{CK3}.
These are compared in recent work of Bezrukavnikov and
Losev \cite{BLet}, building on work of Cautis, Dodd, and Kamnitzer
\cite{CDK}.

\begin{proposition}[\mbox{\cite[3.4]{BLet}}]
 On the sum $\bigoplus_\mu
  D^b(\bigoplus_{w\in \GWeyl}(\cOg)_{w\nu})$, we have an isomorphism of
  functors $\Theta_i\cong\mathscr{T}_{s_i}$. 
\end{proposition}

\excise{if we let
$\red^1$ and $\red^2$ be the reduction functors for two different GIT
parameters (and thus two different resolutions), then the twisting
equivalence is given by $\red^2\circ \red^1_!$, with inverse given by
$\red^1\circ \red^2_*$.  

\begin{proposition}
  This functor $\red^2\circ \red^1_!$ agrees with the twisting functor
  $\LLoc^2\circ \Rsecs^1$.
\end{proposition}
\begin{proof}
  Note that $\Rsecs^2\circ \red^2\cong \Rsecs^1\circ\red^1$, since these
  functors agree on the algebra $\Diff$ itself; both send it to $A$.
  Thus, we have that 
  \begin{equation*}
    \LLoc^2\circ \Rsecs^1 N\cong \LLoc^2\circ
    \Rsecs^1\circ \red^1\circ \red_!^1 N\cong  \LLoc^2
    \Rsecs^2\circ \red^2(\red_!^1 M)
  \end{equation*}
\end{proof}}

\section{Quiver varieties: special cases}
\subsection{Tensor product actions}

\nc{\rep}{\mathsf{V}}

A {\bf tensor product action} of $\bT$ on $\fM^\la$ is one induced by a cocharacter 
$\bT\to PG_\Bw$.  This is the same as assigning weakly decreasing weights
$\vartheta_1,\dots, \vartheta_{\ell}$ to the different new edges in the
Crawley-Boevey quiver.  By Assumption \ref{positivity}, we have
$|\vartheta_i-\vartheta_j|>\max(|a|,|b|)$.  These actions played an important role in
Nakajima's definition of the  ``tensor product quiver variety''
\cite{Nak01}.  The eigenspaces of this action on $W$ decompose this
space into a sum $W^{1}\oplus\cdots \oplus W^{\ell}$, which are again
ordered by increasing eigenvalue.  Let $\la_i$ be a weight such that
$\al_j^\vee(\la_i)=\dim W^i_j$.

In this case, the algebra $\bar{W}^\vartheta_\nu$ is isomorphic by
\cite[3.5]{WebwKLR} to one
which appeared earlier in the work of the author \cite[\S4.2]{Webmerged}; this is the algebra
$\tilde{T}^\bla_\nu$.  The steadied quotient of this
algebra corresponding to positive powers of the determinant characters
is the {\bf tensor product algebra} $T^\bla$ also defined in
\cite[\S 4.2]{Webmerged}.  This is an algebra whose representation category categorifies the tensor
product $\rep_{\la_1}\otimes \cdots \otimes \rep_{\la_\ell}$ of the
simple $\fG$-representations $\rep_{\la_i}$ of highest weights $\la_i$ in an
appropriate sense.

\begin{theorem}\label{alg-bla}
  Assume  $\bT$ is given by a tensor product action.
Then hypothesis \hyperlink{dagger}{$(\dagger)$} holds in  this case.
\end{theorem}

Recall that $L_\Bi$ is the object in the derived category of D-modules defined in
\eqref{eq:L-def}.  We call a loading $\Bi$ {\bf violating} if there exist
$k\in \R$ with $k<\vartheta_j$ for all $j$ with $\Bi(k)\neq 0$.  When (as
in \cite[\S 3.2]{WebwKLR}) we think of a loading as encoding a horizontal
slice in a Stendhal diagram (as defined in \cite[4.1]{Webmerged}), the
violating loadings correspond to slices with a black strand left of
all reds, meaning that the corresponding diagram is violating (as
defined in \cite[4.3]{Webmerged}).  

The summands of  $L_\Bi$ with $\Bi$ violating are precisely the set of
sheaves that Li denotes $\mathcal{N}_{V,D^\bullet}$ in \cite[\S
8]{Litensor}.  He shows that these coincide with the summands of
$L_\Bi$ for $\Bi$ arbitrary which have unstable singular support,
i.e. are killed by $\red$, a set he denotes $\mathcal{M}_{V,D^\bullet}$.

\begin{lemma}[\mbox{\cite[8.2.1(4)]{Litensor}}]
 We have $\mathcal{M}_{V,D^\bullet} =\mathcal{N}_{V,D^\bullet}$, that
 is, a summand $L$ of $L_\Bi$ for any arbitrary loading $\Bi$ is killed by
 $\red$ if and only if it is a summand of $L_\Bi$ for $\Bi$ violating. 
\end{lemma}
\excise{\begin{proof}
  This is just a restatement of \cite[8.2.1(4)]{Litensor}; in Li's
  notation $\mathcal{N}_{V,D^\bullet}$ is the set of simples $\preO$
  which are summands of $L_\Bi$ with $\Bi$ violating, and $\mathcal{M}_{V,D^\bullet}$ is the set of
simples with unstable microsupport, i.e., those killed by $\red$.
Thus, the statement of the lemma is simply that $\mathcal{N}_{V,D^\bullet}=\mathcal{M}_{V,D^\bullet}$.
\end{proof}}

\begin{lemma}\label{lem:red-zero}
For any tensor product action, any simple in $\preO$ with non-zero
  reduction is a summand of $L_\Bi$ for some loading $\Bi$.  
\end{lemma}
\begin{proof}
By adding an appropriate multiple of $\eta$, we can assume that we have
chosen the character $\chi$ so that localization holds on $\fM^\la$
for $\cD_\chi$ by \cite[Th. A]{BLPWquant}.  Thus, we have a natural equivalence $\cOg\cong
\cOa$.  Recall that $A_\chi$ is the $\bS$-invariant section algebra of
$\cD_\chi$.  Let $A_\chi^k$ be the weight spaces for the induced action
of $\bT$ on $A_\chi$, and $A_\la^{\geq 0}$ the non-negative weight spaces.   Every simple module in $\cOa$ is a quotient of a unique
standard module, which is of the form $A_\chi\otimes_{A_\la^{\geq 0}}M$
where $M$ is a finite dimensional module over $\mathsf{C}(A_\chi)=A_\chi^{
  0}/\sum_{k>0}A_\chi^{ -k}A_\chi^{ k}$, considered as a $A_\la^{\geq
  0}$-module by pullback.  Thus, the number $m$ of simples in $\cOg$
is equal to the number of finite dimensional modules over $\mathsf{C}(A_\chi)$.

  The fixed point subvariety of $\bT$
 is symplectomorphic to the product
  $\fM^{\la_1}\times\cdots \times \fM^{\la_\ell}$ of quiver varieties
  attached to the weights for the different eigenvalues of $\bT$
  acting on $W$.  
Adding another positive multiple of $\eta$ to $\chi$ if necessary, we
can also apply \cite[5.2]{LosOqsr} to show that $\mathsf{C}(A_\chi)$ is
the sections of a quantization of $\fM^{\la_1}\times\cdots \times
\fM^{\la_\ell}$.  
Since the pullback of an ample line bundle to a subvariety is again
ample, the result \cite[5.7]{LosOqsr} shows that adding $\eta$ to
$\chi$ has the effect of adding an ample class to the period of the
quantization of each component of $\fM^{\la_1}\times\cdots \times
\fM^{\la_\ell}$.  Thus, perhaps adding a positive multiple of $\eta$
again, we can assume that localization holds for this quantization of $\fM^{\la_1}\times\cdots \times
\fM^{\la_\ell}$.  In this case, the finite dimensional modules over
the sections are in bijection with core modules, i.e. the geometric
category $\cO$ for the trivial action.  Thus, there are $m$ simple
core modules.   
  Now, by Lemma \ref{prop-L}, $m$ is no more than the number of
  components of the core $C$, that is $m\leq \prod \dim \rep_{\la_i}$.  On the other hand, by
  work of Li \cite[8.2.1(5)]{Litensor}, the
  summands of $L$ which survive under $\red$ are in canonical
  bijection with the canonical basis of $ \rep_{\la_1}\otimes \cdots
  \otimes \rep_{\la_\ell}$.  In particular, the number of them is also $\prod \dim \rep_{\la_i}$. Thus, all simples in
  $\cOg$ must be of this form by the pigeonhole principle.  
\excise{
  The general case follows from this one since for the relative core
  for any tensor product action is contained inside the relative core
  for one with distinct weights.  Thus, any simple in the category
  $\cOg'$ for any tensor product action is a summand of $L'$ for a
  possibly different action.  But using Li's work again, we see that
  the number of simples in $\cOg'$ that lie in the category $\cOg$ is no more
  than the number that are summands of $L$; thus they must all be
  summands of $L$.}
\end{proof}

\begin{proof}[Proof of Theorem \ref{alg-bla}]
  This is sufficiently similar to the proof of Proposition \ref{prop-L} that we
  only give a sketch. As before, we let $E^\la:=E^\Bw_\Bv$.  We can add a weight $\la'$ to
  $\la$, and pull back by the map $p\colon E^{\la+\la'}\to E^\la$
  induced the projection $\C^{\al_i^\vee(\la+\la')}\to
  \C^{\al_i^\vee(\la)}$.   Assume that $\cM$ is a
  $\Diff_{E^\la}$-module in $\preO$.  If $\la'$ is chosen so that
  $\al_i^\vee(\la')\geq v_i$, then the pullback $p^*\cM$ has a stable
  point in its singular support, so $\red(p^*\cM)\neq 0$.
Thus, we must
  have that $p^*\cM$ is a summand of $L_\Bi$ for some $\Bi$.  This, in turn,
  shows that our original sheaf was also of this type.
\end{proof}

Having established the hypothesis \hyperlink{dagger}{$(\dagger)$}, we
can give a description of the category $\dOg$ as $\stR\dgmod$,
where $\stR $ is as defined in Section
\ref{sec:constr-objects-preo}. Corollary \ref{dagger-W} and
\cite[3.6]{WebwKLR} will now establish Case 1 of Theorem \ref{main1}:
\begin{corollary}\label{tensor-equivalence}
  The category $\dOg$ for a tensor product action is quasi-equivalent
  to $T^\bla\dgmod$, and $\dpreO$ to $\tilde{T}^\bla\dgmod$.
\end{corollary}

Furthermore \hyperlink{maltese}{$(\maltese)$} holds in this case as
well by Corollary \ref{dagger-implies-maltese}.  Thus, if $\bT$ has
isolated fixed points, then the algebra $T^\bla$ must be Koszul.

\begin{theorem}
  If all the weights $\bla$ are minuscule, then the algebra $T^\bla$
  is Koszul.
\end{theorem}
\begin{proof}
  As noted in Nakajima \cite{Nak01}, the fixed points of $\bT$ acting
  on the quiver variety $\fM$ are the product of the quiver varieties
  attached to the weights $\la_i$.  If all $\la_i$ are minuscule, then
  these quiver varieties are all finite sets of points, so $\bT$ has
  isolated fixed points.  Thus Theorem \ref{maltese-koszul} implies
  the Koszulity of these algebras.  
\end{proof}

At least in finite type, we can also easily understand the cell
filtration in terms of the isotypic and BBD filtrations introduced in
Section \ref{sec:background}.

\begin{proposition}
  If $\Gamma$ is an ADE Dynkin diagram, then the 2-sided cell,
  isotypic and BBD filtrations all coincide.
\end{proposition}
\begin{proof}
  Since the 2-sided cell filtration is sandwiched between the isotypic
  and BBD filtrations, if these two coincide, the 2-sided cell
  filtration must match them.  By \cite[Rmk. 3.28]{Nak98}, the
  isotypic and BBD filtrations coincide, so this is indeed the case.
\end{proof}

We can generalize Corollary \ref{tensor-equivalence} a bit to include interactions between
different category $\cO$'s for the different tensor product actions.

The set of cocharacters $\vartheta\colon \bT\to T_W$ carries an $S_n$
action which preserves the weight spaces of $\bT$ and the set of
weights which occur, while permuting the order of the weight spaces.
There are functors relating the category $\cO$'s for different $\bT$
actions in the most obvious way possible: we have a the obvious
inclusion $i^\vartheta\colon D_{\cOg^\vartheta}\to
D^b(\cD\mmod_{\operatorname{hol}})$ into the derived category of
$\cD$-modules with Lagrangian support and
its left and right adjoints $i_!^\vartheta, i_*^\vartheta$ defined in \cite[8.7]{BLPWgco}.  For two
different $\bT$-actions, we can always take the {\bf shuffling functor}
$\mathscr{S}^{\vartheta,\vartheta'}:=i_*^{\vartheta'}\circ i^\vartheta\colon D_{\cOg^\vartheta}\to
D_{\cOg^{\vartheta'}}$.

Similarly, we have an inclusion functor 
$I^\vartheta\colon \dpreO\to \Diff_E\mmod$ defined
before reduction. 
\begin{theorem}\label{tensor-shuffling}
  The shuffling functors between $D_{\cOg^{\sigma\vartheta}}$ for all
  permutations $\sigma$ and a fixed tensor product action $\vartheta$ give
  a weak action of the action groupoid for the $\ell$-strand braid
  group on total orders of an $\ell$ element set.  This is intertwined
  by the equivalences with the action of this group on dg-modules over
  $T^{\sigma\bla}$ defined in \cite[6.18]{Webmerged}.
\end{theorem}
\begin{proof}
Let $\coho\colon \dOg\to T^\bla\dgmod$ and  $p\coho\colon \dpreO\to
\tilde{T}^\bla\dgmod$ be the equivalences of Corollary
\ref{tensor-equivalence}.  Note that by definition,
\begin{equation}
\coho\circ
\red(M)\cong T^\bla\Lotimes_{\tilde{T}^\bla}p\coho(M)\qquad p\coho\circ
\red_!(M)\cong \coho(M)\label{eq:coho}
\end{equation}
where in the second equality, $\coho(M)$ is
inflated to a $\tilde{T}^\bla$-module by the surjection
$\tilde{T}^\bla\to T^\bla$. 

Consider the  bimodule $\Hom(I^{\sigma\bla}(L'),
I^{\bla}(L))$ where $L,L'$ are the sum of all $L_\Bi$'s with
$|\Bi|=\la-\mu$ defined using the $\bT$ actions corresponding to
$\bla$ and $\sigma\bla$ respectively.  By \cite[4.12]{WebwKLR}, this
bimodule is the bimodule
$B^{\bla,\sigma\bla}_\nu$ as defined in \cite[\S 2.5]{WebwKLR}; for a
tensor product action, this is exactly the bimodule
$\tilde{\mathfrak{B}}_\sigma$ defined in
\cite[6.3]{Webmerged} by  \cite[3.6]{WebwKLR}.  Since
$L,L'$ are generators of the dg-categories $D_{p\cOg^{\bla}}, D_{p\cOg^{\sigma\bla}}$, we have functorial isomorphisms
\begin{equation}
 \Ext^\bullet(I^{\sigma \bla} M,I^{\sigma\bla}\circ p\coho^{-1}(
 \tilde{\mathfrak{B}}_\sigma\Lotimes_{\tilde{T}^\bla} N) \cong
\Ext^\bullet(I^{\sigma \bla} M,I^{\bla}\circ  p\coho^{-1}(N))\label{eq:B-commute}
\end{equation}
for $N$ in $\tilde{T}^\bla\dgmod$ and $M\in \cOg^{\sigma\bla}$.

Now, we consider the case where $M=\red_!\cM$ and $N$ is the inflation
of an module $\cN$ in ${T}^\bla\dgmod$.  In this case
$p\coho^{-1}(N)\cong \red_!\coho^{-1}(N)$ by \eqref{eq:coho}. Thus,
the RHS of \eqref{eq:B-commute} can be rewritten 
\begin{align*}
 \Ext^\bullet(I^{\sigma \bla} \red_!\cM,I^{\bla}(
  \red_!\coho^{-1}\cN))&=\Ext^\bullet(\red_! i^{\sigma \bla} \cM,
                         \red_! i^{\bla}(\coho^{-1}\cN))\\
&=\Ext^\bullet(i^{\sigma \bla} \cM,
                         i^{\bla}(\coho^{-1}\cN))\\
&=\Ext^\bullet(\cM,
                         \mathscr{S}^{\sigma} (\coho^{-1}\cN))
\excise{
\mathfrak{B}_\sigma\Lotimes_{T^\bla} \coho(M)&\cong 
T^\bla\Lotimes_{\tilde{T}^\bla}\tilde{\mathfrak{B}}_\sigma\Lotimes
                                p\coho(\red_!N) &&\text{by \eqref{eq:coho}}\\
&\cong
T^\bla\Lotimes_{\tilde{T}^\bla} p\coho\circ I_!^{\sigma\bla}\circ
  I^{\bla}(\red_!N) &&\text{by \eqref{eq:B-commute}}\\
&\cong
\coho\circ\red\circ  I_!^{\sigma\bla}\circ I^{\bla}(\red_!N)
                                                &&\text{by \eqref{eq:coho}}\\
&\cong
\coho\circ  i_!^{\sigma\bla}\circ i^{\bla}(N) &&\text{by  Lemma \ref{I-intertwine}}}
\end{align*}
On the other hand, the LHS of \eqref{eq:B-commute} can be rewritten as 
\begin{equation}
 \Ext^\bullet(\red_! i^{\sigma \bla} \cM,\red_! i^{\sigma \bla}\circ \coho^{-1}(
 {\mathfrak{B}}_\sigma\Lotimes_{T^\bla} \cN) )\cong \Ext^\bullet(\cM,\coho^{-1}(
 {\mathfrak{B}}_\sigma\Lotimes_{T^\bla} \cN)
\end{equation}
Since $\cM$ is arbitrary, this shows the isomorphism of functors
${\mathfrak{B}}_\sigma\Lotimes_{T^\bla}\coho\cong
\coho\circ \mathscr{S}^{\sigma}$ as desired.  The functors ${\mathfrak{B}}_\sigma\Lotimes_{T^\bla}-$ satisfy the relations
of the braid groupoid by  \cite[6.18]{Webmerged},
so the same is true of $\mathscr{S}^{\sigma}$.
\end{proof}

\subsection{Affine type A}

Of course, the case of a non-tensor product action is more
complicated; there are of necessity more simple modules. The first
interesting such case is when $\fg$ is $\widehat{\mathfrak{sl}}_{e}$.
In this case, we'll identify the nodes of $\Gamma$ with the residues
in $\Z/e\Z$ with arrows from $i$ to $i+1$ for every $i$.  
We'll want to include the case where $e=1$, that is, the Jordan quiver.
Since this orientation contains an oriented cycle, we will take
$a=b=1$, so $\bS$ is just the scaling action on $T^*E^\la_\mu$.

In this case, we can think of the space $V\cong \oplus_i V_i$ as a
single $\Z/e\Z$-graded space.  For each $i$, we have a map
$x_{i,i+1}\colon V_i\to V_{i+1}$.  We can view the sum
$x=\sum x_{i,i+1}$ of the maps along the edges as a single map
$V\to V$.  This is homogeneous of degree 1 in the $\Z/e\Z$-grading; we
use this convention throughout the section.  With a choice of highest
weight $\la$ of level $\ell$ (i.e. with $\sum_i\al_i^\vee(\la)=\ell$),
we obtain a Crawley-Boevey quiver with $\ell$ new edges.  We enumerate
these edges $e_1,\dots, e_\ell$, and let $r_i$ be the vertex, thought
of as a residue in $\Z/e\Z$, attached to the edge $e_i$.  We let
$q_{(i)}$ be the map $V\to \C$ given by the row of the matrix of $q$
corresponding to the edge $e_i$.

As we indicated before, in this case $H/Z$ is a quotient of
$PG_{\Bw}\times \C^*$ by a finite central subgroup.  Thus, a rational cocharacter into $H/Z$ is
essentially the same as choosing
rational numbers $\vartheta_i$ given by the weights of the new edges
(though these are only uniquely specified up to simultaneous translation),
as well as the rational number $\ck$ giving the weight of the projection
to the last factor; if $\ck=0$, then we have a tensor product action,
and can use the results of the previous section.  

We assume for the sake of simplicity that the
weights $\vartheta_i$ are all distinct modulo $\ck$.  This is 
sufficient for the associated action to have isolated fixed points
on $\fM^\la_\nu$ for all $\nu$ and for every component of the pre-core
to come from a lift with no positive dimensional fixed subspaces.
In fact, we could strengthen this a bit to a necessary and sufficient
condition for isolated fixed points for all $\mu$: that
$\vartheta_i-\vartheta_j-\ck(r_i-r_j)$ is not divisible by $e\ck$.  

We wish to understand the structure of the relative precore.  As discussed
previously, this is the unstable locus associated to the action of
$\tilde{G}$ with the character $\nu$.  Thus, we consider the
Kirwan-Ness stratification of the relative precore.  We should note, this
stratification depends on a norm of the group $\tilde{G}\cong
G_\Bv\times \bT'$.  This is the same as fixing a norm invariant under
the Weyl group $\prod S_{v_i}$ on the rational Lie algebra
$\tilde{\ft}_\Q$ of a
maximal torus $\tilde{T}$.   

Fixing $\vartheta_i$ gives a choice of lift
$\vartheta\colon \bT'\to \tilde{G}$; we let $  \tilde{T}$ be the
subgroup generated by  $\bT'$ and the diagonal matrices on $\C^{v_i}$.  We have a
basis of $\tilde{\ft}$ given by the derivative of $\dot{\vartheta}_i$
and the derivatives $\dot\epsilon_{i,j}^\vee$ of the coordinate cocharacters $\epsilon_{i,j}^\vee$
which act with weight 1 on the $j$th coordinate of $\C^{v_i}$ and with
weight 0 on all others.   This induces coordinates for rational (or
real, etc.) cocharacters given by $\lift=(\nu(\lift),\{\lift_{i,j}\})$.  We
define a series of norms (for $u\in \Z$) given by 
\begin{equation}
|\lift|_u^2=\nu(\lift)^2+\sum_{i,j}
(\lift_{i,j}-u\nu(\lift))^2.\label{eq:norm}
\end{equation}
This means that if $u\gg 0$, then actions with eigenvalues as large as
possible are ``preferred'' (i.e. have smaller norm), whereas if $u\ll
0$, then there is a symmetric preference for negative eigenvalues.  

Let us translate this general framework into the language of quivers.
We consider rational cocharacters $\lift$ with $\nu(\lift)=1$. The
difference $\lift-\vartheta_i$ is a rational cocharacter of $G_\Bv$,
so we can think of the decomposition of $V$ into eigenspaces for this
action as a grading by $\Q$ by {\it minus} the eigenvalues. We can
extend this grading to $W$ using the action of $H$ on this space; this
grading on $W$ is fixed, and the degrees of basis vectors are given by
$-\vartheta_i$.

\begin{lemma}
  This lift is destabilizing for a point of $T^*E$, thought of as a
  representation of the doubled quiver, if and only if the maps $q,\bar{q}$ 
  are a sum of homogeneous maps of degree $\leq -1$ and $x,\bar{x}$
  are a sum of homogeneous maps of degree $\leq -1\pm \kappa$.
\end{lemma}
Note, it is immediate
that either $x$ must be nilpotent if $\kappa<0$ and  $\bar{x}$ must be
nilpotent if $\kappa>0$.  Assume for simplicity that $\kappa<0$.  

Choose an integer $u$, a multipartition
$\boldsymbol{\xi}=(\xi_1,\dots, \xi_\ell)$ and a multisegment
$\mathbf{m}$.  Recall that a segment is a pair $(r,n)$ of a residue
$r\in \Z/e\Z$ and a positive integer $n\in \Z_{\geq 0}$; we think of this as the segment
$r,r+1,\dots, r+n$, considered modulo $e$ (hence the need to include
its length instead of its endpoint), and a {\bf multisegment} is a multiset
of segments.  

For a multipartition, we have a diagram consisting of the triples
$\{(i,j,k) | 1\leq j\leq (\xi_k)_i\}$; if we think of each such triple
as a box, these will give the Young diagrams of each individual
$\xi_k$.  Each such box has an attached statistic: its {\bf content} modulo $e$, which is just the
residue class of  $r_k+j-i$ in $\Z/e\Z$.

For a multisegment, we similarly have a diagram consisting
of a row of boxes of length $n$ for each segment $(r,n)$.  Unlike with
a Young diagram, we do not think of these rows as being stacked on top
of each other, or in any particular order.  These boxes have contents
as well, given by the residue classes $r,r+1,\dots$ as we read along
the row.

\begin{definition}
  Consider the conjugacy class of cocharacters
  $\lift_{\boldsymbol{\xi},\mathbf{m}}$ where the eigenvalues that
  appear are
  \begin{itemize}
  \item $\vartheta_k+\kappa(i-j)+i+j-1$ for $(i,j,k)$ in the diagram
    of $\boldsymbol{\xi}$ over the vertex for the residue $r_k+j-i$
    and 
\item the eigenvalue
    $u+(1+\kappa)(g-h)/2,u+(1+\kappa) (g-h+2)/2,\dots,
    u+(1+\kappa)(h-g)/2$ over the vertices for $g,g+1,\dots, h$ for each segment $[g,h]$.
  \end{itemize}

\end{definition}

\begin{proposition}\label{prop:LKN}
  For $u\gg 0$, the Lagrangian Kirwan-Ness strata are indexed by the
  conjugacy classes $\lift_{\boldsymbol{\xi},\mathbf{m}}$, and each of
  these strata has irreducible closure.
\end{proposition}
\begin{proof}
First, we note that each of these cocharacters has a non-empty
Kirwan-Ness stratum: associated to each $\boldsymbol{\xi},\mathbf{m}$
we have a representation of the preprojective algebra where the basis $v_{b}$
is indexed by boxes $b$ of the diagrams of $\boldsymbol{\xi}$ and of
$\mathbf{m}$ such that
\begin{itemize}
\item $x$ acts by sending each box in the diagram of $\boldsymbol{\xi}$ to the next in its row
  $x\cdot v_{(i,j,k)}=v_{(i,j-1,k)}$, and each box in a row of the
  multisegment $\mathbf{m}$ to the next in the row $x\cdot v_{k} =v_{k-1}$.
\item $\bar x$ acts by sending each box to
  the next in its column $\bar x\cdot v_{(i,j,k)}=v_{(i-1,j,k)}$ and each box in a row of the
  multisegment $\mathbf{m}$ to 0, and 
\item $q$ acts by sending $v_{(1,1,k)}$ to
  the corresponding basis vector in $W$ and killing all other basis vectors.
\end{itemize}
  Note that in any compatible grading, we have that the basis
vector for $(1,1,k)$ is a sum of vectors of weight $\geq
\vartheta_k+1$, since otherwise the map $q$ will not have a limit.
Similarly, the box $(i,j,k)$ must be a sum of vectors of weight $\geq
\vartheta_k+\kappa(i-j)+i+j-1$, since otherwise one of the maps along
a row or column would have no limit.   Similarly, the multi-segment must
have elements with gradings spaced out by $-1-\kappa$, and thus with
distances from $u$ given by at least $(1+\kappa)(g-h)/2,(1+\kappa)
(g-h+2)/2,\dots, (1+\kappa)(h-g)/2$.  

Thus, if we consider the norm of such a grading according to
\eqref{eq:norm}, 
it must be at least as great as the norm $|\lift_{\boldsymbol{\xi},\mathbf{m}}|_m$.
This shows the optimality of
$\lift_{\boldsymbol{\xi},\mathbf{m}}$, since the Kirwan-Ness stratum
is non-empty.

Now, consider an arbitrary component of the precore $pC$.  A generic
element $y$ of this component can be considered as a representation of
the pre-projective algebra of the Crawley-Boevey quiver.  Let
$U\subset V$ be the maximal submodule with no support on the new
vertex, that is, the maximal submodule of the $V$'s killed by $q$.
Modding out by $U$, we obtain a stable representation $V/U$, which we
can think of as a point $\bar y$ in $\fM^\Bw_{\Bv-\Bu}$. Let
$\mathbf{m}$ be the multisegment corresponding to the Jordan type of
$x$ on $U$, and let $\boldsymbol{\xi}$ be the multipartition
corresponding to the component of the point $\bar y$.  Note that these
are generically constant on the component.

We can grade $V$ compatibly with a conjugate of
$\lift_{\boldsymbol{\xi},\mathbf{m}}$ such that we obtain a limit: we
assign the weights $(1+\kappa)(g-h)/2,(1+\kappa)
(g-h+2)/2,\dots, (1+\kappa)(h-g)/2$ to each Jordan block on $U$, and
on the quotient $V/U$, we can use the weights
$\vartheta_k+\kappa(i-j)+i+j-1$ for $(i,j,k)$ in the diagram of $
\boldsymbol{\xi}$ by the definition of this component.  
In order to define this grading on $V$, we must choose a splitting
$V/U\to V$.  Since $u\gg 0$, any matrix coefficient with source in
$V/U$ and target in $U$ thus has a very large weight, and does not
interfere with having a limit for the associated lift.  The lift
$\lift_{\boldsymbol{\xi},\emptyset}$ is the optimal cocharacter for
$V/U$ by definition and $\lift_{\emptyset,\mathbf{m}}$ is the optimal
cocharacter for $U$ as argued above.  Thus,
$\lift_{\boldsymbol{\xi},\mathbf{m}}$ is (up to conjugacy) the
optimal character for $y$.   This shows that there is no Kirwan-Ness
stratum open in a component for any other cocharacter, and in
particular, none which is Lagrangian.  Note, this also shows that two
components which both contain stable points correspond to different
Kirwan-Ness strata, since they correspond to different fixed points in
$\fM^\la$.

Now, let $C$ be a Lagrangian component.  As in the proof of Theorem
\ref{alg-bla}, we can consider the preimage of this component under
the map $p\colon E^{\Bw+\Bw'}\oplus (E^\Bw)^*\to T^*E^{\Bw}$ for some dominant weight
$\la'$ induced by the map $\C^{w_i +w_i'}\cong W_i\oplus W_i'\to
\C^{w_i}\cong W_i$ forgetting the
last $w_i'$ components.  This preimage is Lagrangian in $T^*
E^{\Bw+\Bw'}$ and lies in the precore for any $\bT$ action which
acts with very large weights on $W_i'$.  

Without loss of generality, we can order the segments $[m_i,n_i]$ of the
multisegment $\mathbf{m}$ so that the lengths $n_1\geq
n_2\geq \cdots$ are weakly decreasing.  We let $w_i'$ be the number of
multisegments with residue $m_j+n_j\equiv i\pmod e$, and extend $\bT$ so that it has
weights $\{z+\epsilon j\mid m_j+n_j\equiv i\pmod e\}$ on $W_i'$ for $z\gg 0$ chosen
large enough that $p^{-1}(C)$ lies in the precore of the corresponding
action.  

Note, since $x$ is nilpotent, any destabilizing subrepresentation of a
point in $p^{-1}(C)$ contains a vector killed by $x$.  Since each
segment contributes a line of vectors in $U$ killed by $x$, the
dimension vector of this subspace is $\Bw'$. Since the map
$q'\colon V_i\to W_i'$ is chosen generically, this map is injective on the
kernel of $x$ at a generic point in $
p^{-1}(C)$.  Thus, a
generic point in this component is stable.  

Now, let us calculate which Kirwan-Ness stratum a generic point of
this component lies in.  Since the vector with weight $z+\epsilon$ in
$W_{m_1+n_1}$ is hit by some element of $U$, we must have an element of grading at least $-z-\epsilon+1$.  Furthermore,
if $x^{-1}(v)\neq \emptyset$ for any element $v$ whose leading order is $\geq -z-\epsilon+1$ is non-zero, the elements of this preimage
must have leading order $\geq -z-\epsilon-\kappa+2$ in this preimage,
and more generally if $x^{-j}(v) \neq \emptyset$, the leading order of
it elements must be must be an element of grading
$-z-\epsilon-j\kappa+j+1$.  Since $q'$ is injective on the kernel of
$x$, we can assume that some element of the image of $x^{n_1-1}$ over
vertex $m_1+n_1$ has leading order $-z-\epsilon+1$, and its preimages
of have leading orders $-z-\epsilon-j\kappa+j+1$ for $j\leq n_1-1$.  
This fixes the gradings that appear for $n_1$ elements.

Similarly, the genericity of $q'$ shows that there is an element of
the image of $x^{n_2-1}$ which maps non-trivially to the vector of
weight $z+2\epsilon$ in $W_{m_2+n_2}$ which has not already had its
leading order fixed.  As argued above, this shows that we have such an
element with leading order  $\geq -z-2\epsilon+1$, and thus taking
preimages, elements with leading order $-z-2\epsilon-j\kappa+j+1$ for
$j\leq n_2-1$.  

Thus, applying this construction inductively, we find that the leading
orders that appear in $U$ must be greater than the weights that appear in
$\xi_{(n_1),(n_2),\dots}$, the cocharacter associated to the
multipartition where each partition has a single row, with lengths
$n_1,\dots$, and the component partitions have associated weights
$z+\epsilon,z+2\epsilon,\dots$.

Thus, the basis vectors in $V/U$ contribute $\boldsymbol{\xi}$ to our
multipartition, and the subspace $U$ will contribute $\big((n_1),(n_2),\dots\big)$, so the component
$p^{-1}(C)$ is sent under reduction to the component of the relative
core for $\fM^{\Bw+\Bw'}$ corresponding to $\big(\xi_1,\dots,
\xi_\ell, (n_1),(n_2),\dots\big)$.  Since this component only depends
on $\boldsymbol{\xi}$ and $\mathbf{m}$ (the residues $m_i$ are
implicit in the choice of $W_i'$ and the $\bT$-action), this shows that this KN
stratum has a unique component in its closure.  
\end{proof}
We call a loading $\Bi$
{\bf unsteady} in this case if there exists $a\in \R$ with  $a<\vartheta_j$ for
all $j$ with every point in $[a-|\kappa|,a]$ sent to 0 by $\Bi$, but
for some $a'<a-|\kappa|$, we have $\Bi(a')\neq 0$.  If there are
precisely $j$ such values $a'$, we call the loading $j$-unsteady. The algebra
$T^\vartheta$ is the quotient of $\bar W^\vartheta$ by the idempotents
corresponding to unsteady loadings.  For more details, the reader can refer to that given in \cite[\S
4.1]{WebRou}.  
\begin{proposition}\label{O-is-T-affine}
  Every simple in 
  $\preO$ is a 
  summand of $L_\nu$.  The sheaves with unstable characteristic varieties are
exactly the summands of such push forwards where $\Bi$ is unsteady.  
\end{proposition}

\begin{proof}
Corresponding to the actions with Lagrangian KN strata
$\lift_{\boldsymbol{\xi},\mathbf{m}}$, we have loadings
$\Bi_{\boldsymbol{\xi},\mathbf{m}}$.  The sheaf $L_{\Bi_{\boldsymbol{\xi},\mathbf{m}}}$ is
supported on precisely the image of the component $C_{\Bi_{\boldsymbol{\xi},\mathbf{m}}}$ in $E$, and thus
includes the conormal to its smooth locus (which is dense in
$C_{\Bi_{\boldsymbol{\xi},\mathbf{m}}}$) with multiplicity one in its
characteristic cycle. The other components that appear are 
lower in the KN order by Theorem \ref{th:pushforward}.  As
in the proof of Lemma \ref{lem:red-zero}, this means that
$\red(L_{\Bi_{\boldsymbol{\xi},\mathbf{m}}})=0$ unless
$\mathbf{m}=0$.  The sheaves
$\red(L_{\Bi_{\boldsymbol{\xi},\emptyset}})$ give semi-simple objects in $\cOg$
such that taking the highest component appearing in the support
defines a bijection between these objects and the components of the
relative core.  By \cite[6.17]{BLPWgco}, this means that
the semi-simple object
$\red(L_{\Bi_{\boldsymbol{\xi},\emptyset}})$ has a unique simple summand
$\Lambda_{\boldsymbol{\xi}}$ where
$C_{\Bi_{\boldsymbol{\xi},\mathbf{m}}}$ is the highest component in
its support, and the only other simples are supported on strictly
lower components.  Since
$\Lambda_{\boldsymbol{\xi}}$ is complete list of simples in $\cOg$ by
\cite[6.5]{BLPWgco}, this means we have obtained all the simples in
$\cOg$.  

Now, assume there is a simple $M$ in $\preO$ which is not a summand of
$L_{\Bi}$.  There must be some highest component (in the KN order) in
the support of this module, and some simple summand $M'$ of $L_{\Bi}$ which has
the same highest component in its support.  As in the proof of
Proposition \ref{prop:LKN}, we can pull back by the map $p\colon
E^{\Bw+\Bw'}\to E^{\Bw}$ for $\Bw'$ sufficiently large.  The pullbacks
$p^*M$ and $p^*M'$ will
still give two modules in $\preO$ on this new variety, still with the
same highest component in their support, and now neither is killed by
$\red$.  But as argued above, two simples with the same highest
component must be sent under $\red$ to the same simple.  Thus, we have
$\red(M)\cong \red(M')\neq 0$, but this is only possible for simples if
$M\cong M'$ by the exactness of $\red$.  It follows that there is no such $M$.

Now, consider the sheaf $L_{\Bi_{\boldsymbol{\xi},\mathbf{m}}}$ with
$\mathbf{m}\neq \emptyset$.  At any point the singular support of this
module, we have a destabilizing submodule whose dimension is given by
at least 
the number of boxes in $\mathbf{m}$.  Thus, we have produced a
different simple module for each such
$({\boldsymbol{\xi},\mathbf{m}})$ which is killed by $\red$.  This
must be all such modules, since the modules
$\red(L_{\Bi_{\boldsymbol{\xi},\emptyset}})$ are non-zero, and contain
a number of distinct simple modules equal to the number of multi-partitions.
\end{proof}

Given this result, we can apply Corollary \ref{dagger-W}.  This
relates $\dpreO$ and $\dOg$ to weighted KLR algebras for the cyclic
quiver and its Crawley-Boevey quiver.  We'll use $W$ for the
weighted KLR algebra of the cycle with the weight $\kappa$ along every
cyclically oriented edge, $\tilde{T}^\vartheta$ to denote the
weighted KLR algebra for the Crawley-Boevey quiver with the weighting
$\vartheta_i$ along the new edge $e_i$ and $\kappa$ along oriented
edges of the cycle.  Then we'll use $T^\vartheta$, as in
\cite{WebRou}, to denote the steadied quotient of
$\tilde{T}^\vartheta$, that is the quotient by all idempotents
corresponding to unsteady loadings.
\begin{corollary}\label{cor:affine-equivalence}
For $\Gamma$ a cyclic quiver and a generic $\bT$-action, the hypothesis
\hyperlink{dagger}{$(\dagger)$} holds.
  
In particular, for fixed $\la,\mu, \vartheta,$ we have quasi-equivalences \[\tilde{T}^\vartheta_{\mu}\dgmod \cong
D_{\preO}\qquad T^\vartheta_\mu\dgmod\cong D_{\cOg}.\]
\end{corollary}
It follows from \cite[4.14]{WebwKLR} that these equivalences
intertwine convolution with induction of representations of weighted
KLR algebras.  In particular, we have that 
\begin{corollary}\label{cor:action-match}
  The equivalence $T^\vartheta_\mu\dgmod\cong D_{\cOg}$ intertwines
  the algebraic categorical action of \cite{Webmerged} on the LHS and
  geometric categorical action of \cite{Webcatq} on the RHS.  
\end{corollary}

We can also understand the cell filtration in this case;
unfortunately, this is more challenging than the finite
type case, since the variety $\fN^\la_{\mu}$ has strata which are not of the
form $\fN^\la_{\nu}$.  For example,
$\fN^{\omega_0}_{n\delta}\cong \C^{2n}/(\Z/e\Z\wr S_n)$ and counting shows that there are not enough weight spaces to account for
all the strata.  

In particular, for a weight $\mu$ and integer $n$, there is a
stratum closure we denote $\fN^\la_{\nu;n}$ in $\fN^\la_{\mu}$
given by representations of the preprojective algebra isomorphic to a sum
of
\begin{itemize}
\item a simple representation with dimension vector $v_i$ with
  $v_\infty=1$ and $\la=\nu+\sum_{i\in I} v_i\al_i$, 
\item a semi-simple representation with dimension vector given by
  $v_\infty'=0$ and 
  $v_i'\leq n$ for all $i\in I$, and
\item a  trivial representation.
\end{itemize}
To simplify notation, we denote the subcategory of $\cOg$ supported on
$\fN^\la_{\nu;n}$ by $\cOg^{\nu;n}$.  Note that for $n$ sufficiently
large that $\mu'-n\delta \leq \mu$, this subset is independent of $n$,
so we can use $n=\infty$ to indicate this stable range, and thus speak
of $\fN^\la_{\nu;\infty}$ and $\cOg^{\nu;\infty}$.

In this case, we can also refine the isotypic filtration to account
for the presence of a $\glehat$-action on the vector space $K(\cO)$;
there are both isotypic filtrations for $\slehat$ and $\glehat$.  
While in general precisely how to categorify $\glehat$ is a rather
delicate question, for us, it suffices to define certain convolution functors.
As discussed in \cite[4.3]{WebwKLR}, the D-modules or perverse sheaves
on the moduli space of representations of a quiver have a monoidal
structure induced by the moduli space of short exact sequences.  In
our case, we wish to study D-modules on the moduli space of
representations of the Crawley-Boevey quiver as a left and a right
module over the category of D-modules on the moduli of representations
of the original graph $\Gamma$.  

More precisely, given a decomposition $V_i\cong V_i'\oplus V_i''$, we consider the spaces 
\[E_{\Bv';\Bv''}^{0,\Bw}\cong 
E_{\Bv'}^0\oplus E_{\Bv''}^{\Bw} \oplus\bigoplus_{e\in \Omega}
\Hom(V''_{t(e)},V'_{h(e)}) \]
\[E_{\Bv';\Bv''}^{\Bw,0}\cong 
E_{\Bv'}^\Bw\oplus E_{\Bv''}^{\Bw} \oplus\bigoplus_{e\in \Omega}
\Hom(V''_{t(e)},V'_{h(e)}). \]  These are equipped with the obvious action of \[G_{\Bv';\Bv''}=\{g\in G_{\nu}|
g(V_i')=V_i'\}\] and with natural maps 
\[\tikz[very thick,->]{\matrix[row sep=11mm,column sep=8mm,ampersand
    replacement=\&]{
\& \node (a) {$E_{\Bv';\Bv''}^{0,\Bw}/G_{\Bv';\Bv''}$}; \&  \\
\node (b) {$E_{\Bv'}^{0}/G_{\Bv'}$};\& \node (c) {$E_{\Bv}^{\Bw}/G_{\Bv}$}; \&  \node (d) {$E_{\Bv''}^{\Bw}/G_{\Bv''}$}; \\
};
\draw (a)--(b) node[above left,midway]{$\pi_s$};
\draw (a)--(c)  node[left,midway]{$\pi_t$};
\draw (a)--(d)  node[ above right,midway]{$\pi_q$};
}\quad \tikz[very thick,->]{\matrix[row sep=11mm,column sep=8mm,ampersand
    replacement=\&]{
\& \node (a) {$E_{\Bv';\Bv''}^{\Bw,0}/G_{\Bv';\Bv''}$}; \&  \\
\node (b) {$E_{\Bv'}^{\Bw}/G_{\Bv'}$};\& \node (c) {$E_{\Bv}^{\Bw}/G_{\Bv}$}; \&  \node (d) {$E_{\Bv''}^{0}/G_{\Bv''}$}; \\
};
\draw (a)--(b) node[above left,midway]{$\pi_s$};
\draw (a)--(c)  node[left,midway]{$\pi_t$};
\draw (a)--(d)  node[ above right,midway]{$\pi_q$};
}\]
This allows us to construct the convolution of D-modules \[\mathcal{F}_1 \star \mathcal{F}_2:=
(\pi_t)_*(\pi_s^*\mathcal{F}_1\otimes \pi_t^*\mathcal{F}_2)[\sigma]\]
where $\sigma$ is the relative dimension of the map $\pi_s\times
\pi_t$ (i.e. the dimension of the source minus that of the target).
To
avoid confusion, we let $\preO^0$ denote the category $\preO$ of
$E^0_\Bv$ for an (implicit) dimension vector $\Bv$.  
\begin{lemma}\label{lem:gle-preserve}
  For a fixed $\bT$-action $\bT\to H/Z$ where
  \hyperlink{dagger}{$(\dagger)$} holds for $E^\Bw_\Bv$ for all $\Bv$,
  we have that if $\cM\in \dpreOz$ and $\cN\in \dpreO$, then $\cM\star
  \cN$ and $\cN\star \cM$ lie in $\dpreO$.
  If $\red(\cN)=0$, then
  $\red(\cN\star\cM)=0$ as well, so $-\star\cM$ defines an action of
  $\dpreOz$ on $\dOg$, which preserves the subcategory $\cOg^{\mu;\infty}$.
\end{lemma}
\begin{proof}
  Since \hyperlink{dagger}{$(\dagger)$} holds, and $\star$ is exact, we can assume without
  loss of generality that $\cM=L_{\Bi}$ and $\cN=L_{\Bj}$ both arise
  from loadings.  By \cite[4.14]{WebwKLR}, this means that $\cM\star
  \cN\cong L_{\Bi\circ \Bj}$ and $\cN\star \cM=L_{\Bj\circ \Bi}$ where
  $\circ$ is induction of loadings as defined in \cite[\S
  2.4]{WebwKLR}.  These modules lie in $\dpreO$ by Theorem
  \ref{th:pushforward}.

  Similarly, to establish compatibility with $\red$, we note that it
  suffices to assume that $\Bj$ is an unsteady loading by
  \hyperlink{dagger}{$(\dagger)$}.  The induction $\Bj\circ \Bi$ is
  obviously unsteady as well.  This shows that we have the desired
  action of $\dOg$.  
\excise{Given any simple $L$ in
  $\cOg^{\mu;\infty}$, we consider a component of the support of $L$.
  The corresponding component of $pC$ can be analyzed by taking the
  associated K-N character for $u\ll 0$.   By assumption, the largest
  possible quotient which is only on $\Gamma$ of a generic point in
  the support has dimension given by $\Bv'$, the dimension vector of the
  associated multisegment.  Thus, if we take the associated loading
  $L_{\Bj}$, we'll have $L$ as a
  summand, and every point in the support will have a quotient whose
  dimension vector is pointwise $\geq \Bv'$.  Thus, we'll have that $\red(L_{\Bj})$ is also in
  $\cOg^{\mu;\infty}$.

By the same argument, any point in the support of  $L_{\Bj\circ
  \Bi}$ will have a quotient just supported on $\Gamma$ whose
dimension is $\geq \Bv'+|\Bi|$.  Thus, we necessarily have $\red(L_{\Bj\circ
  \Bi})\in \cOg^{\mu;\infty}$, so the same is true of $\red(L\star
L')$ where $L'$ is any summand of $L_{\Bi}$.  }
\end{proof}
The category $\preO$ on $E^0_{\Bv}$ in the case of an $e$-cycle is
simply the category of D-modules with singular support where $x$ is
nilpotent,  or where $\bar x$ is nilpotent depending on whether $\kappa$ is
negative or positive.  Thus, in terms of the conventional support of
these D-modules, we require that the underlying endomorphism (which is
just $x$) is nilpotent, or that the Fourier transform of our D-module
has this condition.  

The Grothendieck group of this category can be
identified with the generic Hall algebra by the function-sheaf
correspondence via the map of \cite[2.12]{SWschur} (defined in more
generality in \cite[4.17]{WebwKLR}).  Thus it can be thought of as the lower half $U^-(\glehat)$  as discussed in Hubery
\cite{Hub} and developed in greater detail in \cite[\S 2.3]{DDF}.  The
subcategory of D-modules with nilpotent singular support (i.e. both
$x$ and $\bar x$ must be nilpotent) is a
proper subcategory whose Grothendieck group naturally corresponds to
$U^-(\slehat)\subset U^-(\glehat)$.  In fact, $ U^-(\glehat)\cong
U^-(\slehat)\otimes \C[z_1,z_2,\dots]$ where the elements $z_i$ are
central, and form half of a Heisenberg subalgebra in the full
$\glehat$.  However, these elements $z_i$ do not correspond to honest
D-modules in the Grothendieck group, and thus are harder to deal
with.  
We can think of the action of $\preO^0$ via
convolution as inducing an action of $U^-(\glehat)$ on $K(\cOg)$, with
the restriction to $U^-(\slehat)$ having a natural categorical
interpretation. 

The action of $U^-(\slehat)$ is part of the categorical $\slehat$-action
defined in \cite[Th. A]{Webcatq}; it is generated by the skyscraper
sheaves $\mathscr{F}_i$ on $E^0_{\al_i}\cong \{0\}/\C^*$. Since this
action can also be interpreted as convolution with Harish-Chandra
bimodules by \cite[3.3]{Webcatq}, they preserve the set of modules
supported on any system of subvarieties closed under convolution with
Nakajima's Hecke correspondence $Z$.  Thus, applying this to $\fM^\la_{\mu;0}$, we have:
\begin{lemma}\label{lem:sle-preserve}
  The categorical $\slehat$-action preserves the subcategory
$\cOg^{\mu;0}$.
\end{lemma}

We let $J_{\mu;\gamma}$ be the intersection of the spaces generated
under $U^-(\glehat)$ by vectors of weight $\geq \mu$ and under
$U^-(\slehat)$ by vectors of weight $\geq \gamma$.  The most
interesting case is $ J_{\nu;\nu-n\delta}$, which is the same as the
vectors that can be obtained from vectors of weight $\geq \nu$ by
elements of the Heisenberg of weight up to $-n\delta$ and by
$U^-(\slehat)$.  For each weight $\mu$, note that
$J_{\nu;\nu-n\delta}$ contains the whole $\mu$ weight space unless
$\nu> \widehat{W}_e\cdot \mu$, that is $\nu$ is greater than every
element of the $ \widehat{W}_e $-orbit of $\mu$.

\begin{theorem}\label{th:affine-cells}
 The special strata of $\fN^\la_\mu$ are exactly those of the form
  $\fN^\la_{\nu;n}$ for $\nu$ dominant with $\la\geq \nu\geq \nu-n\delta\geq \widehat{W}_e\cdot \mu$ and the 2-sided cell filtration on
  $K(\cOg)$ matches the refined isotypic filtration, with
  $K(\cOg^{\nu;n})\cong J_{\nu;\nu-n\delta}$ and more generally,
  $J_{\nu;\gamma}\cong K(\cOg^{\nu;\infty})\cap K(\cOg^{\gamma;0})$.
\end{theorem}
The proof of this theorem is quite long, so we will give it after some
preliminary lemmata.

Note that this shows that these quantizations of $\fM^\la_\mu$ are
{\bf interleaved} in the sense of \cite[\S 6]{BLPWgco}.
This result confirms in the case of an integral character $\chi$ the
conjecture \cite[9.6]{BLet}, which conjectures a more complicated
description of this filtration for a general $\chi$.  This conjecture
was inspired in turn by one of Etingof for rational Cherednik algebras
\cite{EtiSRA}.

A little combinatorics makes it easier to understand the weights $\nu$
appearing in Theorem \ref{th:affine-cells}.
The Lie algebra $\glehat$ is given by a central extension of $\mathfrak{gl}_e[t,t^{-1}]$ by
$\C$, and taking the extension by an element $\partial_i$ for which
the Chevalley generator $E_0$ has weight $1$ and all others have
weight 0.
Consider an $e$-tuple $(t_1,\dots, t_e)$ of integers, and extend $t_i$ to
all integers via the rule $t_{i+e}=t_i-\ell$.  A
level $\ell$
dominant integral weight $\nu$ of $\glehat$ corresponds to an $e$-tuple
$(t_1,\dots, t_e)$ with $\al_i^\vee(\nu)=t_i-t_{i+1}\geq
0$ for all $i\in \Z$ together with the eigenvalue $\eta$ of $\partial$.  We can think of
the $t_i$'s as specifying a cylindrical partition.
\begin{definition}
  The transpose of  $(t_1,\dots, t_e),\eta$ is the sequence
  $(s_1,\dots s_\ell),\phi$ defined by letting $s_k$ be the largest integer
  such that $t_{s_k}\geq k$, with $\phi=-\eta$.  
\end{definition}
We
  have $s_{k}\geq s_{k+1}$, and $s_{k+\ell}=s_k-e$, so this is a
  level $e$ weight for $\sllhat$.  This is one manifestation of
  rank-level duality.

In these terms, $\al_i$ gives the $e$-tuple $(0,\dots, 1,-1,\dots 0)$
with $\eta=\delta_{i,e}$ (so $\eta$ functions as an ``odometer'' that
notices the difference between $\delta$ and $0$).  Another weight
$(t_1',\dots,t_e')$ and $\eta'$ is below $(t_1,\dots, t_e),\eta$ in
the root order if there are $v_i\geq 0$ such that
$t_i'=t_i-v_i+v_{i-1}$ and $\eta'=\eta- v_e$.  One can think of this
as dominance order for cylindrical partitions: we can think of
subtracting $\al_i$ as moving a box to the right, with $\al_0$
subtracting from $\eta$ to remind the user that the box used has gone
around the cylinder.  Just as transpose reverses dominance order for
usual partitions, transpose will reverse this root order as well.

In the category $\cOg^0$, there is a simple given by the vector space
$\C$ considered as a sheaf on $E^0_0=\{0\}$.  Convolution on the left
and right with this object gives the identity functor, so we call it
{\bf neutral}.  We define a {\bf non-neutral} object to be a
non-trivial object in $\cOg^0$ which does not contain the neutral
object as a summand.  In particular, a simple non-neutral object must have
support that is disjoint from $E^0_0$.

\begin{lemma}\label{left-right-induction}
  Assume $\mu$ is dominant.  The simples in $\preO$ that have a closed free orbit in their
  support are precisely those which are not summands of a left or right
  convolution with a non-neutral object in $\cOg^0$. 
\end{lemma}
\begin{proof}
Note that an orbit is free if and only if its automorphism group 
as a preprojective module is given by the scalars $\C^*$ (these are
not included in $G_\Bv$, since they also act nontrivially on $W$), and closed
if and only if the underlying object is semi-simple.  Thus, the orbit
will be free and closed if and only if the object is simple.

If we perform a left convolution with
  non-neutral object  in $\preO^0$, then every point of the singular
  support of the resulting sheaf $\cM$ has a destabilizing submodule considered as a module over the
 preprojective algebra of the Crawley-Boevey quiver by Proposition \ref{O-is-T-affine}.
 Symmetrically, if we perform a right convolution, it has a
 non-trivial quotient just supported on $\Gamma$
 (so it is unstable for the opposite stability condition).  In either
 case, no point in the support corresponds to a simple pre-projective module.

Consider a component $C$ of the relative precore $pC$ which does not contain a simple
representation of the preprojective algebra.  If every point in $C$ is
unstable, there is a semi-simple
module $L_{C}$ in $\preO$ which is supported on $C$ and components
below it in the KN order, and by construction, this is a summand of
$L_{\Bi_{\boldsymbol{\xi}, \mathbf{m}}}\cong L_{\Bi_{\emptyset,\mathbf{m}}}\star L_{\Bi_{\boldsymbol{\xi}, \emptyset}}$ with
$\mathbf{m}\neq 0$.  Thus, it is
a summand of a convolution and has no closed free orbits in its
support.  If the generic point of $C$ is stable, then we can apply the
same argument, but using the KN decomposition is $u\ll 0$, that is,
for the opposite stability condition.  This shows 
that there is a simple $L_{C}\subset L_{\Bi_{\boldsymbol{\xi}, \emptyset}} \star  L_{\Bi_{\emptyset,
    \mathbf{m}}}$, where now $\mathbf{m}$ comes from the Jordan type
of $x$ acting on a maximal quotient of $V$.  Note that this assigns a
unique simple in $\preO$ to each component, since no modules from the
first list could coincide with ones on the second: the latter always
has an component which is generically stable in their support, and the
former never do. 

As proven in
  Proposition \ref{O-is-T-affine}, there is a bijection between
  Lagrangian components and simples in $\preO$.  Furthermore, if the
  component has a free closed orbit, i.e. its generic point is 
  semi-simple, then its corresponding simple must have a free and
  closed orbit in its support.  Since we've already shown that the set
  of simples without a free and closed orbit in their support is at least as large as
  the set of components without a free and closed orbit, by the
  pigeonhole principle, the converse must hold: if the support of a
  simple has a free and closed orbit, so must the corresponding
  component (even though the singular support need not be irreducible.
\end{proof}
As mentioned before, by \cite[4.14]{WebwKLR}, the equivalence of
Corollary \ref{cor:affine-equivalence} intertwines convolution with
induction of modules over weighted KLR algebras.  Thus, a simple
module in $\preO$ is a left or right convolution with a non-neutral
object of $\preO^0$ if and only if the corresponding module over
$\tilde{T}^\bla$ is an induction (in the sense of \cite[\S 2.4]{WebwKLR}) of a
module over $\tilde{T}^\bla$ tensored on the left or right with a
module over $W$.

\nc{\ep}{\epsilon}

We'll use both in the proof below and later in the paper that there is
a crystal structure on components of $(\fM^\la_\mu)^+$.  The Kashiwara
operators are defined as follows:
\begin{itemize}
\item  $\tilde{e}_i(C)$ is the unique component such that a generic
  point in $\tilde{e}_i(C)$ is a submodule of some generic point of
  $C$ with 1-dimensional cokernel supported on $i$.
\item $\tilde{f}_i(C)$ is the unique component such that a generic point in
  $\tilde{f}_i(C)$ contains some generic point of $C$ as a submodule
  with 1-dimensional cokernel supported on $i$.
\end{itemize}
 The proof that these give a crystal in Kashiwara-Saito \cite[5.2.6]{KS97} carries through without changes;
the tensor product case of this result is observed without proof by Nakajima
in \cite[4.3]{Nak01}.  We let $\ep_i(C)$ for a component be the
dimension of a maximal quotient of a generic point of $C$ supported
on the single vertex $i$.

The argument of \cite [6.4]{TS} shows that:
\begin{theorem}\label{crystal-match}
  Under the bijection $\boldsymbol{\xi}\mapsto C_{\boldsymbol{\xi}}$,
  the $\vartheta$-weight combinatorial crystal structure on charged
  partitions of \cite[5.5]{WebRou} is matched with this geometric structure.
\end{theorem}

For any component, we have a string parameterization.  This is the
sequence of numbers \[a_1=\ep_1(C) \quad
a_2=\ep_2(\tilde{e}_1^{a_1}C)\quad
a_k=\ep_k(\tilde{e}_{k-1}^{a_{k-1}}\cdots \tilde{e}_1^{a_1}C).\]
Since $\ep_i(C)\leq v_i$, we must have $a_k(C)=0$ for $k\gg 0$, and
$\tilde{e}_{k-1}^{a_{k-1}}\cdots \tilde{e}_1^{a_1}C$ stabilizes at
some component $C^+$. 

We can construct all the simples in $\dOg$ by
starting with the simples that correspond (under the KN order) to the
components with $\ep_i(C)=0$ for all $i$ letting
$H_{C}$ denote the corresponding simple. For an arbitrary component
with $a_k$ and $C^+$ as above, we consider
$H_C'=\EuScript{F}_{i_1}^{(a_1)}\cdots \EuScript{F}_{i_k}^{(a_k)}
H_{C^+}$.  This sheaf has multiplicity 1 along $C$ in its singular
support, so it has a unique simple summand $H_C$ which contains $C$ in its
singular support. Every other component $D$ in the support of $H_C'$,
and thus of $H_C$, has string
parameterization which is longer in lexicographic order, or has the
same string parameterization and $D^+< C^+$ in the KN order.  
Thus, $C\mapsto H_C$ gives a new bijection between components and
sheaves.  
\begin{lemma} \label{lem:O-plus}
  We have $H_C\in \cO^{\gamma;0}$ if and only if $\gamma\leq \wt(C^+)$.
\end{lemma}
\begin{proof}
If $\ep_i(C)=0$ for all $i$, then $C=C^+$.  
If we consider the generic point of
$C$ as a representation, and mod out by the submodule generated by the
Crawley-Boevey vertex, then we obtain a representation of the
preprojective algebra of the cycle with no quotients supported on a
single vertex, where $x$ acts nilpotently.  Since $x$ and $\bar x$
commute (by the preprojective relation), this representation
decomposes into summands according to the generalized eigenspaces of
$\bar x$.  If generically the generalized $0$-eigenspace of $\bar x$
is non-trivial, then we will have a quotient where $x$ and $\bar x$ act
trivially, which we have assumed is not the case.  Therefore, we only
have $H_C\in \cO^{\gamma;0}$ if and only if $\gamma\leq \wt(C)$, as desired.

For a general $C$, we have $H_C'\in \cO^{\wt(C^+);0}$ by construction, since $\EuScript{F}_{i}$
preserves this subcategory. Thus, the same is true of $H_C$.  On the
other hand, if $\gamma\nleq \wt(C^+)$, we do not have $H_C\in
\cO^{\gamma;0}$, since the semi-simplification of the generic point in
$C$ contains a submodule isomorphic to the generic point of $C^+$.
This generic point 
has no trivial summands, as argued above, so the support is outside
$\fN^\la_{\gamma;0}$.  
\end{proof}

\begin{proof}[Proof of Theorem \ref{th:affine-cells}]
  First, note that we must have
  $J_{\nu;\gamma}\subset K(\cOg^{\nu;\infty}\cap\cOg^{\gamma;0})$
  since $K(\cOg^{\nu;\infty})$ is closed under the action of $U^-(\glehat)$
  by Lemma \ref{lem:gle-preserve}, and $K(\cOg^{\gamma;0})$ under the
  action of $\slehat$ by Lemma \ref{lem:sle-preserve}.  

On the other hand, assume $\cM\in
\cOg^{\nu;\infty}\cap\cOg^{\gamma;0}$ is a simple module supported on
$\fM^\la_\mu$.  We wish to show by induction that $[\cM]\in
J_{\nu;\gamma}$; we assume this holds for any weight $>\mu$.  
By Lemma   \ref{left-right-induction}, we either have
that a generic point in the support of $\cM$ is a simple module or
this generic point has a proper quotient supported on $\Gamma$.  In the former case, we must have
$\mu\geq \nu$ and $\mu\geq \gamma$, so obviously $[\cM]\in
J_{\nu;\gamma}$.  In the latter, we have that $\cM$ is a summand of
convolution with a non-neutral object by Lemma
\ref{lem:gle-preserve}, and by construction, the object we induce with
lives in $\cOg^{\nu;\infty}$ as well.
Thus, by induction $[\cM]$ is in the subspace generated over by
elements of weight $\geq \nu$ over $U^-(\glehat)$.

By Lemma \ref{lem:O-plus}, we must have $\cM=H_C$ for some
component $C$ with $\wt(C^+)\geq \gamma$.   Thus, $H'_C$ is generated from a vector of weight $\wt(C^+)$ by
the action of functors $\EuScript{F}_{i}$ so $[H_C']$ is obtained from
a vector of weight $\geq \gamma$ by elements of $U^-(\slehat)$.  The
same is true for $[H_D']$ for $D$ every other component appearing in
the support of $H_C'$.  Since $[\cM]$ is a linear  combination of the
classes $[H_D']$, we ultimately find that $[\cM]\in J_{\nu;\gamma}$
whenever $\cM\in
\cOg^{\nu;\infty}\cap\cOg^{\gamma;0}$.

We also need to show that if $\supp\cM\subset \fN^\la_{\nu;n}$ and it
is not contained in any smaller stratum of this type, then
the smallest stratum containing the support
of $\cM$ is precisely $\fN^\la_{\nu;n}$.  Since $H_C$ and $H_{C^+}$
have the same minimal stratum containing them, we can assume that
$\cM=H_{C}$ with $C=C^+$.
 
Applying Fourier transform if necessary, we may assume that
$\kappa>0$.  Fix an integer $n$ and $g, \epsilon>0$ chosen so that  $g\gg \ck \gg \epsilon$. Consider the D-module $Y_{\Bi}$ where $\Bi$ is the
loading putting a dot labeled $i$ at $g+\epsilon i,\dots, ng+\epsilon
i$ for $i=1,\dots, e$; that is, this loading has 
$n$ subsets separated by the ``long'' distance $g$, with each subset
consisting of points with total dimension vector $\delta$ which are
tightly clustered.
The space $C_i$ given by the first $i$ subsets is invariant in the
usual sense and thus gives an invariant flag the representation for each point in
$X_{\Bi}$ with $\dim C_i/C_{i-1}=\delta$.  On $C_i/C_{i-1}$, the map
going around the cycle is an endomorphism of a 1-dimensional vector
space, and thus a scalar, which is the same at all
points of the cycle.  Thus, we have a natural map $X_{\Bi}\to\C^n$
sending a representation with flag to the $n$-tuple of scalars
associated to $C_i/C_{i-1}$.  If we let $ \C_\circ^n=\{(x_1,\dots, x_n) \in \C^n\mid x_i\neq x_j\neq
0 \text{ for all } i\neq j\}$ then  we have an open inclusion $\C_\circ^n\to
X_{\Bi}$ sending $(x_1,\dots, x_n)$ to the representation where the
spaces are all $\C^n$ equipped with the standard flag 
 and the map along one edge is $\operatorname{diag}(x_1,\dots, x_n)$ and along all
the others is the identity.  
We have a Cartesian diagram \[
\tikz[very thick, ->]{
\matrix[row sep=12mm,column sep=17mm,ampersand replacement=\&]{
\node (a) {$ \C_\circ^n$}; \& \node (c) {$X_{\Bi}$};\\
\node (b) {$\C_\circ^n/S_n$}; \& \node (d) {$E_{n\delta}/G_{n\delta}$};\\
};
\draw (a) -- (c) node[above,midway]{};
\draw (b) --(d) node[below,midway]{};
\draw (c)--(d) node[right,midway]{};
\draw (a) -- (b) node[right,midway]{};
}
\]
since two points in $\C_\circ^n$ will give isomorphic representations
if and only if they differ by a permutation.  Thus, to each representation of
$S_n$ (and thus to a partition $\xi$ of $n$), we have an induced local system on $\C_\circ^n/S_n$, and thus an
intermediate extension D-module $Z_\xi'$ on $E_{n\delta}/G_{n\delta}$.
The Cartesian diagram above shows that these are all summands of
$Y_\Bi$ and lie in $\preO^0$.  Since we have been using $\kappa<0$ for
most of the paper, let use switch back to this case, by replacing
$Z_\xi'$ by their Fourier transform $Z_\xi$; that is, we switch the
role of $x$ and $\bar x$.  In particular, the map to $\C_\circ^n/S_n$
is now given by the spectrum of $\bar x$, and $x$ must be  nilpotent.

The convolutions $H_D\star Z_{\xi}$ for $D$ a component with a free
and closed orbit in it give a collection of modules.  Each of these
has support on the unique component $D_{;n}$ whose generic point is an
extension of a generic point in $D$ by a module $U$ which is an
extension of $n$ simple modules over the preprojective algebra with
$\bar x$ acting on each with a different non-zero eigenvalue.  Since
the restriction of $H_D$ to $D$ gives a trivial local system, the
restriction of $H_D\star Z_{\xi}$ to this component gives the pullback
of the local system associated to the partition $\xi$ under the
birational map $D_{;n}\dashrightarrow \C_\circ^n/S_n$ sending a
generic point to the spectrum of $\bar x$ on $U$.

We let $H_{D,\xi}\subset H_D\star Z_{\xi}$ be the unique summand which
contains this component in its support.  No two of these with the same
$D$ are isomorphic, since they give different local systems on
$D_{;n}$.  No two with different $D$ coincide, as the Kirwan-Ness order with $u\ll 0$
shows.

Thus, we have a set of simples whose size is the number of components
with $C=C^+$, and all of which contain such a component in their
support.  By the pigeonhole principle, this must be a complete list of
the modules $H_C$ with $C=C^+$, and the smallest stratum which
contains their support is indeed $\fM^\la_{\nu;n}$ where $\nu=\wt(D)$
and $n$ is the number of boxes in $\xi$.

Finally, we need to show that if $\la\geq \nu\geq \nu-n\delta\geq \widehat{W}_e\cdot
\mu$, and $\nu$ is dominant, there is at least one simple in the
2-sided cell of $\fM^\la_{\nu;n}$, and thus to check that there is at
least one component of $(\fM^\la_{\mu})^+$ which intersects this
stratum non-trivially.  Our condition on weights precisely guarantees
that this component is non-empty, and a non-empty stratum will always
have points with limits under a Hamiltonian action.  
\end{proof}

One important consideration is how this filtration can be realized
algebraically.   Let $\mathcal{J}_{\mu;\gamma}$ be the intersection of the
  subcategories generated by objects $T^\vartheta\dgmod$ of weight
  $\geq \mu$ under induction
  with projective modules over the weighted KLR algebra $W$ of $\Gamma$,
  and by objects of weight $\geq \gamma$ under the action of the
  categorical $\slehat$ action.  These latter can be thought of as
  induction with loadings where the gaps between points are
  $>|\kappa|$; we call these loadings {\bf Hecke}. It's manifest
  from the match of induction and convolution  (\cite[4.14]{WebwKLR}) that:
\begin{proposition}
The quasi-equivalence $T^\vartheta\dgmod\cong \cOg$ induces
quasi-equivalences $\mathcal{J}_{\mu;\gamma}\cong
\cOg^{\mu;\infty}\cap \cOg^{\gamma;0}$. \qed
\end{proposition}

In \cite{WebRou}, we introduced {\bf change-of-charge} functors which
are weighted analogues of the R-matrix functors from \cite[\S 6]{Webmerged};
they have a very similar geometric definition.  These can be assembled
into an strong action of the affine braid group $\widehat{B}_\ell$.  

By a retread of the argument in Proposition \ref{tensor-shuffling}, we
see that:
\begin{corollary}\label{affine-shuffling}
  The quasi-equivalence  $T^\vartheta\dgmod\cong D_{\cOg}$ intertwines the
  change-of-charge functor $\mathscr{B}^{\vartheta,\vartheta'}\Lotimes
  -$ with the shuffling functor $\mathscr{S}^{\vartheta,\vartheta'}$.\qed
\end{corollary}

\subsection{Koszul duality}
\label{sec:dual-twist-shuffl}

The ring $T^\vartheta$ is Koszul and its Koszul dual is another algebra of the same
type. In order to state this precisely,  let us briefly describe the combinatorics underlying this
duality.  Fix a positive integer $w$ and a $\ell\times e$ matrix of integers $U=\{u_{ij}\}$, and
let $s_i=\sum_{j=1}^e u_{ij}$ and $t_j=\sum_{i=1}^\ell u_{ij}$.   Associated
to the $g$th column of $U$, we have a charged $e$-core partition.  An
$e$-core is, by definition, a partition which has no removable ribbons of length $e$.  

Our desired $e$-core can be
characterized as the
unique charged partition to which is possible to add one $e$-ribbon
with contents $e(u_{gk}-1)+k,e(u_{gk}-1)+k+1,\dots, eu_{gk}+k-1$ for
each $k=1,\dots, e$, and no others.
Let $v_i$ be
the unique integer such that $v_i-w$ is the
total
number of boxes of residue $i$ in all these partitions.  We
wish to consider the affine quiver variety for the highest weight
$\la:=\sum_{i}\omega_{s_i}$ with the dimension vector $v_i$; that is,
with weight $\mu:=\la-\sum v_i\al_i$.  

As discussed earlier, the associated weight $\mu$ is dominant if
$t_1\geq t_2\geq \cdots \geq t_{e}\geq t_1-\ell$.  Similarly, we call
$s_i$ {\bf JMMO} if $s_1\geq s_2\geq \cdots \geq s_{\ell}\geq
s_1-e$.  Note that just as the dominant weights are a fundamental
region for the action of the affine Weyl group $\widehat{W}_e$, we
have that the JMMO weights are a fundamental region for the level $e$ action of
$\widehat{W}_\ell$ on $\Z^\ell$.  In fact, these lift to a
$\widehat{W}_e\times \widehat{W}_\ell$-action on the matrices $U$ by applying
the level 1 action of the $\widehat{W}_e$ to each row, and the level 1
action of $\widehat{W}_\ell$ to each column.

In order to describe a category $\cO$, we also need to consider a
$\bT$-action on the corresponding quiver variety $\fM^\la_\mu$.  This is equivalent to a
weighting of the Crawley-Boevey graph for $\la$, where we enumerate
the new edges so that $e_i$ connects to the node corresponding to the
residue of $s_i\pmod e$.  
We let $\vartheta_{U}$ be  the weighting where we give each edge of the oriented cycle
weight $\ell$, and the new edge $e_i$ the weight $s_i\ell+ie$.  Up to
a multiple, this is the Uglov weighting attached $\Bs$ by 
\cite[2.6]{WebRou}.

\begin{proposition}\label{all-theta}
  Every generic cocharacter $\bT\to H/Z$ has a category $\cO$ for
  $\fM^\la_\mu$ which 
  coincides with that of $\pm \vartheta_{U}$ for some $U$.
\end{proposition}
\begin{proof}
We can  scale the cocharacter until the weight of the cycle is $\ell e$, and
  choose a lift where all old edges have weight $\ell$.  Now, for each edge $e_i$ with weight
  $\vartheta$, we write $\vartheta_i=s_i\ell+\vartheta_i'$ with
  $0\leq \vartheta_i'< \ell e$ and $s_i\equiv r_i\pmod e$.  We can reindex the
  edges so that $\vartheta_1'<\dots < \vartheta_\ell'$; by genericity,
  these are all distinct.  

Having made this reindexing, the action corresponding to $s_i$ is the one
we desire.  To see that it is equivalent to our original action, we
consider the action of $\bT$ on the tangent space of a $\bT$-fixed
point.  Applying \cite[2.11]{NY04}, the weights that appear in these
tangent spaces are of the form
$(\vartheta_i-r_i\ell)-(\vartheta_j-r_j\ell)+ge\ell$ for $g\in \Z$; since 
\[(\vartheta_i'+(s_i-r_i)\ell)-(\vartheta_j'+(s_j-r_j)\ell)+ge\ell.\]
On the other hand for our standard action, the same vector will have
weight
\[(ie+(s_i-r_i)\ell)-(je+(s_j-r_j)\ell)+ge\ell.\]

If $s_i-r_i+ge> s_j-r_j$, then this weight is positive in both cases
(since $\vartheta_i'-\vartheta_j'<\ell e$ and similarly with $(i-j)e
<\ell e$) and similarly with the opposite inequality.  If $s_i-r_i+ge=
s_j-r_j$, then $\vartheta_i'<\vartheta_j'$ if and only if $i<j$ (by
definition).  This shows that the relative cores of the two actions
coincide, so the corresponding category $\cO$'s do as well.
\end{proof}

Let $U^!$ be the transpose of $U$, and $\mu^!=\la(U^!),\la^!=\mu(U^!)$ the weights attached
to $U^!$ and $w$ by the recipe above.  We have switched the roles of
$\mu$ and $\la$ here, since changing $\mu$ and holding $\la$ constant
will change $\mu^!$ and vice versa.  
More generally, we have an order
reversing bijection from dominant weights in the interval
\[[\la,\widehat{W}_e\cdot \mu]:=\{\nu\mid \la\geq \nu\geq w\cdot \mu
\text{ for all }w\in \widehat{W}_e\}
\] to the dominant weights in $[\mu^!,\widehat{W}_\ell\cdot \la^!]$, which we denote by
$\nu\mapsto \nu^!$, given by taking the tranpose of the corresponding
cylindrical partition and sending $\eta\mapsto w-\eta$.  

\begin{theorem}[\mbox{\cite[Th. C]{WebRou}}]
  The graded abelian categories $T^{\pm\vartheta_{U}}_\mu\mmod$ and
  $T^{\mp\vartheta_{U^!}}_{\la^!}\mmod$ are Koszul dual. 
\end{theorem}
Note that Theorem \ref{maltese-koszul} gives a new proof that these
algebras are Koszul.
Using Theorem \ref{O-is-T-affine}, we can give a 
reformulation of this result.  We let $\cO^{\la;\vartheta}_\mu$ be the
category $\cOg$ for Nakajima's stability condition on the quiver
variety $\fM^\la_\mu$ for the $\bT$ action coming from $\vartheta$.

\begin{corollary}\label{Koszul-dual}
  We have an equivalence of abelian categories
  $\cO^{\la;\vartheta_{U}}_\mu\cong T^{-\vartheta_{U^!}}_{\la^!}\umod$.
  In particular, categories $\cO^{\la;\vartheta_{U}}_\mu$ and
  $\cO^{\mu^!;-\vartheta_{U^!}}_{\la^!}$ have graded lifts which are
  Koszul dual.  
\end{corollary}

By Proposition \ref{all-theta}, this describes the Koszul dual of
every integral geometric category $\cO$ for an affine type A quiver
variety.  Of course, this functor is quite inexplicit; it would be
very interesting to define a more concrete functor between these
categories.

One thing we do know about this equivalence is that it categorifies
rank-level duality.  
That is,  we can consider the Grothendieck group
of $\mathbf{O}:=\bigoplus \cO^{\la;\vartheta_{U}}_\mu$ for $U$ varying across all
integer matrices $U$ and all $w$, and $\mathbf{O}^!:=\bigoplus \cO^{\la^!;-\vartheta_{U^!}}_{\mu^!}$ the same category
with the $\bT$ action negated and $e$ and $\ell$ switching roles.  Koszul duality induces an
isomorphism of the Grothendieck groups
$K(\mathbf{O})=K(\mathbf{O}^!)$. 
Note that some care about signs is needed here, as discussed in
\cite[10.28]{BLPWgco}; thus, this isomorphism should not send standard
classes to standard classes, but rather to signed standards.  We can
capture this sign by identifying both Grothendieck groups with  $\bigwedge\nolimits^{\infty/2}(\C^\ell\otimes
  \C^e)[u,u^{-1}]$, sending the standards from $\mathbf{O}$ to the
  pure wedges where we order basis vectors in $\C^\ell\otimes
  \C^e$ lexicographically considering the left tensor factor first
  (for the obvious order on the basis of $\C^\ell$ and $\C^e$), and
  the standards from $\mathbf{O}^!$ with the pure wedges where we
  order considering the right factor first.  
The set of simples in $\cO^{\la;\theta_U}_\mu$ are indexed by
  $\ell$-multipartitions where the number of boxes of content $i$ is $v_i$; these
  can be interpreted as $\ell$-strand abaci.  As
  proven in \cite[6.4]{WebRou}, the
  Koszul duality bijection is given by  cutting the abacus into
  $\ell\times e$ rectangles and flipping, as in the picture below:
  
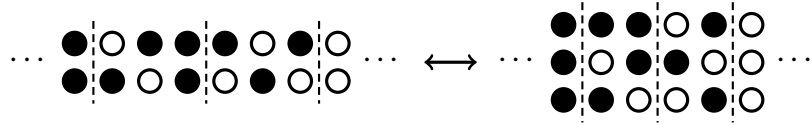
\begin{equation}
\label{eq:flip}\begin{tikzpicture}
\node (a) at (-3,0){
    \begin{tikzpicture}[very thick,scale=.5]
      \draw[thick, densely dashed] (-.5,1.6) -- (-.5,-.6);
      \draw[thick, densely dashed] (-3.5,1.6) -- (-3.5,-.6);
      \draw[thick, densely dashed] (2.5,1.6) -- (2.5,-.6);
      \node[circle, draw, inner sep=3pt] at (0,0){}; \node[circle,
      draw, inner sep=3pt,fill=black] at (0,1) {}; \node[circle, draw, inner
      sep=3pt,fill=black] at (1,0) {}; \node[circle, draw, inner sep=3pt] at
      (1,1) {}; \node[circle, draw, inner sep=3pt] at (2,0) {};
      \node[circle, draw, inner sep=3pt,fill=black] at (2,1) {}; \node[circle,
      draw, inner sep=3pt] at (3,0) {}; \node[circle, draw, inner
      sep=3pt] at (3,1) {}; \node[circle, draw, inner sep=3pt,fill=black] at
      (-1,0) {}; \node[circle, draw, inner sep=3pt,fill=black] at (-1,1) {};
      \node[circle, draw, inner sep=3pt] at (-2,0) {}; \node[circle,
      draw, inner sep=3pt,fill=black] at (-2,1) {}; \node[circle, draw, inner
      sep=3pt,fill=black] at (-3,0) {}; \node[circle, draw, inner sep=3pt] at
      (-3,1) {}; \node[circle, draw, inner sep=3pt,fill=black] at (-4,0) {};
      \node[circle, draw, inner sep=3pt,fill=black] at (-4,1) {}; \node at
      (-5.2,.5) {$\cdots$}; \node at (4.2,.5) {$\cdots$};
    \end{tikzpicture}};
\node (b) at (3,0){    \begin{tikzpicture}[very thick,scale=.5]
      \draw[thick, densely dashed] (-.5,1.6) -- (-.5,-1.6);
      \draw[thick, densely dashed] (-2.5,1.6) -- (-2.5,-1.6);
      \draw[thick, densely dashed] (1.5,1.6) -- (1.5,-1.6);
      \node[circle, draw, inner sep=3pt,fill=black] at (0,0){}; 
      \node[circle, draw, inner sep=3pt] at (0,1) {}; 
      \node[circle, draw, inner sep=3pt] at (1,0) {}; 
      \node[circle, draw, inner sep=3pt,fill=black] at (1,1) {}; 
      \node[circle, draw, inner sep=3pt] at (2,0) {};
      \node[circle, draw, inner sep=3pt] at (2,1) {}; 
      \node[circle, draw, inner sep=3pt] at (-1,-1) {}; 
      \node[circle, draw, inner sep=3pt,fill=black] at (-2,-1) {}; 
      \node[circle, draw, inner sep=3pt,fill=black] at (-1,0) {}; 
      \node[circle, draw, inner sep=3pt,fill=black] at (-1,1) {};
      \node[circle, draw, inner sep=3pt] at (-2,0) {}; 
      \node[circle, draw, inner sep=3pt,fill=black] at (-2,1) {}; 
      \node[circle, draw, inner sep=3pt,fill=black] at (-3,0) {}; 
      \node[circle, draw, inner sep=3pt,fill=black] at (-3,1) {}; 
      \node[circle, draw, inner sep=3pt] at (0,-1) {};
      \node[circle, draw, inner sep=3pt,fill=black] at (1,-1) {}; 
      \node[circle, draw, inner sep=3pt,fill=black] at (-3,-1) {}; 
      \node[circle, draw, inner sep=3pt] at (2,-1) {}; 
      \node at (-4.2,0) {$\cdots$}; 
      \node at (3.2,0) {$\cdots$};
    \end{tikzpicture}};
\draw[<->,very thick] (a)--(b);
\end{tikzpicture}
\end{equation}

Using this isomorphism,
  we can endow both spaces with simultaneous $\slehat$ and $\sllhat$ actions.
It follows from Corollary
\ref{cor:action-match} and 
\cite[5.8
]{WebRou} that:
\begin{lemma}
  We have an isomorphism of $\slehat\times \sllhat$-modules
  $K(\mathbf{O})\cong \bigwedge\nolimits^{\infty/2}(\C^\ell\otimes
  \C^e)[u,u^{-1}]$ of the Grothendieck group with a semi-infinite
  wedge space, identifying pure wedges with standard modules.
\end{lemma}

The classes of the
simples or projectives in $\mathbf{O}$ or $\mathbf{O}^!$ form a perfect basis in the sense of \cite[\S
5]{BeKa} for $\slehat$ by \cite[5.20]{CR04}, and since these are
exchanged under Koszul duality they are simultaneously a perfect basis
for $\slehat\times \sllhat$.  Thus, the basis itself has a crystal structure for
this algebra.   Let this crystal be denoted $\mathbf{B}$.   

From the bijection above, we can describe this
crystal structure explicitly using \cite[5.6]{WebRou}; this is, in
fact, a straightforward consequence of the theory of ``highest weight
categorifications'' due to Losev \cite{LoHWCI}.  We identify each rectangle in the
picture above with a basis of $\C^\ell\otimes \C^e$ (with $\C^\ell$
running vertically, and $\C^e$ horizontally).  
A Kashiwara operator $\tilde{f}_i$
for $\slehat$ pushes a bead whose $x$-coordinate is $i\mod e$ one
step to the right. As usual, we choose the bead by putting open parenthesis as each
spot where we can do this, and a close parenthesis over each spot
where it can be undone (where we can push a bead with $x$-coordinate
is $i+1\mod e$ to the left), using height as a tie-breaker (higher
beads are ``further right'').  We then  push the bead corresponding to the
leftmost uncanceled parenthesis.  The $\tilde{f}_j$ for $\sllhat$ are
given by the same procedure, but push up instead of right, where a
bead which is pushed off the top appears at the bottom $e$ steps to
the right.  

We would like to understand how various geometric and
algebraic constructions match under this Koszul duality. In
particular, the categories $T^{-\vartheta_{U^!}}_{\la^!}\umod$ are interesting not
just on their own, but because they carry an interesting action of
change-of-charge functors.  These correspond to changing $
\vartheta_{U^!}$, while keeping $\la^!$ constant; according to the
prescription given above, this would be accomplished by acting on
$\mu$ by the affine Weyl group $ \widehat{W}_e$ of $\slehat$, but as we already noted
in Section \ref{sec:twisting-functors}, acting on this weight and
keeping the stability condition constant has the same effect as
changing stability condition while keeping the dimension vector
constant.  Thus, we let $\cOg^w$ denote the category $\cO$ for the
stability condition $w^{-1}\cdot \xi$, and let $\red^w$ denote the
corresponding quotient functor.

Thus, the change-of-charge functors relate category $\cO$'s which are
actually quotients of a single common category $\preO$ for the $G_\mu$-action on $E^\la_\mu$ and the $\C^*$-action corresponding to
$\vartheta_U$.  Furthermore, it is relatively easy to understand the
structure of these quotient functors.  First note that:
\begin{lemma}
  If $L$ is a simple object in $\preO$ with $\red^v(L)\neq 0$ and $P$
  is a projective cover of $\red^v(L)$, then
  $\red_!^v(P)$ is a projective cover of $L$.

In particular, if $L$ is a simple whose singular support contains a
closed free $G_\mu$-orbit, then $\red^v(L)$ is non-zero for all
stability conditions, and there is a projective $Q$ in $\preO$ such
that $\red^v(Q)$ is the projective cover of $\red^v(L)$ in the category
$\cO$ for any $v\in \widehat{W}_e$.
\end{lemma}
\begin{proof}
  Since $\red$ is exact, $\red_!(P)$ is projective, and obviously has
  a natural surjective map $ \red_!(P)\to L$. Since
  $\End(\red_!(P))\cong \End(P)$ is local, this projective is
  indecomposable and thus a projective cover of $L$. 

 Now, turn to the second part; by the first part, we have $Q\cong
 \red_!^v(P)$, so $P\cong \red^v(Q)$.
\end{proof}
Thus, one  common tie between these different categories $\cO$ is the
collection of simples whose singular supports contain a closed free orbit.  
We wish to understand how these simples in $\preO$ match up with
simple modules over $T^{\vartheta_{U^!}}_{\la^!}$.  As a first step,
we consider which projectives of $T^{\vartheta_{U}}_\mu$ they match
with.  

More generally, we should understand how the cell filtration behaves
under Koszul duality.  We expect that the quiver varieties for dual
data $\fM^\la_\mu$ and $\fM^{\mu^!}_{\la^!}$ will have an order
reversing bijection between their special strata.

\begin{proposition}\label{prop:strata-bijection}
  The map $\fN^{\la}_{\nu;n}\mapsto \fN^{\mu^!}_{\nu^!+n\delta;n}$ is
  an order reversing bijection between special strata.  Koszul duality
  induces a bijection between simple modules which sends modules with
  support contained in $\fN^{\nu}_{\mu;n}$ but not any smaller stratum
  to those supported in $\fN^{\nu^!+n\delta}_{\la^!;n}$ but not any
  smaller stratum.
\end{proposition}
\begin{proof}
  The statement about strata follows immediately from Theorem
  \ref{th:affine-cells} and the fact that $\nu\mapsto \nu^!$ is an
  order reversing bijection and $(\nu-n\delta)^!=\nu^!+n\delta$.  

Taking the sum of the subspaces $J_{\nu-n\delta,\nu}$ over all $U$, we obtain 
a subspace $J_{\nu-n\delta,\nu}\subset
\bigwedge\nolimits^{\infty/2}(\C^\ell\otimes \C^e)[u,u^{-1}]$. This
subspace is invariant under $\slehat$ by definition and under $\sllhat$ for all
$\nu,\gamma$, since the central extension of $\mathfrak{sl}_\ell[t,t^{-1}]$ and
$\glehat$ commute, and the element $\partial\in \sllhat$ commutes past
$\glehat$ by inducing a derivation.  

Thus, the 2-sided cell is just defined by which
$\sllhat\times \slehat$ highest weight simples lie in it.  All such
simples in the 2-sided cell $\fN^{\la}_{\nu;n}$ are supported on
$\fN^{\nu^!}_{\nu-n\delta}$, and they are precisely the $\slehat$ highest weight
objects supported on this variety, since all the $\slehat$ highest
weight vectors in this weight space are obtained from the unique
highest weight vector of weight $\nu$ by the action of the Heisenberg.

Koszul duality sends these vectors to the $\sllhat$ highest weight
objects of weight $\nu^!$, that is, the $ \sllhat\times
\slehat$ highest weight objects in the 2-sided cell for
$\fN^{\mu^!}_{\nu^!+n\delta;n}$.  This completes the proof.
\end{proof}
The simples of $\cOg$ have a finer decomposition into sets called left
and right cells.  We say that $L$ and $M$ are in the same {\bf left
  cell} if their sections $\secs(L)$ and $\secs(M)$ have the same
annihilator where we assume we have translated to a period where
localization holds; this is well-defined
since the sections of $L$ at different choices of $\chi$ where localization
holds are related by a Morita equivalence of the section algebras.
We say that they are in the same {\bf right cell} if $L$ is a composition factor
in $B\otimes M$ for some Harish-Chandra bimodule and {\bf vice versa}.

\begin{conjecture}\label{left-right-cells}
  Two simples are in the same left cell if they have the same weight
  and the same string parametrization for 
  $\sllhat$;  similarly, two simples are in the same right cell if and
  only if they have the same weight, and the same string
  parametrization for $\slehat$.
\end{conjecture}

One important special case of this correspondence is where
$\mu=\nu$; in this case, the simples whose support lie in no smaller
stratum are exactly those whose support contains a free
closed orbit.  Proposition \ref{prop:strata-bijection} shows these
match under Koszul duality with simples supported over the point
strata; these are the same as simples in the category generated by the
highest weight object of weight $\mu^!$.  That is, they match with the
projectives for Hecke loadings, those where the strands are separated by at least $|\kappa|$ units.  
Thus, if we let $P^0$ be the sum of these projectives as before, and
abuse notation to let it denote the corresponding projective in
$\cOg^v$, we have that:
\begin{corollary}\label{cor:same-P}
  The projective covers of simples in $\cOg$ with support that
  contains a free and closed orbit correspond to the modules over
  $T^{\vartheta_{U^!}}_{\la^!}$ for Hecke loadings.  In particular,
  $\red_!^v(P^0)$ is independent of $v$, and has
  $\End(\red_!^v(P^0))\cong T^{\mu^!}_{\la^!}$.
\end{corollary}
From this, we can deduce that:
\begin{theorem}\label{twist-shuff-dual}
  The equivalence $\cOg\cong T^{-\vartheta_U^!}_{\la^!}\umod$
  intertwines twisting functors $\mathscr{T}_i$ with change-of-charge functors $\mathscr{B}^{\vartheta,s_i\vartheta}$.
\end{theorem}
\begin{proof}
Given an exceptional collection $\Delta_i$ indexed by an
ordered set $(\Upsilon, >)$, the {\bf mutation} of this exceptional
collection for a new order $>'$ on $\Upsilon$ is the unique
exceptional collection such that $\Delta_i'$ lies in the triangulated category generated by
$\{\Delta_j\}_{j\geq' i}$ and $\Delta_i'\equiv \Delta_i$ modulo the
triangulated category generated by $\{\Delta_j\}_{j >' i}$.
  The first fact that we need is that both the twisting functor
  $\mathscr{T}_i$ and
  change-of-charge functors send the standard exceptional collection
  in the source category
  to a mutation of the standard exceptional collection. In the first case, this
  is proven algebraically for the Rickard complexes in
  \cite[5.21]{WebRou};
in the second, this follows from
\cite[5.14]{WebRou}.  In both cases, the change of order is that
  induced by the generator  of the affine braid group $\widehat{B}_\ell$.

  Thus, if we consider the composition
  $\mathscr{B}^{\vartheta,s_i\vartheta}\circ \mathscr{T}_i^{-1}$, this
  functor is exact and sends projectives to projectives.  On the other hand, the
  twisting 
  functors send $P^0$ to $P^0$ and induces the identity functor on morphisms by
  Corollary \ref{cor:same-P}. The same is manifestly true for
  change-of-charge functors.  By the faithfulness of the cover
  $\Hom(P^0,-)$, the functor
  $\mathscr{B}^{\vartheta,s_i\vartheta}\circ \mathscr{T}_i^{-1}$ must thus
  be isomorphic to the identity and the result follows.
\end{proof}
\begin{proof}[Proof of Theorem \ref{twist-shuffle}]
  We already know that the Koszul dual equivalence $\dOg\cong
  T^{\vartheta_U}_{\mu}\dgmod$ intertwines {\it shuffling} functors 
  with change-of-charge functors by
  Corollary \ref{affine-shuffling}.  Thus, Theorem \ref{twist-shuff-dual} shows the desired Koszul
  duality of twisting and shuffling, completing the proof of Theorem
  \ref{twist-shuffle}.
\end{proof}

\subsection{Symplectic duality}
\label{sec:symplectic-duality}

This theorem is part of a more general picture, laid out by Braden, Licata,
Proudfoot and the author \cite[\S 9]{BLPWgco}, called {\bf symplectic duality}.  The
underlying idea is that there is a duality operation on symplectic
cones which switches certain geometric data.  

We regard Corollary
\ref{Koszul-dual} as evidence that affine quiver varieties come in
dual pairs $\fN^\la_\mu$ and $\fN^{\mu^!}_{\la^!}$, indexed by rank-level dual weight spaces.
The reader could rightly protest that $\la^!$ depends on the weighting
$\vartheta$. However, different choices of weighting $\vartheta$ result in
$\la^!$ which are conjugate under the action of the Weyl group, and the
cone $\fN^{\mu^!}_{\la^!}$ only depends on the Weyl orbit of $\la^!$ by work of
Maffei \cite{MaffWeyl}.  In fact, for purposes of understanding
duality, it is better to fix the weights $\la^!$ and $\mu$ to be
dominant and think of the varieties $\fM^\la_{w\mu}$
as $w$ ranges over the Weyl group $W$ as the GIT quotients of
$E^\la_\mu$ at the GIT stability conditions $w^{-1}\cdot \det$.  If we also
consider $-w^{-1}\cdot\det$, this gives us a (redundant) list containing a
representative of every GIT chamber.  

The different weightings of the Crawley-Boevey quiver form a similar
chamber structure; in \cite[2.7]{WebRou}, we define
the notion of {\bf Uglovation}, which sends the elements of each chamber to
a standard representative. The
walls that separate them are of the form
\begin{equation*}
\vartheta_i-\vartheta_j-\ck(r_i-r_j+me)=0\quad\text{ for all }\quad
m\in \Z \text{ and }i,j\in \Z/\ell\Z.
\end{equation*}
These walls are unchanged (just reindexed) if we replace $r_i$ by
$s_i$ for any charge with $r_i\equiv s_i\pmod e$.  

We can identify
identity the Lie algebra $\mathfrak{t}_\ell$ of the torus $T_\ell\cong
(\C^*)^\ell$ of $H$ with
the span $\mathfrak{H}_\ell$ of the fundamental weights
$\omega_i$ in the dual Cartan of $\sllhat$ via the map 
\[u_{\Bs}(\vartheta) = \ck e\omega_0+\sum_{i=1}^{\ell}(\vartheta_i-\ck s_i) (\omega_i-\omega_{i+1})\]
these walls are defined by $(u(\vartheta),\al)=0$ for all positive
roots $\al$ of $\sllhat$.  Note that if $w\cdot \Bs=\Bs'$ for $w\in
\widehat{W}_\ell$, then $w\cdot u_{\Bs}=u_{\Bs}\cdot w=u_{\Bs'}$.
Note that
\begin{equation}\label{vartheta-dominant}
u_{\Bs}(\vartheta_U)=e(\omega_1-\omega_2)+2e(\omega_2-\omega_3)+\cdots +e(\ell-1)(\omega_{\ell-1}-\omega_0)+e\ell\omega_0=e\sum\omega_i.
\end{equation}

Thus, under rank-level duality, the possible spaces of choices for
GIT stability conditions and weightings switch; furthermore, these
bijections preserve the appropriate chamber structures, sending the
chamber containing our
chosen weighting to the dominant Weyl chamber, by equation \eqref{vartheta-dominant}.  The 
reader might object that not all positive roots genuinely contribute GIT
walls; this is compensated for by the fact that the corresponding Uglov
weightings have the same relative core $\fM^+$ and associated category
$\cO$.

\begin{proposition}
The hyperplane $(-,\al)=0$ is a GIT wall for the reduction
  that presents $\fM^\la_\mu$ if and only if one of the
  multipartitions indexing a fixed point in $\fM^\la_\mu$ has a
  removable ribbon of residue $\al$.  
\end{proposition}
\begin{proof}
Using the Weyl group action, we can rephrase the characterization of
the GIT walls.  Since the GIT walls for $\mu$ are sent to GIT walls
for $w\cdot \mu$ by the action of $w\in \widehat{W}_e$, we need only
prove this for $\al=\al_i$ a simple root.

 Our claimed characterization is equivalent to the
statement that for any $\mu$, the locus $(-,\al_i)$ is a GIT wall if
and only if one of the partitions indexing a fixed point has a
removable box of residue $i$.  The ``if'' direction is clear; the
$\bT$-fixed point corresponding to any multipartition with such a
removable box provides an example of a strictly semi-stable
representation on this wall.  

Now consider the ``only if.''  The points of $\fM^\la_\mu$ that become
semi-stable on the wall are a closed $\bS\times \bT$-invariant subset which is by
assumption non-empty.  Thus, it must contain at least one point of the
core, which we can assume is $\bT$ invariant.  The corresponding
multi-partition has a removable box of the right residue.
\end{proof}
\begin{proposition}
 The walls in the Hamiltonian torus $\ft_\ell$ are
  given by those hyperplanes which correspond to GIT walls of
  $\fM^{\mu^!}_{\la^!}$ under duality.
\end{proposition}
\begin{proof}
Consider the GIT wall attached to a root $\al_{k'}+\al_{k'+1}+\cdots
\al_{k-1}+\al_{k}+m\delta$.  In terms of abaci, the appearance of a
removable ribbon of the right residue says that there must be some abacus
of the right residue such that a bead of residue $k$ can be moved down
a runner $k-k'+1+me$ slots into an empty spot.  The rank-level dual
condition is that a bead of some fixed residue $r$ in the $k$th runner
can to moved to a slot of residue $r$ in the $k'$th runner in the row
$\ell m$ slots down if $k>k'$, or $\ell(m+1)$ slots down if $k\leq
k'$.  This bead and slot it moves into are at the end of the leg and
arm of some box in the $k$th partition,  and the line in the tangent
space corresponding to this box in the formula of \cite[5.10]{TS} has
trivial $\bT$ action.  Thus the $\bT$-fixed locus is positive
dimensional.  

The same formula shows that if there is no such bead,
the $\bT$ action on the tangent space at each fixed point has no
invariants.  Thus, all $\bT$-fixed points remain isolated, and there is no wall.
\end{proof}

Since the kernel of the $\ft_\ell$ action on $\fM^\la_\mu$ is the
intersection of all the walls with positive dimensional fixed locus,
and the kernel of Kirwan map is the intersection of all the GIT walls,
this further implies that:
\begin{corollary}
  The map $u_{\Bs}$ induces an injective map $\ft^\la_\mu\to
  H^2(\fM^{\mu^!}_{\la^!})$ from the quotient $\ft^\la_\mu$
  of
  $\ft_\ell$ that acts faithfully on $\fM^\la_\mu$.  The image of this
  map coincides with that of the
  Kirwan map.
\end{corollary}

This allows us to restate Corollary \ref{Koszul-dual} and Theorem
\ref{twist-shuffle} in a more symmetric way.  Let $\mathsf{O}^\xi_\eta$ be the
category $\cOg$ attached to $\xi\in\mathfrak{t}_w$ and the GIT 
stability condition $\eta$ (which can switch places to be the same
data for the rank-level dual). 

\begin{theorem}
  There is a Koszul duality equivalence $D^b(\mathsf{O}^\xi_\eta)\cong
  D^b(\mathsf{O}^{-\eta}_\xi)$, which intertwines the shuffling functor
  $\Psi^{\xi,\xi'}\colon D^b(\mathsf{O}^\xi_\eta)\to D^b(\mathsf{O}^{\xi'}_\eta)$
  with the twisting functor $\Phi^{\xi,\xi'}\colon D^b(\mathsf{O}_\xi^{-\eta})\to D^b(\mathsf{O}_{\xi'}^{-\eta})$.
\end{theorem}

For ease of reference, we collect together the pieces of data we
require for a symplectic duality.  Here we employ the notation of \cite{BLPWgco}.

\begin{itemize}
\item The set of simples in $\cO^{\la;\theta_U}_\mu$ are indexed by
  $\ell$-multipartitions whose total content is fixed by $\mu$; these
  can be interpreted as $\ell$-strand abaci.  As
  proven in \cite[6.4]{WebRou}, the
  Koszul duality bijection is given by  cutting the abacus into
  $\ell\times e$ rectangles and flipping, as shown in \eqref{eq:flip}.
\item The poset $\mathscr{S}^{\operatorname{sp}}_{\fM^\la_\mu}$ of special strata is
  in bijection with weights $\nu$ and integers $n$ such that $\la\geq
  \nu\geq \nu-n\delta\geq \mu$, and the desired bijection to
  $\mathscr{S}^{\operatorname{sp}}_{\fM^{\mu^!}_{\la^!}}$ sends
  $(\nu,n)\mapsto (\nu^!+n\delta, n)$.
\item We have described maps $u_{\Bt}\colon \mathfrak{t}^{\mu^!}_{\la^!}\to
H^2(\fM^\la_\mu;\C)$ and $-u_{\Bs}\colon \mathfrak{t}_{\mu}^{\la}\to
H^2(\fM_{\la^!}^{\mu^!};\C)$.  These are isomorphisms if Kirwan
surjectivity holds for affine type A quiver varieties, which we will
assume from now on.

\item The Namikawa Weyl group $W$ is the stabilizer of $\mu$ under the
  action of $\widehat{W}_e$, and the Weyl group $\mathbb{W}$ of the
  group of Hamiltonian automorphisms commuting with $\bS$ is the
  stabilizer of $\mathbf{s}$ under $\widehat{W}_\ell$.
\end{itemize}
\begin{theorem}\label{symplectic-duality}
These bijections together with Koszul duality of Corollary
  \ref{Koszul-dual} define a {\bf symplectic duality} in the sense of
  \cite[\S 10]{BLPWgco} between $\fM^\la_\mu$ and
  $\fM^{\mu^!}_{\la^!}$.
\end{theorem}
\begin{proof}
We must check that:
\begin{itemize}
\item {\it The bijections between simples and special leaves are compatible
  with the map sending a simple to the associated variety of its
  annihilator:}  This follows from  Proposition
  \ref{prop:strata-bijection}.
\item {\it The isomorphism $H^2(\fM^\la_\mu;\C)\cong
  \mathfrak{t}^{\mu^!}_{\la^!}$ is compatible with the action of the
  stabilizer of $\mu$ in 
  $\widehat{W}_e$ on both sides:}  It  is clear that  the isomorphism $\mathfrak{t}_e\cong
  \mathfrak{H}_e$ intertwines  the action of $\widehat{W}_e$ on
  $\mathfrak{t}_e$ by swapping coordinates and translating and the
  natural action on $\mathfrak{H}_e$, and the quotients inherit this
  action.
\item {\it The ample cone in $H^2(\fM^\la_\mu;\C)$ is sent to the chamber
  of the chamber of weightings with the same category $\cO$ as $\vartheta_{U^!}$, and
  similarly for the ample cone of $ H^2(\fM_{\la^!}^{\mu^!};\C)$ and $-\vartheta_{U}$:}
For $u_{\Bt}$, this follows immediately from
\eqref{vartheta-dominant}.  The same
argument works for the dual, since we have taken the negative map $-u_{\Bs}$.
\item {\it The Koszul duality switches twisting and shuffling
    functors:} This follows from Theorem \ref{twist-shuff-dual}.\qedhere
\end{itemize}
\end{proof}

\subsection{The finite type A case}
\label{sec:finite-type-case}

One important special case of Theorem \ref{symplectic-duality} is the
case where the coordinate of the dimension vector $v_0=0$.  This
results in a quiver variety for the finite dimensional Lie algebra
$\mathfrak{sl}_e$.  

The finite type A case is simpler, first because the chamber of the
action $\pm \theta_U$ only depends on the order of the weights on the
new vertices in the Crawley-Boevey quiver.  Thus, we need only
consider the case of $\Bs$ with $1\leq s_i\leq e-1$ and $\ck>0$; note
that if we switch the sign of $\ck$, this has the same effect as
switching $\Bs$ to $\Bs_\circ=(s_\ell,\dots, s_1)$.   A partition
with charge $s_i$ which has no box of charge $0=e$ is one that fits
inside an $s_i\times (e-s_i)$ box.  In terms of abaci this means that
there is one interesting rectangle, corresponding to charges $1$
through $e-1$, with all positions to the left filled, and all to the
right unfilled.
The associated matrices $U$ will only have
entries $0$ or $1$; this matrix actually fixes the multi-partition,
since it has $1$'s in the positions filled with beads and $0$'s in the
empty positions.  

Thus, pictorially, we need only draw this one rectangle.  For example,
if we have $e=4,\ell=3$, and $\Bs=(3,1,2)$, then
\begin{equation*}\{(2),(1,1),(1,2)\}\,\leftrightarrow\,
 \begin{tikzpicture}[very thick,scale=.5,baseline]
      \node[circle, draw, inner sep=3pt] at (0,0){}; \node[circle,
      draw, inner sep=3pt,fill=black] at (0,1) {}; \node[circle, draw, inner
      sep=3pt] at (1,0) {}; \node[circle, draw, inner sep=3pt] at
      (1,1) {}; \node[circle, draw, inner sep=3pt,fill=black] at (2,0) {};
      \node[circle, draw, inner sep=3pt,fill=black] at (2,1) {}; 
      \node[circle, draw, inner sep=3pt] at (0,-1){}; \node[circle, draw, inner
      sep=3pt,fill=black] at (1,-1) {}; \node[circle, draw, inner
      sep=3pt] at (2,-1) {};
\node[circle, draw, inner sep=3pt] at (3,0) {};
      \node[circle, draw, inner sep=3pt,fill=black] at (3,1) {}; 
\node[circle, draw, inner
      sep=3pt,fill=black] at (3,-1) {};
    \end{tikzpicture}
  \end{equation*}
The corresponding dimension vector for the $A_3$ quiver is $(2,3,2)$.

Fix a matrix $U$ with entries only $0$'s and $1$'s.  As usual, we have
the row and column sums $\Bs$ and $\Bt$, the associated
$\mathfrak{sl}_e$-weights $\lambda,\mu$ and the
$\mathfrak{sl}_\ell$-weights constructed from the
transposed matrix.   Attached to this data, we have 3 different
varieties, each of which is equipped with a preferred chamber in the
space of $\C^*$-action.
\begin{itemize}
\item the type $A_{e-1}$-quiver variety $\fM^\la_\mu$ introduced
  earlier, with the chamber determined by the order on $\Bs$.
\item the S3-variety $\fX^{\Bs}_{\Bt}$.  Here we use the notation of \cite[\S
    10.2.2]{BLPWgco}; this is subvariety of the cotangent bundle to
    the space of flags of type $\Bt$, intersected with a slice to the
    orbit through a nilpotent $e_{\Bs}$ with Jordan type $\Bs$.  If we choose a
    basis so that $e_{\Bs}$ is in Jordan normal form, with the order
    given by $\Bs$, the preferred
    chamber is the cocharacters which are diagonal with weakly
    decreasing weights in this basis.
\item the resolved slice $\operatorname{Gr}^{\bar \bmu^!}_{\la^!}$ in
  the affine Grassmannian.  Here, we use the notation of \cite{KWWY}:
  we think of $\mu^!,\lambda^!$ as coweights for the
  Langlands dual group $PGL_\ell$ and let $\bmu^!$ be the sequence of
  fundamental coweights $(\om_{t_1}^\vee,\dots,\om_{t_e}^\vee)$.
  This is the preimage in the convolution variety
  $\operatorname{Gr}^{\om_{t_1}^\vee} * \cdots *
  \operatorname{Gr}^{\om_{t_1}^e}$ of a normal slice to
  $\operatorname{Gr}^{\la^!}$ inside $\operatorname{Gr}^{\mu^!}$.  The
  preferred chamber is that containing $\rho^\vee$.
\end{itemize}
In each case, there is a ``reversing map'' on the combinatorial data;
use the subscript ${}_\circ$ to denote this in each case:
$\Bs_\circ,\Bt_\circ$ are just the reversals of the entries.  We let
$\mu_\circ=w_0\cdot\mu$ and use $\la_\circ $ for the same weight, but
to denote that we now prefer the inverse of the preferred chamber of
cocharacters.  Similarly, $\bmu^!_\circ$ is the reversal of this
sequence, and $\la^!_\circ=w_0\cdot \la^!$.

Note that all of these objects are of type $A$, but they all involve
different ranks.  Work of Maffei \cite{Maf} and Mirkovi\'c-Vybornov
\cite{MV08} establishes that the varieties $\fM^\la_\mu, \fX^{\Bs}_{\Bt}$ and
$\operatorname{Gr}^{\bar \bmu^!}_{\la^!}$ are isomorphism, and
\cite[4.6.4]{Losq} shows that in fact they are symplectomorphic.

\excise{\begin{theorem}[\mbox{}]
The varieties algebraically symplectomorphic to each other.
\end{theorem}
\begin{proof}
  We need only establish that these varieties are symplectomorphic.
  It's enough to establish that the forms on the affinizations agree,
  since the forms on the resolutions are just pullbacks.
  In each case, the variety is question can realized as a slice
  between two strata in an affinization isomorphic to a type A
  nilcone.  This is by definition for the S3-varieties; for the
  dimension vector $(1,2,\dots,s-1)$ with $\la=s\om_{s-1}$ for quiver
  varieties; and the slice to the identity coset inside $
  \operatorname{Gr}^{\om_1}* \cdots *\operatorname{Gr}^{\om_1}$.  In
  all 3 cases, there is a Hamiltonian action of $PGL_s$, and the
  moment maps provide an isomorphism to the nilcone with its standard
  symplectic structure.
\end{proof}}
Of course, this means that Theorem \ref{symplectic-duality} tells us how to
find the symplectic dual of each of these varieties.  However, some
care about signs is needed since our formulation of
this theorem involves negating $\theta$.

\begin{theorem}\label{finite-dual}
  The symplectic dual of $\fM^\la_\mu\cong \fX^{\Bs}_{\Bt}\cong
  \operatorname{Gr}^{\bar \bmu^!}_{\la^!}$ is the variety
  $\fM^{\mu^!}_{\la^!}\cong \fX^{\Bt}_{\Bs}\cong
  \operatorname{Gr}^{\bar {\boldsymbol{\la}}}_{\mu}$ with the opposite of the preferred
  $\C^*$-action.   This is equivalent to instead considering
  \[\fM^{\mu_\circ^!}_{\la^!}\cong \fM^{\mu^!}_{\la^!_\circ}\cong
  \fX^{\Bt_\circ}_{\Bs}\cong \fX^{\Bt}_{\Bs_\circ}\cong
  \operatorname{Gr}^{\bar {\boldsymbol{\la}}_\circ}_{\mu}\cong \operatorname{Gr}^{\bar {\boldsymbol{\la}}}_{\mu_\circ} \] with the preferred $\C^*$-action.
\end{theorem}

 \bibliography{../gen}
\bibliographystyle{amsalpha}
\end{document}